\documentclass{article}%

\usepackage{array}
\usepackage{amsmath}%
\usepackage{amsfonts}%
\usepackage{amssymb}%
\usepackage{graphicx}
\usepackage{fullpage}
\usepackage{palatino}
\usepackage{tikz}
\usepackage{subcaption}
\usepackage[margin=2cm]{caption}
\usepackage{dsfont}
\usepackage[bookmarksnumbered=true]{hyperref}

\usepackage{ntheorem}

\newtheorem{theorem}{Theorem}
\numberwithin{theorem}{section}

\newtheorem{condition}[theorem]{Condition}

\newtheorem{corollary}[theorem]{Corollary}

\newtheorem{lemma}[theorem]{Lemma}

\newtheorem{proposition}[theorem]{Proposition}
\newtheorem{remark}[theorem]{Remark}

\newenvironment{proof}[1][Proof]{\textbf{#1.} }{\ \rule{0.5em}{0.5em}}

\usepackage{thmtools} 
\usepackage{thm-restate}

\usepackage{amsmath}%
\usepackage{amsfonts}%
\usepackage{amssymb}%
\usepackage{graphicx}
\usepackage{booktabs}
\usepackage{bbm}
\usepackage[english]{babel}
\usepackage{tikz}
\usepackage{verbatim}
\usetikzlibrary{plotmarks}
\tikzset{
    scale plot marks/.is choice,
    scale plot marks/false/.code={
        \def\pgfuseplotmark##1{\pgftransformresetnontranslations\csname pgf@plot@mark@##1\endcsname}
    },
    scale plot marks/true/.style={},
    scale plot marks/.default=true
}
\usepackage[ruled, linesnumbered]{algorithm2e}
\SetAlCapSkip{0.5em}

\usepackage{dsfont}
\usepackage[bookmarksnumbered=true]{hyperref}

\newcommand{\sss}{\scriptscriptstyle}
\newcommand{\Prob}{\mathbb{P}}
\newcommand{\E}{\mathbb{E}}

\newcommand{\CMd}{CM_n(\mathbf{d})}
\newcommand{\Cmd}{CM_N(\mathbf{d}')}
\newcommand{\dto}{\overset{d}{\to}}
\newcommand{\pto}{\overset{\Prob}{\to}}

\newcommand{\Cmax}{\mathcal{C}_{\rm max}}
\newcommand{\CMD}{CM_n(\mathbf{d},q)}
\newcommand{\op}{o_{\Prob}}
\newcommand{\Op}{O_{\Prob}}
\newcommand{\Acal}{\mathcal{A}}
\newcommand{\Ncal}{\mathcal{N}}
\newcommand{\Dcal}{\mathcal{D}}

\newcommand{\Vcal}{\mathcal{V}}

\newcommand{\indi}{\mathbbm{1}}

\title{Failure behavior in a connected configuration model\\ under a critical loading mechanism}
\author{Fiona~Sloothaak \and Lorenzo Federico}

\date{\vspace{-1ex}}

\begin{document}
\maketitle

\begin{abstract}
We study a cascading edge failure mechanism on a connected random graph with a prescribed degree sequence, sampled using the configuration model. This mechanism prescribes that every edge failure puts an additional strain on the remaining network, possibly triggering more failures. We show that under a critical loading mechanism that depends on the global structure of the network, the number of edge failure exhibits scale-free behavior (up to a certain threshold). Our result is a consequence of the failure mechanism and the graph topology. More specifically, the critical loading mechanism leads to scale-free failure sizes for any network where no disconnections take place. The disintegration of the configuration model ensures that the dominant contribution to the failure size comes from edge failures in the giant component, for which we show that the scale-free property prevails. We prove this rigorously for sublinear thresholds, and we explain intuitively why the analysis follows through for linear thresholds. Moreover, our result holds for other graph structures as well, which we validate with simulation experiments.
\end{abstract}

\section{Introduction}
Cascading failures is a term used to describe the phenomenon where an initial disturbance can trigger subsequent failures in a system. A typical way to describe cascading failure models is through a network where every node/edge failure creates an additional strain on the surviving network, possibly leading to knock-on failure effects. It appears in many different application areas, such as communication networks~\cite{Motter2002,motter2004,Korkali2017}, road systems~\cite{Chakrabarti2006,Zheng2008}, earthquakes~\cite{BakTangWiesenfeld1987,BakTangWiesenfeld1988,BakTangWiesenfeld1989}, material science~\cite{Pradhan2010}, epidemiology~\cite{Watts2002simplemodelof, morone2015}, power transmission systems~\cite{Carreras2004Chaos,Qi2013,Bienstock2016,SunHouSunQi2018}, and more. 

A macroscopic characteristic that is observed in different natural and engineering systems, is the emergence of scale-free behavior~\cite{Barabasi1999, suki1994, barabasi2005, Clauset2009, Simon1955}. Although this heavy-tailed property sometimes relates to the nodal degree distribution, there are many physical networks that do not share this feature (e.g.~power grids~\cite{Wang2010}). An intriguing question arises why failure sizes still display scale-free behavior for these types of networks. 

In this paper, we introduce a cascading mechanism on a complex network whose nodal distribution does not (necessarily) exhibit scale-free behavior. The cascade is initialized by a small additional loading of the network, and failures occur on edges whenever its load capacity is exceeded. An intrinsic feature in this work is that the propagation of failures occurs non-locally and depends on the global network structure, which continually changes as the failure process advances. This is a novel feature with respect to the existing literature: in order to obtain an analytically tractable model, the cascading mechanisms is often described through local relations, or assumes an independence of the global network structure throughout the cascade. Well-known examples that satisfy at least one of these properties are epidemiology models (where an initially infected node may taint each of its neighbors with a fixed probability), or percolation models (where each edge/node in the network fails with a fixed probability). In this work, we introduce a critical load function that leads to a power-law distributed tail for the failure size. 


\subsection{Model description}
Let $G=(V,E)$ denote a graph, where $V$ denotes the vertex set with $|V|=n$, and $E$ denotes the edge set with $|E|=m$. Typically, we consider graphs that can be scaled in the number of vertices/edges. Suppose that each edge in the network is subjected to a load demand that is initially exceeded by its edge capacity. The difference between the initial load and the edge capacity is called the \textit{surplus capacity} of the edge, and we assume the surplus capacities to be independently and uniformly distributed on $[0,1]$ at the various edges. The failure process is triggered by an initial disturbance: all edges are additionally loaded with $\theta/m$ for some constant $\theta>0$. If the total load increase surpasses the surplus capacity of one or more edges, these edges fail and in turn, cause an additional load increase on all surviving edges that are in the same component upon failure. We call these load increments the \textit{load surges}. We make two assumptions: edge failures occur subsequently, and all edges in the same component upon failure experience the same load surge. We continue with the failure of edges that have insufficient surplus capacity to handle the load surges till there are no more. We are interested in the number of edge failures, also referred to as the failure size. 

We are interested in the setting where the failure size exhibits scale-free behavior. To this purpose, we define the \textit{critical load surge function} $l_j^m(i)$ at edge $j$ after experiencing~$i$ load surges to be
\begin{align}
\left\{\begin{array}{ll}
l_j^m(1) = \theta/m,\\
l_j^m(i+1) = l_j^m(i)+\frac{1-l_j^m(i)}{|E_j^m (i-1)|}, \hspace{1cm} i=1,\ldots,m-1,
\end{array}\right.
\label{eq:LoadSurgeRecursion}
\end{align}
where $|E_j^m(i)|$ is the number of edges in the component that contains edge $j$ after perceiving~$i$ edge failures in that component. 

We observe that as long as two edges remain in the same component during the cascade, this recursive relation implies that the load surges are the same at both edges. Moreover, as long as all edges remain to be in a single component, it holds that $|E_j^m(i)|=m-i$ at every surviving edge, and hence~\eqref{eq:LoadSurgeRecursion} is solved by
\begin{align}
l_j^m(i) = \frac{\theta}{m}+(i-1) \cdot \frac{1-\theta/m}{m} = \frac{\theta+i-1}{m}(1+o(m)), \hspace{1cm} i=1,\ldots,m .
\label{eq:LoadSurgeNoDisconnection}
\end{align}
In particular, applying~\eqref{eq:LoadSurgeRecursion} to a star topology with $n+1$ nodes and $m=n$ edges yields load surge function~\eqref{eq:LoadSurgeNoDisconnection} for all surviving edges.

We point out that the load surges defined in~\eqref{eq:LoadSurgeRecursion} are typically non-deterministic, and are only well-defined as long as edge $j$ has not yet failed throughout the cascade. Edge failures may cause the network to disintegrate in multiple components, which affects $|E_j^m(i)|$ in a probabilistic way that depends on the structure of the graph. In contrast to processes in epidemiology or percolation models, the propagation of failures thus occurs non-locally and depends on the global structure of the graph throughout the cascade. 

To provide an intuitive understanding why~\eqref{eq:LoadSurgeRecursion} gives rise to scale-free behavior for the failure size, consider the following. For a scale-free failure size to appear, the cascade propagation should occur in some form of \textit{criticality}. This happens when the load surges are close to the expected values of the ordered surplus capacities. More specifically, since the surplus capacities are independently and uniformly distributed on~$[0,1]$, the mean of the $i$-th smallest surplus capacity is $i/(m+1)$. If no disconnections occurred during the cascade, this would imply that the additional load surges at every edge should be close to $1/(m+1)$ after every edge failure. This explains intuitively why~\eqref{eq:LoadSurgeNoDisconnection} leads to scale-free behavior for the star-topology, as is shown rigorously in~\cite{Sloothaak2016}. Yet, whenever the network disintegrates in different components, the consecutive load surges need to be of a size such that the load surge is close the smallest expected surplus capacity of the remaining edges in order for the cascade to remain in the window of criticality. More specifically, suppose that the first disconnection occurs after $k$ edge failures and splits the graph in two components of edge size $l$ and $m-k-l$, respectively. Due to the properties of uniformly distributed random variables, we note that the expectation of the smallest surplus capacity in the first component equals $k/(m+1)+(1-k/(m+1))/l$. In other words, in order to stay in a critical failure regime, the additional load surges should be close to $(1-k/(m+1))/l$ after every edge failure until the cascade stops in that component, or another disconnection occurs. This process can be iterated, and gives rise to load surge function~\eqref{eq:LoadSurgeRecursion}. 

In this paper, our main focus is to apply failure mechanism~\eqref{eq:LoadSurgeRecursion} to the (connected) \textit{configuration model} $CM_n(\mathbf{d})$ on $n$ vertices with a prescribed degree sequence $\mathbf{d}=(d_1,...,d_n)$. The configuration model is constructed by assigning $d_v$ half-edges to each vertex $v \in [n]:=\{1,...,n\}$, after which the half-edges are paired randomly: first we pick two half-edges at random and create an edge out of them, then we pick two half-edges at random from the set of remaining half-edges and pair them into an edge, and we continue this process until all half-edges have been paired. For consistency, we assume that the total degree $\sum_{i=1}^n d_i$ is even. The construction can give rise to self-loops and multiple edges between vertices, but these events are relatively rare when $n$ is large~\cite{Hofs17,Jan09,Jan14}. We point out that the number of edges is thus a function of $n$ and $\mathbf{d}$, i.e. $m:= m_n(\mathbf{d}) := \sum_{i=1}^n d_i /2$, where we suppress the dependency on $n$ and $\mathbf{d}$ for the sake of exposition. 

Define $n_i$ as the number of vertices of degree $i$, and let $D_n$ denote the degree of a vertex chosen uniformly at random from $[n]$. We assume the following condition on the degree sequence.
\begin{condition}[Regularity conditions]\label{con:RegularConfiguration}
	Unless stated otherwise, we assume that $CM_n(\textbf{d})$ satisfies the following conditions:
	\begin{itemize}
		\item There exists a limiting degree variable $D$ such that $D_n$ converges in distribution to $D$ as $n \rightarrow \infty$;
		\item $n_0, n_1=0$;
		\item $p_2:= \lim_{n \rightarrow \infty} n_2/n \in (0,1)$;
		\item $n_j=0$ for all $j \geq n^{1/4 -\varepsilon}$ for some $\varepsilon >0$;
		\item $d:= \lim_{n \rightarrow \infty} \E[D_n] < \infty$.
	\end{itemize}
\end{condition}

\noindent
Under these conditions, we can write $p_i := \lim_{n \rightarrow \infty} n_i/n$ as the limiting fraction of degree $i$ vertices in the network. Moreover, under these conditions it is known that there is a positive probability for $CM_n(\mathbf{d})$ to be connected~\cite{Federico2016,Luc92}. Our starting point will be such a configuration model conditioned to be connected, denoted by $\overline{CM}_n(\mathbf{d})$, on which we apply the edge failure mechanism~\eqref{eq:LoadSurgeRecursion}. We are interested in quantifying the resilience of the connected configuration model under~\eqref{eq:LoadSurgeRecursion}, measured by the number of edge failures~$A_{n,\mathbf{d}}$.

\begin{remark}
In Condition~\ref{con:RegularConfiguration}, we assume $n_0=0$ to ensure that the resulting graph has a positive probability to be connected. Moreover, as explained in~\cite{Federico2016}, it suffices to pose the condition that $n_1=o(\sqrt{n})$ for a positive probability of having a connected configuration model. We point that our results do not change when allowing that $n_1 = o(\sqrt{n})$, but for the sake of exposition, we removed the details in this paper.
\end{remark}

\subsection{Notation}
In Appendix~\ref{app:Notation}, we provide an overview of the quantities that are commonly used throughout the paper. Unless  stated otherwise, all limits are taken as $n\rightarrow \infty$, or equivalently by Condition~\ref{con:RegularConfiguration}, as $m \rightarrow \infty$. A sequence of events $(A_n)_{n \in \mathbb{N}}$ happens with high probability (w.h.p.) if $\Prob(A_n)\rightarrow 1$. For random variables $(A_n)_{n \in \mathbb{N}}$, we write $X_n \overset{d}{\rightarrow} X$ and $X_n \overset{\Prob}{\rightarrow} X$ to denote convergence in distribution and in probability, respectively. For real-valued sequences $(a_n)_{n \in \mathbb{N}}$ and $(b_n)_{n \in \mathbb{N}}$, we write $a_n=o(b_n)$ and $a_n=O(b_n)$ if $\lim_{n \rightarrow \infty} a_n/b_n$ tends to zero or is bounded, respectively. Similarly, we write $a_n=\omega(b_n)$ and $a_n=\Omega(b_n)$ if $\lim_{n \rightarrow \infty} b_n/a_n$ tends to zero or is bounded, respectively. We write $a_n=\Theta(b_n)$ if both $a_n=O(b_n)$ and $a_n=\Omega(b_n)$ hold. We adopt an analogue notation for random variables, e.g.~for sequences of random variables $(X_n)_{n \in \mathbb{N}}$ and $(Y_n)_{n \in \mathbb{N}}$, we denote $X_n=o_{\Prob}(Y_n)$ if $X_n/Y_n \overset{\Prob}{\rightarrow} 0$. For convenience of notation, we denote $a_n \ll b_n$ if $a_n = o(b_n)$. Finally, $\textrm{Poi}(\lambda)$ always denotes a Poisson distributed random variable with mean $\lambda$, $\textrm{Exp}(\lambda)$ denotes an exponentially distributed random variable with parameter $\lambda$, and $\textrm{Bin}(n,p)$ denotes a binomial distributed random variable with parameters~$n$ and $p$.

\subsection{Main result}
We are interested in the probability that the failure size exceeds a threshold $k$. In this paper, we mainly focus on thresholds satisfying $1 \ll k \ll m^{1-\delta}$ for some $\delta \in (0,1)$, which we refer to as the sublinear case. Our main result shows that the failure size has a power-law distribution.

\begin{theorem}
	Consider the connected configuration model $\overline{CM}_n(\mathbf{d})$, and suppose $k:= k_m$ such that $1 \ll k \ll m^{1-\delta}$ for some $\delta \in (0,1)$. Then,
	\begin{align}
	\Prob\left(A_{n,\mathbf{d}} \geq k\right) \sim \frac{2\theta}{\sqrt{2\pi}} k^{-1/2}.
	\label{eq:MainResultEdges}
	\end{align}
	\label{thm:MainResultSublinear}
\end{theorem}

To see why~\eqref{eq:MainResultEdges} holds, we need to understand the typical behavior of the failure process as the cascade continues. A first result we show is that it is likely that the number of edge failures that need to occur for the network to become disconnected is of order $\Theta(\sqrt{m})$. This suggests that as long as the threshold satisfies $k=o(\sqrt{m})$, the tail is the same as in the case of a star topology. The latter case is known to satisfy~\eqref{eq:MainResultEdges}, as is shown in~\cite{Sloothaak2016}.

As long as the cascade continues, we show that it typically disconnects small components. This suggests that up to a certain point, the cascading failure mechanism creates a network with a single large component that contains almost all vertices and edges, referred to as the \textit{giant component}, and some small disconnected components with few edges. It turns out that the total number of edges that are outside the giant component is sufficiently small, and hence the dominant contribution to the failure size comes from the number of edges that are contained in the giant component upon failure. Moreover, due to small sizes of the components outside the giant component, the load surge function for the edges in the giant as prescribed by~\eqref{eq:LoadSurgeRecursion} is relatively close to~\eqref{eq:LoadSurgeNoDisconnection}. We show that as long as threshold $k \ll m^{1-\delta}$ for some $\delta \in (0,1)$, the perturbations are sufficiently small such that the failure size behavior in the giant component satisfies~\eqref{eq:MainResultEdges}.

We heuristically explain in Section~\ref{sec:UniversalityPrinciple} that the power-law tail property prevails beyond the sublinear case up to a certain critical point. However, the prefactor is affected in this case, i.e. the failure size tail is different from the star topology. Moreover, this notion holds for a broader set of graphs, and we validate this claim by extensive simulation experiments in Section~\ref{sec:UniversalityPrinciple}.

We would like to remark that the disintegration of the network by the cascading failure mechanism is closely related to percolation, a process where each edge is failed/removed with a corresponding removal probability. In fact, percolation results will be crucial in our analysis to prove Theorem~\ref{thm:MainResultSublinear}, combined with first-passage theory of random walks over moving boundaries. Before we lay out a more detailed road map to prove Theorem~\ref{thm:MainResultSublinear}, we provide an outline of the remainder of this paper.

\subsection{Outline}
The proof of Theorem~\ref{thm:MainResultSublinear} requires many different steps, and therefore we provide a road map of the proof in Section~\ref{sec:RoadMapProof}. We explain that we need to derive novel percolation results for the sublinear case, which we show in Section~\ref{sec:Disintegration}. We rigorously identify the impact of the disintegration of the network on the cascading failure process in Section~\ref{sec:CascadingFailureProcess}, where we use first-passage theory for random walks over moving boundaries to conclude our main result. Finally, we study the failure size behavior beyond the sublinear case in Section~\ref{sec:UniversalityPrinciple}, as well as the failure size behavior for other graphs.

\section{Proof strategy for Theorem~\ref{thm:MainResultSublinear}}\label{sec:RoadMapProof}
Our proof of Theorem~\ref{thm:MainResultSublinear} requires several steps. In this section, we provide a high-level road map of the proof.

\subsection{Relation of failure process and sequential removal process}
There are two elements of randomness involved in the cascading failure process: the sampling of the surplus capacities of the edges, and the way the network disintegrates as edge failures occur. The second aspect determines the values of the load surges, and only when the surplus capacity of an edge is insufficient to handle the load surge, the cascading failure process continues. Recall that as long as edges remain in the same component, they experience the same load surge. Since the surplus capacities are i.i.d., it follows that every edge in the same component is equally likely to be the next edge to fail as long as the failure process continues in that component. This is a crucial observation, as this provides a relation with the sequential edge-removal process. That is, suppose that we sequentially remove edges uniformly at random from a graph. Given that a new edge removal occurs in a certain component, each edge in that component is equally likely to be removed next, just as in the cascading failure process. Consequently, this observation gives rise to a coupling of the disintegration of the network through the cascading failure process to one that is caused by sequentially removing edges uniformly at random.

\begin{figure}[htb]
	\centering
	\includegraphics[width=0.4\textwidth]{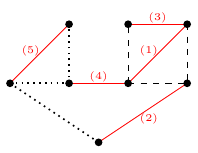}
	\caption{Sample of sequentially removing five edges uniformly at random from a connected graph. We remove the red edges subsequently in successive order, leaving two disconnected component. The first component is connected by the dotted lines, whereas the second component is connected by the dashed lines.}
	\label{fig:SequentialRemovalLoadSurgeRelation}
\end{figure}

More specifically, suppose that sequentially removing edges uniformly at random yields the permutation $\{e_{(1)},...,e_{(m)}\}$. For the cascading failure process, sample $m$ uniformly distributed random variables and order them so that $U_{(1)}^m \leq ... \leq U_{(m)}^m$ and assign to edge~$e_{(j)}$ surplus capacity~$U_{(j)}^m$. In particular, this implies that if the cascading failure process would always continue until all edges have failed, then the edges would fail in the same order as the sequential edge-removal process. Moreover, it is possible to exactly compute  the load surge function from the order $\{e_{(1)},...,e_{(m)}\}$ over the edge set. We illustrate this claim by the example in Figure~\ref{fig:SequentialRemovalLoadSurgeRelation}. In this example, we observe that $m=11$, and we consider the sequential removal process up to five edge removals. We see that the first three edge failures do not cause the graph to disconnect. It holds for every dotted edge~$e_j \in [m]$,
\begin{align*}
|E_j^m(i)|  = \left\{\begin{array}{ll}
11-i & \textrm{if } i\in \{0,1,2,3\} \\
4 & \textrm{if } i=4,\\
3 & \textrm{if } i=5,
\end{array}\right.
\end{align*}
and for every dashed edge $e_j \in [m]$,
\begin{align*}
|E_j^m(i)|  = \left\{\begin{array}{ll}
11-i & \textrm{if } i\in \{0,1,2,3\} \\
3 & \textrm{if } i=4.
\end{array}\right.
\end{align*}
Recursion~\eqref{eq:LoadSurgeRecursion} yields the corresponding load surge values. 

This example illustrates that using our coupling, the sequential removal process gives rise to the load surge values as prescribed in~\eqref{eq:LoadSurgeRecursion}. We point out that due to our coupling, if after step $j-1$ an edge $e_{(j)}$ for some $j \in [m]$ has sufficient surplus capacity to deal with the load surge, then so do all the other edges in that component. In other words, the cascade stops in that component. To determine the failure size of the cascade, one needs to subsequently compare the load surge values to the surplus capacities in all components until the surplus capacities at all the surviving edges are sufficient to deal with the load surges. 

In summary, sequentially removing edges uniformly at random gives rise to the load surge values in every component. This idea decouples the two sources of randomness: first one needs to understand the disintegration of the network through the sequential removal process, leading to the load surge values through~\eqref{eq:LoadSurgeRecursion}. Then, we can determine the failure size by comparing the surplus capacities to the load surges up to the point where the cascade stop in every component. 

\subsection{Disintegration of the network through sequential removal}
The next step is to determine the typical behavior of the sequential removal process on the connected configuration model $\overline{CM}_n(\mathbf{d})$. A first question that needs to be answered is how many edges need to fail, or equivalently, need to be removed uniformly at random from the network $\overline{CM}_n(\mathbf{d})$ to become disconnected. It turns out that this is likely to occur when $\Theta(\sqrt{m})$ edges are removed.

\begin{restatable}{theorem}{Firstdisconnection}\label{thm:Firstdisconnection}
	Suppose that we subsequently remove edges uniformly at random from $\overline{CM}_n(\mathbf{d})$. Define $T_{n,\mathbf d}$ as the smallest number of edges that need to be removed for the network to be disconnected. Then,
	\begin{align*}
	m^{-1/2}T_{n,\mathbf d} \dto T(D),
	\label{eq:FirstDisconnectionConvergence}
	\end{align*}
	where $T(D)$ has a Rayleigh distribution with density function 
	\begin{align}
	f_{T(D)}(x)= \frac{4p_2 x}{d-2p_2}\mathrm e^{-\frac{2x^2p_2}{d-2p_2}}.
	\end{align}
\end{restatable}

After this point, more disconnections start to occur as more edges are removed uniformly at random from the graph. In Section~\ref{sec:TypicalStructuresPercolation}, we focus on the typical network structure if $\sqrt{m} \ll i \ll m^{1-\delta}$ edges have been removed uniformly at random for some $\delta \in (0,1/2)$. Typically, there is a giant component that contains almost all edges and vertices. The components that detach from the giant component are isolated nodes, line components, and possibly isolated cycles. We show that the total number of edges contained in these components are likely to be of order $\Theta(i^2/m)$, while the number of edges in more complex component structures are negligible in comparison. This leads to the following result, which is proved in Section~\ref{sec:NumberOfEdgesOutsideTheGiant}.

\begin{restatable}{theorem}{FailureSublinear}\label{thm:FailureSublinear}
	Suppose we remove $i:=i_m$ edges uniformly at random from the connected configuration model $\overline{CM}_n(\mathbf{d})$ with $\sqrt{m}\ll i\ll m^{1-\delta}$ for some $\delta >0$. Write $\hat{E}_m(i)$ for the set of edges in the largest component of this graph, and let $|\hat{E}_m(i)|$ denote its cardinality. Then,
	\begin{align}
	(m-i-|\hat{E}_m(i)|)\frac{m}{i^2}\pto \frac{ 4p_2^2}{(d-2p_2)^2}.
	\end{align}
\end{restatable}

We stress that determining the typical behavior of the network disintegration is not enough to prove our main result. In addition, for each of our results, we need to show that it is extremely unlikely to be far from its typical behavior, which we consider in Section~\ref{sec:LargeDeviationBoundsBeyondTypicalSizes}. Moreover, we combine these large deviations results to show the following result in Section~\ref{sec:NumberOfEdgesOutsideTheGiant}.

\begin{restatable}{theorem}{DeviationsOutsideGiant}
	Suppose $i:=i_m$ edges have been removed uniformly at random from the connected configuration model $\overline{CM}_n(\textbf{d})$, where $\sqrt{m} \ll i \ll m^{\alpha}$ for some $\alpha \in (1/2,1)$. Then,
	\begin{align}\label{eq:FirstClaimCorollaryEdgesOutsideGiant}
	\Prob\left( m-i - |\hat{E}_m(i)| > i^\alpha \right) = O(m^{-3}).
	\end{align}
	Moreover, for every $k:=k_m$ for which $k = o(m^{\alpha})$ for some $\alpha \in (0,1)$,
	\begin{align}
	\Prob\left( m-i - |\hat{E}_m(i)| > i^\alpha \textrm{ for some } i \leq k \right) = o(m^{-1/2}).
	\label{eq:SecondClaimCorollaryEdgesOutsideGiant}
	\end{align}
	\label{thm:DeviationsOutsideGiant}
\end{restatable}

\subsection{Impact of disintegration on the failure process}\label{sec:RoadMapCascFailureProcess}
The sequential removal process gives rise to the load surge function at every edge, and we need to compare these values to the surplus capacities in every component until the cascade stops. To keep track of the cascading failure process in every component may seem cumbersome at first glance. However, due to the way the connected configuration model is likely to disintegrate, it turns out that it only matters what happens in the giant component.

Intuitively, this can be understood as follows. By Theorem~\ref{thm:Firstdisconnection} and~\ref{thm:FailureSublinear}, removing any sublinear number $i_m$ of edges always leaves~$o(i_m)$ edges outside the giant. Therefore, if $k:=k_m \ll m$ edges are sequentially removed, then only $o(k)$ of these edges were contained outside the giant upon removal. Moreover, even if the cascading failure process struck every edge of the components outside the giant, the contribution to the failure size would be at most $o(k)$. This is negligible with respect to the $k-o(k)$ failures that occur in the giant. The failure size $A_{n,\mathbf{d}}$ should therefore be well-approximated by the failure size in the giant. We formalize these ideas in Section~\ref{sec:CascadingFailureProcess}.

More specifically, write
\begin{align*}
\hat{A}_{n,\mathbf{d}} &= \textrm{\# edges contained in the giant upon failure during the cascade},\\
\tilde{A}_{n,\mathbf{d}} &= \textrm{\# edges contained outside the giant upon failure during the cascade},
\end{align*}
and hence $A_{n,\mathbf{d}}= \hat{A}_{n,\mathbf{d}}+ \tilde{A}_{n,\mathbf{d}}$. Moreover, define
\begin{align*}
|\tilde{E}_m(i)| =& \textrm{\# remaining edges outside the giant when $i$ edges are removed uniformly at random,}\\
\kappa(i) =& \textrm{\# edges removed from giant when $i$ edges are removed uniformly at random.}
\end{align*}

The main idea is that since the number of edges outside the giant is likely to be $o(i)$ when a sublinear number of $i$ edges are removed uniformly at random, the contribution of edge failures that occur outside the giant is asymptotically negligible. That is, we bound the probability that $\{A_{n,\mathbf{d}} \geq k \}$ to occur due to the cascading failure process by the same event in two related processes. For the upper bound, we consider the scenario where all edges in the small components that disconnect from the giant immediately fail upon disconnection. For the lower bound, we consider the scenario where none of the edges in the small components that disconnect from the giant fail. 

To make this rigorous, we first consider the probability that the number of edge failures in the giant exceeds $\kappa(k)$.  That is, if we sequentially remove $k$ edges uniformly at random, a set of $\kappa(k)$ edges were contained in the giant upon failure. We are interested in the probability that the coupled surplus capacities of all these edges are insufficient to deal with the corresponding load surges. By translating this setting to a first-passage problem of a random walk bridge over a moving boundary, we show the following result in Section~\ref{sec:AsympBehvGiant}. 

\begin{restatable}{proposition}{SublinearK}
	If $k= o(m^{\alpha})$ for some $\alpha \in (0,1)$, then as $n \rightarrow \infty$,
	\begin{align*}
	\Prob\left(\hat{A}_{n,\mathbf{d}} \geq \kappa(k) \right) \sim \frac{2\theta}{\sqrt{2\pi}} k^{-1/2}.
	\end{align*}
	\label{prop:SublinearK}
\end{restatable}

Second, we use this result to derive an upper bound for the failure size tail. Trivially, the failure size is bounded by the number of edges that are contained in the giant component upon failure plus all the edges that exist outside the giant after the cascade has stopped. We introduce the stopping time
\begin{align}
\upsilon(k) = \min\{j \in \mathbb{N} : j + |\tilde{E}_m(j)| \geq k \},\label{eq:upsilonDefinition}
\end{align}
the minimum number of edges that need to be removed uniformly at random for the number of edges outside the giant to exceed $k-\upsilon(k)$. In other words, after we have removed $\upsilon(k)$ edges uniformly at random, the sum of the number of edges outside the giant and the number of removed edges exceeds $k$. The number of edge removals in the giant is given by $\kappa(\upsilon(k))$, and hence
\begin{align*}
\left\{A_{n,\mathbf{d}} \geq k  \right\} \subseteq \left\{\hat{A}_{n,\mathbf{d}} \geq \kappa(\upsilon(k)) \right\}.
\end{align*}
We prove that $\upsilon(k) = k - o(k)$ with sufficiently high probability, and hence
\begin{align*}
\Prob\left(A_{n,\mathbf{d}} \geq k  \right) \leq \Prob\left(\hat{A}_{n,\mathbf{d}} \geq \kappa(\upsilon(k)) \right) \sim \frac{2\theta}{\sqrt{2\pi}} \left(k-o(k)\right)^{-1/2} \sim \frac{2\theta}{\sqrt{2\pi}} k^{-1/2} .
\end{align*}

Third, we derive a lower bound. We note that $\hat{A}_{n,\mathbf{d}} \leq A_{n,\mathbf{d}}$, and hence
\begin{align*}
\left\{A_{n,\mathbf{d}} \geq k  \right\} \supseteq \left\{\hat{A}_{n,\mathbf{d}} \geq k\right\} = \left\{\hat{A}_{n,\mathbf{d}} \geq \kappa(\varrho(k) ) \right\},
\end{align*}
where
\begin{align}
\varrho(k) = \min\{j \in \mathbb{N} : \kappa(j) \geq k \}.\label{eq:varrhoDefinition}
\end{align}
That is, $\varrho(k)$ is the number of edges that need to be removed uniformly at random for $k$ failures to have occurred in the giant. We show that that $\varrho(k)=k+o(k)$ with sufficiently high probability, and hence
\begin{align*}
\Prob\left(A_{n,\mathbf{d}} \geq k  \right) \geq \Prob\left(\hat{A}_{n,\mathbf{d}} \geq \kappa(\varrho(k)) \right) \sim \frac{2\theta}{\sqrt{2\pi}} \left(k+o(k)\right)^{-1/2} \sim \frac{2\theta}{\sqrt{2\pi}} k^{-1/2} .
\end{align*}
Since the upper and lower bounds coincide, this yields Theorem~\ref{thm:MainResultSublinear}. We formalize this in Section~\ref{sec:ProofOfMainResultRG}.

\section{Disintegration of the network}\label{sec:Disintegration}
As explained in the previous section, the cascading failure process can be decoupled in two elements of randomness. In this section, we study the first element of randomness: the disintegration of the network as we sequentially remove edges uniformly at random. In view of Theorem~\ref{thm:MainResultSublinear}, our main focus is the case where we remove only $O(m^{1-\delta})$ edges with $\delta \in (0,1)$. In particular, our goal is to prove Theorems~\ref{thm:Firstdisconnection}-\ref{thm:DeviationsOutsideGiant}. 

This section is structured as follows. First, we show in Section~\ref{sec:PercolationSequentialRelation} that the sequential removal process is well-approximated by a process where we remove each edge independently with a certain probability, also known as percolation. This is a well-studied process in the literature, and particularly in case of the configuration model~\cite{Jan09b}. In Section~\ref{sec:ExplosionAlg}, we provide an alternative way to construct a percolated configuration model by means of an algorithm as developed in~\cite{Jan09b}. This alternative construction allows for simpler analysis, and is used in Section~\ref{sec:TypicalStructuresPercolation} to show that for the percolated (connected) configuration model, typically, the components outside the giant component are either isolated nodes, line components or possibly cycle components. In Section~\ref{sec:LargeDeviationBoundsBeyondTypicalSizes}, we derive large deviations bounds on the number of edges outside the giant. We prove Theorem~\ref{thm:Firstdisconnection} in Section~\ref{sec:FirstDisconnections}, and in addition, we provide a large deviations bound for the number of edges that need to be removed for the connected configuration model to become disconnected. We prove Theorems~\ref{thm:FailureSublinear} and~\ref{thm:DeviationsOutsideGiant} in Section~\ref{sec:NumberOfEdgesOutsideTheGiant}. Although these sections focus on the (connected) configuration model under a sublinear number of edge removals, we briefly recall results known in the literature involving the typical behavior of the random graph beyond the sublinear window in Section~\ref{sec:LinearEdgeFailures}. This will be important to obtain intuition of what happens to the failure size behavior beyond the sublinear case, the topic of interest in Section~\ref{sec:UniversalityPrinciple}.

\subsection{Percolation on the connected configuration model}\label{sec:PercolationSequentialRelation}
To prove our results regarding the sequential removal process, we relate this process to another one where each edge is removed independently with a certain removal probability, also known as percolation. This is illustrated by the following lemma.

\begin{lemma}\label{lem:CascadePercolation}
	Let $G=(V,E)$ be a graph, and write $E'(G(q))$ as the set of edges that have been removed by percolation with parameter $q \in [0,1]$, and $\tilde{E}'(i)$ as the set of edges when $i$ edges are removed uniformly at random. It holds for every $i \in [m]$, $q\in [0,1]$, and $B\subseteq E$ such that $|B|=i$,
	\begin{align}\label{eq:iUniform}
	\Prob(\tilde{E}'(i)=B)=\Prob\left(E'(G(q))=B \;\big|\; |E'(G(q))|=i \right).
	\end{align}
	Moreover, if $q=q_m=\omega(m^{-1})$, then as $m \rightarrow \infty$,
	\begin{align}\label{eq:RemovedEdges}
	\frac{|E'(G(q))|}{qm}\overset{\mathbb{P}}{\rightarrow}1.
	\end{align}
\end{lemma}

Lemma~\ref{lem:CascadePercolation} is a direct consequence of the concentration of the binomial random variable, and the proof is given in Appendix~\ref{app:PercolationResults}. This lemma establishes that sequentially removing $i:=i_m = \omega(1)$ edges uniformly at random is well-approximated by a percolation process with removal probability $q=i/m$. In particular, this holds for the (connected) configuration model. This allows us to study many questions involving the connected configuration model subject to uniformly removing edges in the setting of percolation. The study of percolation processes on finite (deterministic or random) graphs is a well-established research field, which dates back to the work of Gilbert~\cite{Gil59} and is still very active these days. In particular, percolation on the configuration model is a fairly well-understood process, as established by Janson in~\cite{Jan09b}.

It is known that there exists a critical parameter $q_c := 1- \E [D]/\E[D(D-1)]$ such that if $q < q_c$, then the largest component of the percolated (connected) configuration model contains a non-vanishing proportion of the vertices and edges~\cite{MolRee95,Jan09b}. In particular, it is implied by the formula for $q_c$ that in order for a phase transition to appear, it is necessary that $\E[D(D-1)]/\E[D] \in (1, \infty)$. If $\E[D(D-1)]/\E[D] \leq 1$, then the largest component already has a sublinear size in $n$ even before percolation, while if $\E[D(D-1)]/\E[D] = \infty$, then there exists a connected component of linear size in the percolated graph for every $q \in (0,1)$. Typically, the (w.h.p.~unique) component of linear size is referred to as the \emph{giant component}. All other components are likely to be much smaller, i.e.~of size $\Op( \log n)$ under some regularity conditions on the limiting degree sequence~\cite{MolRee95}. If $q\geq q_c$, there is no giant component in the percolated graph and all the components have sublinear size~\cite{Jan09b}. 

In order to derive Theorem~\ref{thm:MainResultSublinear}, these results in the literature are not sufficient. In particular, in view of Theorems~\ref{thm:Firstdisconnection} and~\ref{thm:FailureSublinear}, a more detailed description of the network structure is needed in case that $q=o(m)$. A critical tool to derive the results in this setting is the \textit{explosion algorithm}, as we describe next.

\subsection{Explosion algorithm}\label{sec:ExplosionAlg}
Key to prove Theorems~\ref{thm:Firstdisconnection} and~\ref{thm:FailureSublinear} is the \textit{explosion algorithm}, designed by Janson in~\cite{Jan09b}. This algorithm prescribes that instead of applying percolation on $\CMd$ with removal probability $q$, we can run the procedure as described in Algorithm~\ref{alg:ExplosionAlgorithm}.

\begin{center}
	\begin{algorithm}
		\KwIn{A set of $n$ vertices, such that for every $j \in [n]$ the vertex $v_j$ has $d_j$ half-edges attached to it according to the degree sequence $\mathbf{d}$.}
		\KwOut{Graph $CM_n(\mathbf{d},q)$.}
		\begin{enumerate}
			\item Remove each half-edge independently with probability $( 1-( 1-q)^{1/2})$. Let $R_n$ be the number \\of removed half-edges.
			\item Add $R_n$ degree-one vertices to the vertex set. Define the new degree sequence as $\mathbf{d}'$ with \\$N=n+R_n$ vertices.
			\item Pair $CM_N(\mathbf{d}')$.
			\item Remove $R_n$ vertices of degree 1 from $CM_N(\mathbf{d}')$ uniformly at random.
		\end{enumerate}
		\caption{Explosion algorithm~\cite{Jan09b}.}
		\label{alg:ExplosionAlgorithm}
	\end{algorithm}
\end{center}

Janson proved in~\cite{Jan09b} that the algorithm produces a random graph that is statistically indistinguishable from $CM_n(\mathbf{d},q)$, the graph obtained by removing every edge in (not necessarily connected) $CM_n(\mathbf{d})$ with probability $q \in [0,1]$. Yet, the graph obtained by the explosion method is significantly easier to study as it is simply a configuration model with a new degree sequence and a couple of vertices of degree~$1$ that have been removed, where the latter operation does not significantly change the structure of the graph. Since the graphs obtained from the percolation process and the explosion method are identically distributed, we use the denomination $CM_n(\mathbf{d},q)$ both for the configuration model after percolation and for the graph we obtain via the explosion algorithm. 

\begin{remark}\label{rem:NonVanishingProbabilityInitiallyConnected}
	The observant reader may notice that the explosion algorithm is designed for the configuration model that is not necessarily initially connected. Fortunately, as established in~\cite{Federico2016}, connectivity has a non-vanishing probability to occur under Condition~\ref{con:RegularConfiguration} as $n \to \infty$. This is a crucial observation, since it implies that certain events that happen w.h.p. on $CM_n(\mathbf{d})$ also happen w.h.p. on $\overline{CM}_n(\mathbf{d})$. More specifically, for a sequence of events $(A_n)_{n \in \mathbb{N}}$,
	\begin{align}\label{eq:ConnectedConditioning}
	\Prob (A_n \mid  \CMd \text{ is connected})\leq \frac{\Prob (A_n)}{ \Prob( \CMd \text{ is connected})},
	\end{align}
	where under Condition~\ref{con:RegularConfiguration}~\cite{Federico2016},
	\begin{align*}
	\liminf_{n \to \infty} \Prob( \CMd \text{ is connected})>0.
	\end{align*}
	In particular, this implies that if for some sequence of random variables $(X_n)_{n \in \mathbb{N}}$ it holds that $X_n \pto c$ for some constant $c \in \mathbb{R}$ in $CM_n(\mathbf{d})$, then the same statement holds for the graph $\overline{CM}_n(\mathbf{d})$. Similarly, if $X_n = \op(a_n)$ for some sequence $(a_n)_{n \in \mathbb{N}}$ in $CM_n(\mathbf{d})$, then this is also true for $\overline{CM}_n(\mathbf{d})$.
\end{remark}

Next, we point out some properties we use extensively in this section. First, we observe that if $q=o(1)$, then by Taylor expansion, the removal probability in the explosion algorithm satisfies $1-(1-q)^{1/2} = (q/2)(1+o(1))$. Therefore, the probability of a vertex of degree $l$ to retain $j$ half-edges satisfies
\begin{align}
p_{l,j} = \binom{l}{j} \left( (1-q)^{-1/2}\right)^j\left( 1-(1-q)^{1/2}\right)^{l-j} = \binom{l}{j} (q/2)^{l-j}(1+o(1))
\end{align}
if $q=o(1)$ and $j \leq l$. Moreover, let $n_{l,j}$ represent the number of vertices of degree $l$ that retain $j$ half-edges. That is, $n_{l,j}$ are random variables with distribution $n_{l,j} \overset{d}{=} Bin(n_l,p_{l,j})$. Due to Markov's inequality, it holds that
\begin{align}\label{eq:NljMarkovInequality}
\Prob(n_{l,j} >0) \leq n_l p_{l,j}.
\end{align}
Moreover,
\begin{align}
\begin{array}{ll}
n_{l,j} \overset{d}{\rightarrow} \textrm{Poi}(a) & \textrm{if } n_l p_{l,j} \rightarrow a \in (0,\infty), \\
n_{l,j} = n_l p_{l,j}(1+o_{\Prob}(1)) & \textrm{if } n_l p_{l,j} \rightarrow \infty,
\end{array}\label{eq:NljConvergences}
\end{align}
due to the Poisson limit theorem and the law of large numbers, respectively. Note that 
\begin{align}\label{eq:HighToZeroDegreeVertices}
\E\left(\sum_{l=h}^\infty n_{l,0}\right) = \sum_{l=h}^\infty n_{l}p_{l,0} \leq n \frac{(q/2)^h}{1-q/2}(1+o(1)) = O(n q^h), \hspace{1cm} h\geq 2.
\end{align}
By Markov's inequality, it holds for every $\epsilon >0$
\begin{align}
\Prob\left( \sum_{l=h}^\infty n_{l,0} > \epsilon \right) = O(n q^h), \hspace{1cm} h\geq 2.
\label{eq:ZeroDegreeNodeSumBoundExplosionAlg}
\end{align}
Similarly, it follows that for all $q=o(1)$,
\begin{align}
\E\left( \sum_{l=h}^\infty n_{l,1} \right) \leq n \sum_{l=h}^\infty l (q/2)^{l-1} (1+o(1)) = O(n q^{h-1}), \hspace{1cm} h\geq 2,
\end{align}
and hence for every $\epsilon>0$,
\begin{align}
\Prob\left( \sum_{l=h}^\infty n_{l,1} >\epsilon \right) = O(n q^{h-1}), \hspace{1cm} h\geq 2,
\label{eq:OneDegreeNodeSumBoundExplosionAlg}
\end{align}

Finally, we observe that for every $1/m \ll q \ll 1$, by the law of large numbers,
\begin{align}
R_n = 2m(1-(1-q)^{-1/2})(1+\op(1)) = \frac{nd}{2} q (1+\op(1))\label{eq:RnLLN}
\end{align}

\subsection{Typical structure of the percolated configuration model}\label{sec:TypicalStructuresPercolation}
We recall that our focus is on the case where the number of edges that are removed from the giant is of order $o(m)$. In view of Lemma~\ref{lem:CascadePercolation}, this number is well-approximated by the number of edge removals in a percolation process with removal probability $q=o(1)$. In particular, in this regime there is a unique giant component w.h.p.~and other components are likely to be much smaller, i.e.~w.h.p.~the number of vertices and edges ourside the giant is of order $o(m)$. Even more can be said about the structure of these components. In this section, we show in that these components are typically isolated nodes, line components or potentially isolated circles. More complex structures are relatively rare. 

\begin{remark}\label{rem:CanAlsoLookAtExplodedCMd}
	Remark~\ref{rem:NonVanishingProbabilityInitiallyConnected} implies that often it suffices to consider $CM_n(\mathbf{d})$ to prove an analogous result for $\overline{CM}_n(\mathbf{d})$. Moreover, the explosion algorithm is used to construct $CM_n(\mathbf{d},q)$ from $\Cmd$ by the removal of $R_n$ degree-one vertices, and hence we often also focus on $\Cmd$ in this section. We point out that the operation of removing degree-one vertices does not affect the connectivity of a component. Moreover, the probability that the giant component in $\Cmd$ is not unique is exponentially small~\cite{MolRee95}. If $q=o(1)$, then the probability for $R_n$ not to be of sublinear size is also exponentially small. Therefore, the giant component in $\Cmd$ remains the giant component in $CM_n(\mathbf{d},q)$ with extremely high probability, i.e. the probability that the complement is true has an exponentially decaying rate. Therefore, the number of edges outside the giant in $\Cmd$ provides an upper bound (with extremely high probability) for the number of edges outside the giant in $CM_n(\mathbf{d},q)$. Moreover, since line components outside the giant in $\Cmd$ either remain line components or become isolated nodes after the removal of degree-one vertices, the number of edges in $CM_n(\mathbf{d},q)$ outside the giant that are not contained in line components is bounded from above (with extremely high probability) by the same quantity in $\Cmd$.
\end{remark}

First, we explore the degree sequence~$\mathbf{d}'$ of $\Cmd$ in the explosion method from Algorithm~\ref{alg:ExplosionAlgorithm}. Analogously to the notation for the original degree sequence $\mathbf{d}$, we write $n_i'$ for the number of vertices of degree~$i$ in $\mathbf{d}'$ and $p_i' := \lim_{n \rightarrow \infty} n_i'/n$ as the limiting fraction.

\begin{lemma}\label{lem:NewDegreeDistribution}
	Consider the explosion algorithm from Algorithm~\ref{alg:ExplosionAlgorithm} with initial graph $\CMd$ satisfying Condition~\ref{con:RegularConfiguration} and~$q=i/m$ with $1\ll i\ll m$. The degree sequence $\mathbf{d}'$ after explosion satisfies the following properties. For the number of vertices of degree zero,
	\begin{align}\label{eq:n0Prime}
	\begin{array}{ll}
	\Prob(n'_0\neq 0)= O(i^2/m), & \textrm{ if } i\ll m^{1/2},\\
	n'_0 \dto Poi\Big(\frac{c^2p_2}{2d}\Big), & \textrm{ if } i=c\sqrt m,\\
	n'_0= \frac{p_2}{2d} \frac{i^2}{m} (1+\op(1)), & \textrm{ if }  \sqrt{m} \ll i \ll m.
	\end{array}
	\end{align}
	The number of degree-one vertices in $\Cmd$ satisfies
	\begin{align}
	n'_1 &=  \left( i+ \frac{2 p_2}{d} i\right)(1+\op (1)).
	\end{align} 
	Finally, the fraction of vertices of degree $k\geq 2$ in $\Cmd$ converges to the limiting fraction of degree-$k$ vertices in $\CMd$, i.e.
	\begin{align}
	p'_k &=p_k, \hspace{0.5cm} \textrm{ for all } k \geq 2.
	\end{align}
\end{lemma}

\begin{proof}
	Recall that if $q = o(1)$, then 
	\begin{align*}
	1-(1-q)^{1/2} &=(q/2)(1+o(1))=(i/(2m))(1+o(1)),\\
	p_{l,j} &=\binom{l}{j}(q/2)^{l-j}(1+o(1)),
	\end{align*}
	and $n_{l,j}\overset{d}{=}Bin(n_l,p_{l,j})$. Using~\eqref{eq:NljConvergences}, we obtain
	\begin{align*}
	n_{2,0}= \frac{p_2i^2}{2dm}(1+\op(1)) \text{ if } m^{1/2} \ll i \ll m, \quad  n_{2,0} \dto Poi\Big(\frac{c^2p_2}{2d}\Big) \text{ if } i=c\sqrt m .
	\end{align*}
	By \eqref{eq:HighToZeroDegreeVertices} we know that in both cases, these are the only leading-order contributions to $n_0'$, since~$n_0'=\sum_{h=2}^\infty n_{l,0}$ with $\sum_{h=3}^\infty n_{l,0}= \Op(n i^3/m^2)=\op(i^2/m)$.
	From~\eqref{eq:ZeroDegreeNodeSumBoundExplosionAlg}, we obtain that
	\begin{align*}
	\Prob\left( n_0'\neq 0\right) = \Prob\left( \sum_{l=2}^\infty n_{l,0}\neq 0 \right)=O(i^2/m), 
	\end{align*}
	if $i \ll \sqrt m$. Similarly, it follows from~\eqref{eq:NljConvergences},~\eqref{eq:OneDegreeNodeSumBoundExplosionAlg} and~\eqref{eq:RnLLN},
	\begin{align*}
	n'_1 =(R_n+ n_{2,1})(1+\op(1))=  \left( i+ \frac{2 p_2}{d} i\right)(1+\op (1)).
	\end{align*}
	Finally, we note that, since every removal of a half-edge in the first step of the explosion algorithm changes the degree of one vertex,
	\begin{align*}
	n_l-R_n\leq n'_l \leq n_l+R_n,
	\end{align*}
	with $R_n=\Op(i)=\op(n)$ and hence $p'_l \pto p_l$ for all $l \geq 2$.
\end{proof}

We use the degree sequence $\mathbf{d}'$ to study the structure of the components outside the giant in $\overline{CM}_n(\mathbf{d},q)$. Again, we will show that these components are either isolated nodes, line components or possibly isolated cycles. More complex structures are rather unlikely to appear as the network disintegrates. 

We begin by proving a bound on the number of edges belonging to components that have a more complex structure, i.e.~contain a vertex of degree at least three when $q\ll m^{-\delta}$ for some $\delta > 0$. We show that this is not the leading-order term for the number of edges outside the giant, since the number of lines and isolated vertices is much larger. For a graph $G$, define $d_{\max}(G)$ as the largest degree of any vertex in $G$ and $E(G)$ as the edge set of $G$. For every edge $e$ and vertex $v$, write $\mathcal C (e)$ and $\mathcal C(v)$ for the connected component that contains $e$ or $v$, respectively. Finally, let $\Cmax$ denote the largest component in a graph $G$.
\begin{proposition}
	Consider the graphs $\overline{CM}_n(\mathbf{d},q)$, $CM_n(\mathbf{d},q)$ and $\Cmd$ with $q=i/m$, where $m^\delta \ll i \ll m^{1-\delta}$ for some $\delta >0$. For all three graphs, if $i = o(\sqrt{m})$, then
	\begin{align}
	\Prob(\#\{ e \notin \Cmax\colon d_{\max}(\mathcal C(e)) \geq 3\}\neq 0)= o(i^2/m).
	\label{eq:OtherStuffLessThanSquareRoot}
	\end{align}
	For both graphs, if $i \gg \sqrt{m}$, then
	\begin{align}\label{eq:OtherStuffMoreThanSquareRoot}
	\# \{ e \notin \Cmax\colon d_{\max}(\mathcal C(e)) \geq 3\} = \op\left(i^2/m\right).
	\end{align}
	\label{prop:NumberOfOtherStuffOutsideGiant}
\end{proposition}

\begin{proof}
	We prove the result using the explosion algorithm. First, recall Remark~\ref{rem:NonVanishingProbabilityInitiallyConnected}. Applying~\eqref{eq:ConnectedConditioning} to the events
	\begin{align*}
	&\left\{ \#\{ e \notin \Cmax\colon d_{\max}(\mathcal C(e)) \geq 3\}\neq 0\right\}, \\
	&\left\{ \frac{m}{i^2}\#\{ e \notin \Cmax\colon d_{\max}(\mathcal C(e)) \geq 3\} > \epsilon \right\}, \;\;\; \epsilon >0,
	\end{align*}
	implies that to prove~\eqref{eq:OtherStuffLessThanSquareRoot} and~\eqref{eq:OtherStuffMoreThanSquareRoot} for $\overline{CM}_n(\mathbf{d},q)$, it suffices to show~\eqref{eq:OtherStuffLessThanSquareRoot} and~\eqref{eq:OtherStuffMoreThanSquareRoot} hold for $CM_n(\mathbf{d},q)$. In addition, recall Remark~\ref{rem:CanAlsoLookAtExplodedCMd}, implying that the number of edges in~$\CMD \setminus \Cmax$ in components containing a node with degree at least three is bounded by the same quantity in $\Cmd \setminus \Cmax$ with sufficiently high probability. In other words, it suffices to prove that~\eqref{eq:OtherStuffLessThanSquareRoot} and~\eqref{eq:OtherStuffMoreThanSquareRoot} hold for $\Cmd$.
	
	Recall the degree distribution of $CM_N(\mathbf{d}')$ from Lemma~\ref{lem:NewDegreeDistribution}. It follows from the proof of~\cite[Lemma 11]{MolRee95} that for every supercritical degree sequence $\mathbf{d}'$ (i.e.~$\E[D'_n(D'_n-1)]/\E(D_n')>1$) and any $\gamma \in (0,\infty)$, there exists $c=c(\mathbf d')<1$ such that in $\Cmd$, $\Prob(\exists \mathcal{C}\neq \Cmax : |E(\mathcal{C})|> \gamma \log n)\leq n^2c^{\gamma \log n}$. Therefore, for $\gamma $ large enough,
	\begin{align}
	\Prob(\exists \, \mathcal{C}\neq \Cmax : |E(\mathcal{C})|> \gamma \log n)=o(n^{-1}).
	\end{align}
	
	Consequently, since the number of edges in the giant component is much larger than $\gamma \log n$, it suffices to prove the claims~\eqref{eq:OtherStuffLessThanSquareRoot} and~\eqref{eq:OtherStuffMoreThanSquareRoot} for
	\begin{align}\label{eq:SmallComponentOutside}
	\# \left\{ e \in \Cmd : |E(\mathcal C(e))| \leq \gamma \log n, d_{\max}(\mathcal C(e)) \geq 3\right\}.
	\end{align}

	For this purpose, we use the standard exploration algorithm of $CM_N(\mathbf{d}')$ used in the literature (see e.g.~\cite{Federico2016, DhaHofLeeSen16a} for some similar formulations). At each time $t \in \mathbb{N}$, we define the sets of half-edges $\{ \Acal_t , \Dcal_t , \Ncal_t \}$ as the active, dead and neutral sets, and explore them in the following way:
	\begin{enumerate}
		\item At step $t=0$, pick a vertex $v \in [n]$  with $d_v \geq 3$ uniformly at random and set all its half-edges as active. All other half-edges are set as neutral, and $ \Dcal_0 = \emptyset$.
		\item At each step $t$, pick a half-edge $e_1(t) $ in $\Acal_t$ uniformly at random, and pair it with another half-edge $e_2(t)$ chosen uniformly at random in $\Acal_t \cup \Ncal_t$. Set $e_1(t), e_2(t)$ as dead. If $e_2(t) \in \Ncal_t$, then find the vertex $v(e_2(t))$ incident to $e_2(t)$ and activate all its other half-edges.
		\item Terminate the process when $\Acal_t= \emptyset$.
	\end{enumerate}
	
	\noindent
	We observe that when $\Acal_t=\emptyset$, we have exhausted the exploration of the connected component $\mathcal C (v)$, and the number of steps performed by the exploration algorithm is the number of edges in~$\mathcal{C}(v)$ In order to prove the claim, we thus need to prove that there is no $t\leq \gamma \log n$ such that $\Acal_t=\emptyset$ with sufficiently high probability. Let $(Z^{(v)}_t)_{t\geq 0}$ count the number of active half-edges starting from a vertex $v$ with $d_v\geq 3$. Note that at step $t$ the process can only go down if i) $e_2(t) \in \mathcal{N}_t$ and its incident vertex has degree one, causing $Z^{(v)}_t=Z^{(v)}_{t-1}-1$, or ii) $e_2(t) \in \mathcal{A}_t$, causing $Z^{(v)}_t=Z^{(v)}_{t-1}-2$. We denote these events by $A(t)$ and $B(t)$, respectively. Since $Z^{(v)}_0 \geq 3$, this counting process needs to decrease by at least three in total for the exploration process to die out. Moreover, the values of the counting process is small at the time steps where the process decreases. More specifically, $\left\{\Acal_{\gamma \log n}=\emptyset \right\} \subseteq F_1 \cup F_2 \cup F_3$, where
	\begin{align*}
	F_1 =\bigcup_{s_1,s_2,s_3 \leq \gamma \log n}& A(s_1) \cap A(s_2) \cap A (s_3) \cap \{Z_{s_1}^{\sss(v)},Z_{s_2}^{\sss(v)},Z_{s_3}^{\sss(v)} \leq 3 \}, \\ 
	F_2=\bigcup_{s_1,s_2 \leq \gamma \log n}&  A(s_1) \cap B (s_2) \cap \{ Z_{s_1}^{\sss(v)},Z_{s_2}^{\sss(v)} \leq 3 \},  \\
	F_3=\bigcup_{s_1,s_2 \leq \gamma \log n}&  B(s_1) \cap B (s_2) \cap \{ Z_{s_1}^{\sss(v)},Z_{s_2}^{\sss(v)} \leq 4 \}.
	\end{align*}
	We can bound the probabilities that these events occur by
	\begin{align*}
	\Prob (F_1) &\leq \left(\gamma \log n\right)^3 \left(1+\frac{2p_2}{d} \right)^3 \frac{i^3}{m^3}(1+o(1)) = O \left( \frac{i^3 \log^3 m}{m^3}\right), \\
	\Prob (F_2) &\leq \left(\gamma \log n\right)^2 \left(1+\frac{2p_2}{d} \right) \frac{i}{m}(1+o(1)) \frac{3}{m-2\gamma \log n}  = O \left( \frac{i \log^2 m}{m^2}\right),\\
	\Prob (F_3) &\leq \left(\gamma \log n\right)^2 \left( \frac{4}{m-2\gamma \log n} \right)^2 = O \left( \frac{\log^2 m}{m^2}\right).
	\end{align*}
	Consequently, using the union bound, we obtain that for every $i$ that satisfies $m^\epsilon \ll i \ll m^{1-\epsilon}$ for some $\epsilon >0$,
	\begin{align}
	\E &\left[\# \left\{ e : |E(\mathcal C(e))| \leq \gamma \log n, d_{\max}(\mathcal C(e)) \geq 3\right\}\right] \leq n \gamma \log n \left(\Prob (F_1)+\Prob (F_2)+\Prob (F_3)\right) = o\left( \frac{i^2}{m}\right).
	\label{eq:OtherStuffdprime}
	\end{align}
	By Markov's inequality, it follows that
	\begin{align*}
	\Prob\left(\# \left\{ e : |E(\mathcal C(e))| \leq \gamma \log n, d_{\max}(\mathcal C(e)) \geq 3\right\} \neq 0\right )= o(i^2/m),
	\end{align*}
	and 
	\begin{align*}
	\# \left\{ e : |E(\mathcal C(e))| \leq \gamma \log n, d_{\max}(\mathcal C(e)) \geq 3\right\} = \op\left(i^2/m\right).
	\end{align*}
\end{proof}

The following proposition specifies the number of vertices and edges in lines and the number of isolated nodes which are disconnected from the giant by a percolation process, which constitutes the leading-order term for $\CMD \setminus \Cmax$.

\begin{proposition}
	Consider $CM_n(\mathbf{d},q)$ with $q=i/m$ with $\sqrt{m}\ll i\ll m$. Define $L_k(n)$ as the number of isolated lines of length $k$ and $N_0(n)$ the number of isolated vertices. Then,
	\begin{align}\label{eq:LinesMoreThanSquareRoot}
	\frac{m}{i^2}\Big(N_0(n)+\sum_{k=2}^\infty kL_k(n)\Big)\pto \frac{ 2d p_2}{(d-2p_2)^2},\\
	\frac{m}{i^2}\Big(\sum_{k=2}^\infty (k-1)L_k(n)\Big)\pto \frac{ 4 p_2^2}{(d-2p_2)^2}. \label{eq:LinesMoreThanSquareRoot2}
	\end{align}
	
	\noindent
	Instead, if $i= o(\sqrt{m})$, then
	\begin{align}\label{eq:LinesLessThanSquareRoot}
	\Prob\Big(N_0(n)+\sum_{k=2}^\infty kL_k(n) \neq 0\Big)= O(i^2/m).
	\end{align}
	
	\noindent
	Moreover, \eqref{eq:LinesMoreThanSquareRoot}-\eqref{eq:LinesLessThanSquareRoot} hold also for $\overline{CM}_n(\mathbf{d},q)$.
	\label{prop:NumberOfLinesNodesOutsideGiant}
\end{proposition}

Before moving to the proof of Proposition~\ref{prop:NumberOfLinesNodesOutsideGiant}, we first consider the higher moments of  $L_k'(n), k\geq 1$, the number of isolated lines of length~$k$ in $\Cmd$. 

\begin{lemma}
	For any sequence $\mathbf{r}=\{r_2,...,r_k\}$ with $k\geq2$ of positive integer values, it holds as $n\rightarrow \infty$,
	\begin{align}
	\E [L_2'(n)^{r_2}\cdots L_k'(n)^{r_k}]\Big(\frac{m}{i^2}\Big)^{r_2+...+r_k}\to \prod_{j=2}^{k} \left( \frac{1}{4}\Big(1+ \frac{2p_2}{d}\Big)^2\Big(\frac{2p_2}{d}\Big)^{j-2}\right)^{r_j}.
	\label{eq:MultivariateMoments}
	\end{align}
	\label{lem:MultiVariateMoments}
\end{lemma}

From Lemma~\ref{lem:MultiVariateMoments}, we can bound the number of edges in isolated line components in $\Cmd$.

\begin{corollary}
	For every $j\geq 1$, as $n \rightarrow \infty$,
	\begin{align}\label{eq:HigherMoments}
	\E\Big[\Big(\frac{m}{i^2}\sum_{k=2}^\infty kL_k'(n)\Big)^j\Big]\to  \Big(\frac{(d - p_2) (d + 2 p_2)^2}{2 d (d - 2 p_2)^2}\Big)^j.
	\end{align}
	\label{cor:HigherMoments}
\end{corollary}

The proof of Lemma~\ref{lem:MultiVariateMoments} and Corollary~\ref{cor:HigherMoments} is given in Appendix~\ref{app:PercolationResults}. Next, we prove Proposition~\ref{prop:NumberOfLinesNodesOutsideGiant}.

\begin{proof} [Proof of Proposition \ref{prop:NumberOfLinesNodesOutsideGiant}]
	Again, note that it follows from Corollary~\ref{cor:HigherMoments} that
	\begin{align*}
	\E\Big[\sum_{k=2}^\infty kL_k'(n)\Big] &= O\left( \frac{i^2}{m} \right).
	\end{align*}
	By Markov's inequality and Lemma~\ref{lem:NewDegreeDistribution}, if $i=o(\sqrt{m})$,
	\begin{align*}
	\Prob\left(n_0' + \sum_{k=2}^\infty kL_k'(n) \neq 0\right) =  O\left( \frac{i^2}{m} \right).
	\end{align*}
	Recall that in the final step of the explosion algorithm, we only remove vertices of degree 1. Therefore, the only way for the number of vertices in line components and isolated vertices to increase is when a component in $\Cmd$ with a vertex that has a degree at least three turns into a line or an isolated node. By Proposition~\ref{prop:NumberOfOtherStuffOutsideGiant}, such components appear in $\Cmd$ with probability $\op(i^2/m)$, and we conclude that~\eqref{eq:LinesLessThanSquareRoot} holds.
	
	In order to prove \eqref{eq:LinesMoreThanSquareRoot} and~\eqref{eq:LinesMoreThanSquareRoot2}, we also need second moments. By Corollary~\ref{cor:HigherMoments}, 
	\begin{align*}
	\E\Big[\Big(\frac{m}{i^2}\sum_{k=2}^\infty kL_k'(n)\Big)^2\Big]=\E\Big[\Big(\frac{m}{i^2}\sum_{k=2}^\infty kL_k'(n)\Big)\Big]^2(1+o(1)),
	\end{align*}
	and thus, by the second moment method
	\begin{align*}
	\frac{m}{i^2}\sum_{k=2}^\infty kL_k'(n)\pto \frac{(d - p_2) (d + 2 p_2)^2}{2 d (d - 2 p_2)^2}.
	\end{align*}
	Since $n_0'= \frac{p_2i^2}{2dm}$ by Lemma~\ref{lem:NewDegreeDistribution}, we obtain that
	\begin{align*}
	\frac{m}{i^2}\Big(n_0'+\sum_{k=2}^\infty kL_k'(n)\Big)\pto \frac{p_2}{2d}+\frac{(d - p_2) (d + 2 p_2)^2}{2 d (d - 2 p_2)^2}.
	\end{align*}	
	The same arguments can also be used to prove concentration of $\sum_{k=2}^\infty (k-1)L_k'(n)$, since it is dominated by $\sum_{k=2}^\infty kL_k'(n)$. That is,
	\begin{align*}
	\frac{m}{i^2}\sum_{k=2}^\infty (k-1)L_k'(n)\pto \frac{(d + 2 p_2)^2}{4 (d - 2 p_2)^2}.
	\end{align*}
	
	We computed the number of vertices and edges that are contained in isolated node components or in line components in $\Cmd$. To obtain the corresponding value for $\CMD$ we need to subtract the number of degree one vertices that are part of line components and are removed in the last step of the explosion algorithm, and add the number of vertices whose components turn into a line or an isolated vertex by the removal of some degree $1$ vertices. By Proposition~\ref{prop:NumberOfOtherStuffOutsideGiant}, the number of components that can turn into a line or isolated vertex is bounded by $\op(i^2/m)$, and hence  the contribution of these type of events is negligible. Therefore, it suffices to consider the number of vertices and edges that are removed in the final step of the explostion algorithm from the line components in $\Cmd$. We observe that there are in total
	\begin{align}
	\sum_{k=2}^\infty 2L_k'(n)= \frac{ (d + 2 p)^2}{2d (d - 2 p)}\frac{i^2}{m}(1+\op(1)),
	\end{align}
	vertices of degree one in the line components out of the $i(1+2p_2/d)(1+\op(1))$ that exist in $\mathbf d'$. We define $L_R(n)$ as the number of vertices removed from line components in the last step of the explosion algorithm. We remove $i(1+\op(1))$ edges of degree one uniformly at random, so the number of the degree-one vertices removed from lines is given by an hypergeometric variable with 
	\begin{align}
	\E[L_R(n)]=\frac{ (d + 2 p_2)^2}{2d (d - 2 p_2)}\frac{d}{d+2p_2} (1+o(1)) = \frac{ d + 2 p_2}{2 (d - 2 p_2)} (1+o(1)) .
	\end{align}
	A hypergeometric random variable with diverging mean concentrates around the mean, and hence by the law of large numbers, 
	\begin{align*}
	\frac{m}{i^2}\Big( N_0(n)+\sum_{k=2}^\infty kL_k(n)\Big)\pto  \frac{p_2}{2d}+\frac{(d - p_2) (d + 2 p_2)^2}{2 d (d - 2 p_2)^2}-\frac{ d + 2 p_2}{2 (d - 2 p_2)} =\frac{2dp_2}{(d-2p_2)^2},
	\end{align*}
	as claimed.
	
	We do the same computation for the number of edges, this time accounting for the fact that if both vertices of a line of length $2$ are removed, only one edge is removed, while all the other vertex and edge removals are in bijection. The number of lines of length $2$ which are removed is given by 
	\begin{align}
	L_2'(n)\left(1+\frac{2p_2}{d}\right)^{-2}=\frac{i^2}{m} \frac{1}{4}\Big(1+ \frac{2p_2}{d}\Big)^2\Big(1+ \frac{2p_2}{d}\Big)^{-2}(1+\op(1))=\frac{i^2}{4m} (1+\op(1)).
	\end{align} 
	We conclude that
	\begin{align}
	\frac{m}{i^2}\sum_{k=2}^\infty (k-1)L_k(n) \overset{\Prob}{\rightarrow} \frac{(d + 2 p_2)^2}{4 (d - 2 p_2)^2}- \frac{ d + 2 p_2}{2 (d - 2 p_2)}+ \frac{1}{4} =  \frac{4p_2^2}{(d-2p_2)^2}.
	\end{align}
	
	Finally, we conclude that the claims also hold when for $\overline{CM}_n(\mathbf{d},q)$ by Remark~\ref{rem:NonVanishingProbabilityInitiallyConnected}, i.e. by applying \eqref{eq:ConnectedConditioning} to the events 
	\begin{align*}
	&\Big\{\Big|\frac{m}{i^2}\Big(N_0(n)+\sum_{k=2}^\infty kL_k(n)\Big)- \frac{2 d p_2}{(d-2p_2)^2}\Big|\geq \varepsilon\Big\},\\
	&\Big\{\Big|\frac{m}{i^2}\Big(\sum_{k=2}^\infty (k-1)L_k(n)\Big)- \frac{ 4 p_2^2}{(d-2p_2)^2}\Big|\geq \varepsilon\Big\},\\
	&\Big\{N_0(n)+\sum_{k=2}^\infty kL_k(n) \neq 0\Big\}.
	\end{align*}
\end{proof}

Proposition~\ref{prop:NumberOfLinesNodesOutsideGiant} indicates that the typical number of isolated nodes and line components outside the giant component is of order $\Theta(i^2/m)$. Naturally, the isolated nodes do not contribute to the number of edges outside the giant component, and hence we are mostly interested in the total number of edges in the line components, which is likely to be of order $\Theta(i^2/m)$ due to~\eqref{eq:LinesMoreThanSquareRoot2}.

Finally, we would like to comment on the number of isolated cycles in $\Cmd$. Let $C_k'(n), k\geq 1,$ denote the number of isolated cycles with $k$ edges. In view of Lemma~\ref{lem:NewDegreeDistribution}, if $q=o(1)$, then the degree distribution $\mathbf{d}'$ satisfies all conditions in~\cite{Federico2016} (with extremely high probability) except one, namely $n_1' \neq O(\sqrt{m})$. However, the proof of Theorem~3.3 in~\cite{Federico2016} does not use this condition to prove a bound on the number of isolated cycles: this condition was only needed to bound the number of line components. Therefore, it follows from~\cite[(5.18)]{Federico2016},
\begin{align}
\lim_{n \rightarrow \infty} \E\left[ \sum_{k\geq 1} k C_k'(n) \right] < \infty.
\label{eq:CycleComponentBound}
\end{align}
In other words, the expected number of edges outside the giant that are contained in cycle components is finite. Since~$CM_n(\mathbf{d},q)$ is created from $\Cmd$ by removing $R_n$ vertices of degree one, we observe that all isolated cycles that exist in $\Cmd$ remain isolated cycles in $CM_n(\mathbf{d},q)$. Moreover, more isolated cycles can be formed from more complex components in $\Cmd$. Using Proposition~\ref{prop:NumberOfLinesNodesOutsideGiant} and~\eqref{eq:CycleComponentBound} together with Markov's inequality, we observe that if $q=i/m$ with $\sqrt{m} \ll i \ll m^\delta$ for some $\delta \in (0,1/2)$, then the number of edges in $CM_n(\mathbf{d},q)$, outside the giant that are contained in cycle components is $\op(i^2/m)$. In view of Remark~\ref{rem:NonVanishingProbabilityInitiallyConnected}, the same statement holds for $\overline{CM}_n(\mathbf{d},q)$.

\subsection{Probabilistic bounds on component sizes outside the giant}\label{sec:LargeDeviationBoundsBeyondTypicalSizes}
In this section, we provide large deviation bounds on the number of edges outside the giant for the percolated connected configuration model with removal probability $q=i/m$ with $\sqrt{m} \ll i \ll m^{1-\delta}$ for some $\delta \in (0,1/2)$. Again, in view of Remarks~\ref{rem:NonVanishingProbabilityInitiallyConnected} and~\ref{rem:CanAlsoLookAtExplodedCMd}, it suffices to consider $\Cmd$ only.

First, we provide a sharper bound on the probability of having a number of edges in these complex components outside the giant that is even of a higher order of magnitude than $O(i^2/m)$.

\begin{proposition}\label{prop:OtherStuffDecay}
	Consider $\Cmd$ obtained by removal probability $q=i/m$ with $\sqrt{m} \ll i \ll m^{1-\delta}$ for some $\delta \in (0,1/2)$. For every $\alpha >0$,
	\begin{align}
	\Prob\left(\frac{m}{i^2} \#  \{ e \notin \Cmax\colon d_{\max}(\mathcal C(e)) \geq 3\}\geq m^{\alpha}\right) = O(m^{-3}) .
	\end{align}
	\label{prop:ExponentialDecayofOtherStuffOutsideGiant}
\end{proposition}

\begin{proof}
	We note  that, by the proof of \cite[Lemma 11]{MolRee95}, for $\gamma>0$ sufficiently large,
	\begin{align}
	\Prob(\exists\, \mathcal{C}\neq \Cmax : |E(\mathcal{C})|\geq \gamma \log n)=o(m^{-3}).
	\end{align}
	We are left to bound the contribution from components that contain at most $\gamma \log n$ edges. We use the method of moments. We observe that
	\begin{align*}
	(\#\{\mathcal C_l \ s.t. \ |E(\mathcal{C}_l)|\leq \gamma \log n, &d_{\max}(\mathcal C_l) \geq 3 \})^j \leq (\#\{v \in [n] : d_v \geq 3,|E( \mathcal{C}(v))|\leq \gamma\log n\})^j \\ 
	&= \sum_{v_1,...,v_j \in [n]} \indi_{\{|E(\mathcal{C}(v_1))|\leq \gamma \log n \colon d_{v_1}\geq 3 \}}  \cdots \indi_{\{|E(\mathcal{C}(v_j))|\leq \gamma \log n \colon d_{v_j}\geq 3 \}}.
	\end{align*}
	We stress that the vertices $v_1,...,v_j$ in the summation are not necessarily mutually distinct. For the purpose of exposition, write for a vertex $v \in [n]$,
	\begin{align*}
	\mathcal{I}(v)=\{ d_{v}\geq 3,|E(\mathcal{C}(v))| \leq \gamma \log n  \}.
	\end{align*}
	Applying the tower property, we obtain
	\begin{align*}
	\E&\left[(\#\{\mathcal C_l \ s.t. \ |E(\mathcal{C}_l)|\leq \gamma \log n, d_{\max}(\mathcal C_l) \geq 3 \})^j \right] \leq \E\left[\sum_{v_1,...,v_j \in [n]} \indi_{\mathcal{I}(v_1)}  \indi_{\mathcal{I}(v_2)} \cdots \indi_{\mathcal{I}(v_j)}\right] \\
	&\hspace{3cm}= \E \left[ \sum_{v_1,...,v_{j-1} \in [n]} \indi_{\mathcal{I}(v_1)}  \indi_{\mathcal{I}(v_2)} \cdots \indi_{\mathcal{I}(v_{j-1})}\E\left[ \sum_{v_j \in [n]} \indi_{\mathcal{I}(v_j)} \,\big\vert\, \mathcal{I}(v_1), ..., \mathcal{I}(v_{j-1}) \right] \right].
	\end{align*}
	For a graph (or component) $G$, write $V(G)$ as the vertex set of that graph. Then,
	\begin{align*}
	\E\left[ \sum_{v_j \in [n]} \indi_{\mathcal{I}(v_j)} \,\big\vert\, \mathcal{I}(v_1), \mathcal{I}(v_2), ..., \mathcal{I}(v_{j-1}) \right] &= \E\left[ \sum_{v_j \in [n]} \indi_{\mathcal{I}(v_j)} \colon v_j \in \bigcup_{h=1}^{j-1} V(\mathcal{C}(v_h)) \,\big\vert\, \mathcal{I}(v_1), \mathcal{I}(v_2), ..., \mathcal{I}(v_{j-1}) \right] \\
	& \hspace{0.3cm}+ \E\left[ \sum_{v_j \in [n]} \indi_{\mathcal{I}(v_j)} \colon v_j \not\in \bigcup_{h=1}^{j-1} V(\mathcal{C}(v_h)) \,\big\vert\, \mathcal{I}(v_1), \mathcal{I}(v_2), ..., \mathcal{I}(v_{j-1}) \right] .
	\end{align*}
	The first term is trivially bounded by
	\begin{align*}
	\E\left[ \sum_{v_j \in [n]} \indi_{\mathcal{I}(v_j)} \colon v_j \in \bigcup_{h=1}^{j-1} V(\mathcal{C}(v_h)) \,\big\vert\, \mathcal{I}(v_1), \mathcal{I}(v_2), ..., \mathcal{I}(v_{j-1}) \right] \leq (j-1) \gamma \log n.
	\end{align*}
	For the second term, we note that we count the number of vertices outside the giant component that have a degree at least three, while disregarding the set $\cup_{h=1}^{j-1} \mathcal{I}(v_h)$. We note that if we remove the components $\cup_{h=1}^{j-1} \mathcal{C}(v_h)$ from $\Cmd$, the remaining graph is a configuration model but with a modified degree sequence. In other words, we remove components that have a total of at most $O(\log n)$ edges. This number is too small to change the degree sequence $\mathbf{d}'$ much. That is, it follows from (the proof of) Proposition~\ref{prop:NumberOfOtherStuffOutsideGiant} that
	\begin{align*}
	\E\left[ \sum_{v_j \in [n]} \indi_{\mathcal{I}(v_j)} \colon v_j \not\in \bigcup_{h=1}^{j-1} V(\mathcal{C}(v_h)) \,\big\vert\, \mathcal{I}(v_1), \mathcal{I}(v_2), ..., \mathcal{I}(v_{j-1}) \right] = o\left( \frac{i^2}{m} \right).
	\end{align*}
	Iterating the argument, we obtain
	\begin{align*}
	\E&\left[(\#\{\mathcal C_l \ s.t. \ |E(\mathcal{C}_l)|\leq \gamma \log n, d_{\max}(\mathcal C_l) \geq 3 \})^j \right] = o\left( (j-1)\log n + o\left(\frac{i^2}{m}\right) \right)^j,
	\end{align*}
	and hence
	\begin{align*}
	&\E\big[(\# \{ e : |E(\mathcal C (e))|\leq \gamma \log n\colon d_{\max}(\mathcal C(e)) \geq 3\})^j\big]\leq \Big((j-1) (\gamma\log n)^2 +o\Big(\frac{i^2\log n}{m}\Big)\Big)^j.
	\end{align*}
	
	\noindent
	Finally, using Markov's inequality, it holds for every $j \in \mathbb{N}$,
	\begin{align*}
	\Prob\Big(&\frac{m\# \{ e : |E(\mathcal C(e))|\leq \gamma \log n\colon d_{\max}(\mathcal C(e)) \geq 3\}}{i^2}\geq m^{\alpha}\Big) \leq \Big((j-1)\gamma^2 (\log n)^2 +o\Big(\frac{i^2\log n}{m}\Big)\Big)^j\frac{m^{j(1-\alpha)}}{i^{2j}}\\
	&\hspace{7cm}= O\left(\Big((\log n)^2 + \frac{i^2\log n}{m }\Big)\frac{m^{(1-\alpha)}}{i^{2}}\right)^j=o\left((\log n)^2m^{-\alpha}\right)^j.
	\end{align*}
	Choosing $j \geq (3+\varepsilon)/\alpha$ for some $\varepsilon >0$, the claim follows.
\end{proof}

Next, we provide a result that shows that the probability for the number of edges in line components in $\Cmd$ to be of a higher order of magnitude than its expectation decays fast. 

\begin{lemma}\label{prop:LinesDecay} Consider  $\Cmd$ obtained by removal probability $q=i/m$, where $\sqrt{m} \ll i \ll m^{1-\delta}$ for some $\delta \in (0,1/2)$. For every $\alpha >0$,
	\begin{align}
	\Prob\Big(\frac{m}{i^2}\Big(\sum_{k=2}^\infty (k-1)L_k'(n)\Big) \geq m^\alpha\Big)= O(m^{-3}).
	\end{align}
\end{lemma}

\begin{proof}
	For every $j\in \mathbb{N}$, by Markov's inequality, 
	\begin{align*}
	\Prob\Big(\frac{m}{i^2}\Big(\sum_{k=2}^\infty (k-1)L_k'(n)\geq m^\alpha\Big)\leq \E\Big[\Big(\frac{m}{i^2}\Big(\sum_{k=2}^\infty kL_k'(n)\Big)^j\Big]m^{-\alpha j}.
	\end{align*}
	By Corollary~\ref{cor:HigherMoments},
	\begin{align*}
	\E\Big[\Big(\frac{m}{i^2}\Big(\sum_{k=2}^\infty kL_k'(n)\geq m^\alpha\Big)^j\Big]m^{-\alpha j}= O(m^{-\alpha j}).
	\end{align*}
	Choosing an integer $j\geq 3/\alpha$ concludes the result.
\end{proof}

\subsection{First disconnections}
\label{sec:FirstDisconnections}
In this section, we consider the question of how many edges need to be removed to cause the (initially connected) configuration model to become disconnected. That is, we would like to prove Theorem~\ref{thm:Firstdisconnection}, and show that the most likely moment for the first disconnection to occur is after $\Theta_{\sss \mathbb P} (\sqrt m)$ edges have been removed.

To prove Theorem~\ref{thm:Firstdisconnection}, we use the equivalence between sequential edge removal and percolation as established in Lemma~\ref{lem:CascadePercolation}. More specifically, in order to prove Theorem~\ref{thm:Firstdisconnection}, it follows from Lemma~\ref{lem:CascadePercolation} that it suffices to show that in a percolation process with a removal probability of the order $q=\Theta(m^{-1/2})$ leads to a positive probability for the configuration model to remain connected.

\begin{proposition}\label{prop:DisconnectivityPercolation}
	Consider the graphs $CM_n(\mathbf{d},q)$ and $\overline{CM}_n(\mathbf{d},q)$. If $q=c/\sqrt{m}$ with $c \in (0,\infty)$, then
	\begin{align}
	\Prob (CM_n(\mathbf{d},q) \text{ is connected} )\to\left(\frac{d-2p_2}{d}\right)^{1/2}\exp\Big\{-\frac{2 c^2p_2}{d-2p_2}\Big\},
	\end{align}
	and
	\begin{align}
	\Prob( \overline{CM}_n(\mathbf{d},q)\text{ is connected}) \to \exp\Big\{-\frac{2 c^2p_2}{d-2p_2}\Big\}.
	\end{align}
\end{proposition}

\begin{proof}
	We build $\CMD$ using the explosion algorithm from Algorithm~\ref{alg:ExplosionAlgorithm}. To obtain the limiting probability that $\CMD$ is connected, we first require more detailed information on the degree sequence of $\textbf{d}'$. Observe that 
	\begin{align*}
	q = \frac{c}{\sqrt{m}}  = \frac{c}{\sqrt{d/2}} \frac{1}{\sqrt{n}}.
	\end{align*}
	Recall~\eqref{eq:ZeroDegreeNodeSumBoundExplosionAlg},~\eqref{eq:OneDegreeNodeSumBoundExplosionAlg},~\eqref{eq:RnLLN} and~\eqref{eq:NljConvergences}, and hence
	\begin{align}
	n'_0&= \sum_{j=2}^{\infty} n_{j,0} = n_{2,0} +\op(1) \dto Poi \left(\frac{c^2p_2}{2d}\right),\label{eq:NewDegreeSequenceZero}\\ 
	n'_1& =R_n+ \sum_{j=2}^{\infty} n_{j,1}= \sqrt{n} \left( \frac{c\sqrt{d}}{\sqrt{2}}+ \frac{c p_2}{\sqrt{d/2}}\right)(1+\op (1)).\label{eq:NewDegreeSequenceOne} 
	\end{align}
	
	\noindent
	In addition, we observe that $n_l - R_n  \leq n_l' \leq n_l+R_n$ for every $l \geq 2$, and hence with high probability,
	\begin{align*}
	p_l'= \lim_{n\rightarrow \infty} n_l'/n = p_l, \hspace{1cm} l \geq 2. 
	\end{align*}
	Moreover, we observe that the average degree of $\mathbf{d}'$ converges in probability to $d$. 
	
	\noindent
	Secondly, we note that removing vertices of degree one cannot disconnect a component. Therefore,
	there are only two ways for $\CMD$ to be connected after the explosion algorithm. That is, either $\Cmd$ is already connected, or $\Cmd$ consists of one (giant) component and other components that are lines of length two (the only component entirely made of vertices of degree one), where all vertices of the line components are removed in the final step of the explosion algorithm. In either case, we observe that one requires that $n_0'=0$, which occurs with probability
	\begin{align*}
	\Prob\left( n_0'=0 \right) = \Prob\Big(Poi\Big(\frac{c^2p_2}{2d}\Big)=0\Big)(1+\op(1)) = \exp\left\{-\frac{c^2 p_2}{2d}\right\}(1+\op(1)) .
	\end{align*}
	Theorem 2.2 of~\cite{Federico2016} implies that if $n_0'=0$, then with high probability the graph $\Cmd$ consists of a giant component, possible components that are lines, possible components that are cycles, but no other type of components. Recall that $L_k'(n)$ represents the number of components that are lines of length $k \geq 2$ in $\Cmd$, and $C_k'(n)$ the number of components that are cycles of length $k \geq 1$. We call component with a single vertex of degree two with a self-loop a cycle component of length one, and a component with two vertices with multi-edges between them a cycle of length two. Theorem~2.2 of~\cite{Federico2016} implies that 
	\begin{align*}
	L_k'(n) &\dto \textrm{Poi}\left(c^2\left(\frac{d}{2}+p_2\right)^2 \frac{(2p_2)^{k-2}}{d^k} \right), \hspace{1cm} k \geq 2,\\
	C_k'(n) &\dto \textrm{Poi}\left(\frac{(2p_2)^k}{2kd^k} \right), \hspace{1cm} k \geq 1,
	\end{align*} 
	and all limits are independent random variables. For $\CMD$ to be connected after the explosion algorithm, no cycles should appear in $\Cmd$, which occurs with probability
	\begin{align*}
	\lim_{n \rightarrow \infty} \Prob&\left(\textrm{no cycle components in } \Cmd \right) = \prod_{k=1}^\infty \lim_{n \rightarrow \infty} \Prob\left(C_k'(n) = 0 \right) \\
	&= \prod_{k=1}^\infty  \exp\left\{ - \frac{(2p_2)^k}{2kd^k} \right\} = \left( \frac{d-2p_2}{d}\right)^{1/2}.
	\end{align*}
	Also, no lines of length more than three should appear, which occurs with probability
	\begin{align*}
	\lim_{n \rightarrow \infty} \Prob\left(L_k'(n) =0 \;\;\; \forall k \geq 3 \right) &= \prod_{k=3}^\infty \lim_{n \rightarrow \infty}  \Prob\left(L_k'(n) = 0 \right) \\
	&= \prod_{k=3}^\infty  \exp\left\{-c^2(\frac{d}{2}+p_2)^2 \frac{(2p_2)^{k-2}}{d^k}\right\} = \exp\left\{-\frac{c^2}{d^2}\left(\frac{d}{2}+p_2\right)^2 \frac{2p_2}{d-2p_2}\right\}.
	\end{align*}
	Finally, we require that all vertices in the line components are removed in the final step of the explosion algorithm. That is, a set of $2 L_2'(n)$ vertices need to be removed out of the $R_n$ randomly chosen vertices of degree one. Note that the number of vertices that are removed from line components in the last step of the explosion algorithm is hypergeometrically distributed: we remove $R_n$ vertices uniformly at random from $n_1'$ vertices of which $2 L_2'(n)$ are in line components. Therefore, conditionally on $n_1'$, $R_n$ and $ L_2'(n)$, the probability of all vertices to be removed in the final step of the explosion algorithm is given by
	\begin{align*}
	\binom{n_1' - 2 L_2'(n)}{R_n-2 L_2'(n)} \Big/ \binom{n_1'}{R_n} &= \frac{R_n (R_n -1)\cdots (R_n-2 L_2'(n)+1)}{n_1'(n_1'-1) \cdots (n_1'-2 L_2'(n)+1)}
	\end{align*}
	In particular, if $L_2'(n)=0$, then with probability one all vertices in line components are removed in the last step of the algorithm. Using the tower property,~\eqref{eq:RnLLN},~\eqref{eq:NewDegreeSequenceOne} and the observation that $L_2'(n)$ converges to a Poisson distribution, we observe that
	\begin{align*}
	\Prob &\left( \textrm{all line components in } \Cmd \textrm{ are removed in } \CMD \right) \\
	&= \E\left[\E\left[\indi_{\{ \textrm{all line components in } \Cmd \textrm{ are removed in } \CMD\}} \mid n_1' , R_n, L_2'(n) \right]  \right]\\
	&= \sum_{k=0}^\infty \Prob(L_2'(n)=k) \E\left[ \binom{n_1' - 2 k}{R_n-2 k} \Big/ \binom{n_1'}{R_n} \mid n_1' , R_n, L_2'(n) = k\right] \\
	&= \sum_{k=0}^\infty \frac{1}{k!} \left(\frac{c^2}{d^2} \left( \frac{d}{2} +p_2\right)^2\right)^{k} \exp\left\{- \frac{c^2}{d^2} \left( \frac{d}{2} +p_2\right)^2 \right\} \left(\frac{d}{d+2p_2} \right)^{2k}(1+o(1))\\
	&=  \exp\left\{ -\frac{c^2}{d^2} \left( \frac{d}{2} +p_2\right)^2 \left(1 - \left(\frac{d}{d+2 p_2} \right)^2\right)\right\} (1+o(1)) = \exp\left\{ -\frac{c^2 p_2(d+p_2)}{d^2}\right\} (1+o(1)).
	\end{align*}
	We conclude that
	\begin{align*}
	\Prob &(\CMD \text{ is connected}) =  (1+o(1)) \sqrt{ \frac{d-2p_2}{d}} \exp\left\{-\frac{c^2 p_2}{2d}-\frac{c^2}{d^2}\left(\frac{d}{2}+p_2\right)^2 \frac{2p_2}{d-2p_2} -\frac{c^2 p_2(d+p_2)}{d^2}\right\}\\
	&= \left( \frac{d-2p_2}{d}\right)^{1/2} \exp\left\{-\frac{2 c^2 p_2}{d-2p_2}\right\}(1+o(1)).
	\end{align*}
	For the second statement of the proposition, it is known from~\cite{Luc92} that
	\begin{align*}
	\Prob (\CMd \text{ is connected})\to \left(\frac{d-2p_2}{d}\right)^{1/2},
	\end{align*}
	and since $\{ CM_n(\mathbf{d},q)\text{ is connected}\} \subseteq \{ \CMd\text{ is connected}\}$, the statement follows.
\end{proof}

Using the partial result of Proposition~\ref{prop:DisconnectivityPercolation}, we can prove Theorem~\ref{thm:Firstdisconnection}.

\Firstdisconnection*

\begin{proof}
	First, connectivity is a monotone property, and thus, once the graph is disconnected it will stay disconnected for the rest of the process. From Lemma~\ref{lem:CascadePercolation} it follows that sequential removal of $c \sqrt{m}$ edges in $\overline{CM}_n(\textbf{d},q)$ is well-approximated by a percolation process with removal probability $q=cm^{-1/2} (1+\op(1))$. Consequently, due to Proposition \ref{prop:DisconnectivityPercolation},
	\begin{align}
	\Prob (T_{n,\boldsymbol d} \geq x\sqrt{m}) = \Prob \left(CM_n(\mathbf{d},x/\sqrt{m}) \text{ is connected}\right) + o(1)  \to \exp\Big\{-\frac{2x^2p_2}{d-2p_2}\Big\}.
	\label{eq:ToReferToTnd}
	\end{align}
	In other words, $m^{-1/2} T_{n,\boldsymbol d}$ converges in distribution to a random variable $T$ whose complementary cumulative distribution function is given by the expression on the right-hand side of~\eqref{eq:ToReferToTnd}. 
\end{proof}

Proposition~\ref{prop:DisconnectivityPercolation} implies that the percolated graph $\overline{CM}_n(\mathbf{d},q)$ with removal probability~$q=o(m^{-1/2})$ is disconnected with probability of order $o(1)$. More detailed bounds can be provided for this range, as is illustrated by the next result.

\begin{proposition}\label{propo:DisconnectivityEarly}
	Consider $\overline{CM}_n(\mathbf{d},q)$ with $q=O(m^{-\alpha})$ for some $1/2 < \alpha < 1$. Then,
	\begin{align}
	\Prob( \overline{CM}_n(\mathbf{d},q)\text{ is disconnected}) =O(m^{1-2\alpha}).
	\end{align}
	Consequently, if we remove $i:=i_m$ edges uniformly at random from $\overline{CM}_n(\mathbf{d})$ with $i = o(m^{\beta})$ for some $\beta \in (0,1/2)$, then the probability of the resulting graph being disconnected is of order $O(m^{2\beta-1})$.
\end{proposition}

\begin{proof}
	If $\overline{CM}_n(\mathbf{d},q)$ is disconnected, then there is at least one component outside the giant that is either an isolated node, a line, a cycle or a component that contains a vertex with degree at least three. To show the result, it therefore suffices to prove that each of these events occur with probability $O(m^{1-2\alpha})$.
	
	First, it follows from Proposition \ref{prop:NumberOfLinesNodesOutsideGiant} that in $\overline{CM}_n(\mathbf{d},q)$,
	\begin{align*}
	\Prob\Big(N_0(n)+\sum_{k=2}^\infty kL_k(n) \neq 0\Big)= O(m^{1-2\alpha}).
	\end{align*}
	In other words, the event that there is an isolated node or a line component (outside the giant) occurs with probability $O(m^{1-2\alpha})$. Moreover, we can apply Proposition \ref{prop:NumberOfOtherStuffOutsideGiant} to obtain that in $\overline{CM}_n(\mathbf{d},q)$,
	\begin{align*}
	\Prob( \exists\, \mathcal{C} \neq \Cmax\colon  d_{\max}( \mathcal{C}) \geq 3\}=o(m^{1-2\alpha}).
	\end{align*}
	Finally, we show that with probability $1-O(m^{1-2\alpha})$ percolation on $\overline{CM}_n(\mathbf{d},q)$ does not create cycles. We observe that
	\begin{align*}
	\Prob\left( \sum_{k=1}^\infty k C_k(n) \neq 0 \,\big\vert\, \CMd \textrm{ is connected} \right)  &= \frac{\Prob\left( \sum_{k=1}^\infty k \tilde{C}_k(n) \neq 0 \colon \CMd \textrm{ is connected} \right)}{\Prob\left(\CMd \textrm{ is connected} \right)} \\
	&\leq \frac{\Prob\left( \sum_{k=1}^\infty k \tilde{C}_k(n) \neq 0 \right)}{\Prob\left(\CMd \textrm{ is connected} \right)},
	\end{align*}
	where $\tilde{C}_k(n)$ denotes the number of newly formed isolated cycles in $CM_n(\mathbf{d},q)$ that do not exist in the initial graph $\CMd$. Since $\CMd$ is connected with non-vanishing probability, it is sufficient to show that
	\begin{align*}
	\Prob\left( \sum_{k=1}^\infty k \tilde{C}_k(n) \neq 0 \right) = O(m^{1-2\alpha}).
	\end{align*}
	Again, we use the explosion method. We stress that running this algorithm on a sampled $\CMd$ results in a graph that is not the same as the graph that would have been obtained if percolation had been applied on the same sampled $\CMd$. Instead, sampling a $\CMd$ and running the explosion algorithm results in a graph that is statistically indistinguishable from one that is obtained by sampling another $\CMd$ and applying percolation on it. Therefore, we need to consider the question how to bound newly formed cycles in $CM_n(\mathbf{d},q)$. The crucial observation to answer this question is that such cycles contain at least one node whose degree in $\mathbf{d}$ was at least three.
	
	More specifically, the number of newly formed isolated cycles in $CM_n(\mathbf{d},q)$ is at most the number of cycles in $CM_n(\mathbf{d},q)$ that contain at least one node whose degree in $\mathbf{d}$ is at least three. Write $Z_n$ for the number of vertices whose degree in $\mathbf{d}'$ is two, but has degree at least three in $d$, i.e.
	\begin{align*}
	Z_n = \sum_{l=3}^\infty n_{l,2}.
	\end{align*}
	Every isolated cycle in $CM_n(\mathbf{d},q)$ that contains at least one node having an original degree at least three in $\mathbf{d}$, either that cycle already exists in $CM_N(\mathbf{d}')$, or it was formed from a component in $CM_N(\mathbf{d}')$ by the removal of degree one vertices. In the second case, this implies that there exists a component in $CM_N(\mathbf{d}')$ outside the giant with a maximum degree at least three, which occurs with probability $o(m^{1-2\alpha})$ by Proposition~\ref{prop:NumberOfOtherStuffOutsideGiant}. What remains is to analyze the first case.
	
	Write $\tilde{C}_k'(n)$ for the number of cycles that exist in $CM_N(\mathbf{d}')$ and contain at least one node whose original degree in $\mathbf{d}$ is at least three. For any $k \geq 1$, a set of (distinct) degree-two vertices $\{v_1,...,v_k\}$ in $\mathbf{d}'$ forms a cycle in $CM_N(\mathbf{d}')$ with probability
	\begin{align*}
	\frac{2(k-1)}{2m-1} \frac{2(k-2)}{2m-3} \cdots \frac{2}{2m-2k+3} \frac{1}{2m-2k+1}
	\end{align*}
	if $k\geq 2$, and $1/(2m-1)$ if $k=1$. The number of sets of $k$ vertices in $\mathbf{d}'$ that all have degree two of which at least one vertex has degree at least three in $\mathbf{d}$ is bounded by $Z_n \binom{n_2'}{k-1}$. Therefore,
	\begin{align*}
	\E[\tilde{C}_k'(n)] \leq \E\left[ Z_n \binom{n_2'}{k-1} \frac{(2(k-1))!!(2m-2k-1)!!}{(2m-1)!!}\right].
	\end{align*}
	We observe that $n_2'\leq n_2+R_n$ and $Z_n \leq R_n$, where $R_n$ is a binomially distributed random variable  with parameters $2m$ and $1-\sqrt{1-q}=i/(2m)(1+o(1))$. Therefore, for every $\epsilon >0$, the probability that $R_n \geq (1+\epsilon)i$ occurs has an exponentially decaying tail. On the other hand, since $n_2/n \rightarrow p_2$, it holds for every $\epsilon>0$ and all $n$ sufficiently large,
	\begin{align*}
	\E[\tilde{C}_k'(n) \mid R_n \leq (1+\epsilon)i ] &\leq  i(1+\epsilon) \binom{n_2+(1+\epsilon)i}{k-1} \frac{(2(k-1))!!(2m-2k-1)!!}{(2m-1)!!} = (1+\epsilon)\frac{i}{2m} \left(\frac{2 p_2+\epsilon}{d-\epsilon}\right)^{k-1}
	\end{align*}
	and hence this sequence converges to zero exponentially fast in $k$. Applying dominated convergence, we obtain
	\begin{align*}
	\E\left[ \sum_{k=1}^\infty k \tilde{C}_k'(n) \mid R_n \leq (1+\epsilon)i\right] = O\left(\frac{i}{m} \right) = o(m^{1-2\alpha}).
	\end{align*}
	By Markov's inequality, we conclude that
	\begin{align*}
	\Prob\left( \sum_{k=1}^\infty k \tilde{C}_k(n) \neq 0\right) &= \E\left( \sum_{k=1}^\infty k \tilde{C}_k'(n) \neq 0 \mid R_n \leq (1+\epsilon)i \right) + o(m^{1-2\alpha}) = o(m^{1-2\alpha}).
	\end{align*}
	This proves the first statement of the proposition.
	
	To prove the second statement of the proposition, we need to relate the percolation process to the process of removing edges uniformly at random from $\overline{CM}_n(\mathbf{d})$. To each edge, attach a uniformly distributed random variable. An alternative description of the percolation process with removal probability $q$ is by removing all edges whose values of the random variable are below $q$. Let $U_{(1)}^m \leq ... \leq U_{(m)}^m$ denote $m$ uniformly distributed order statistics. Then, the probability that $i$ edge removals disconnects $\overline{CM}_n(\mathbf{d})$ is given by $\Prob( \overline{CM}_n(\mathbf{d}, U_{(i)}^m ) \textrm{ is disconnected})$. We note that
	\begin{align*}
	\Prob&\left( \overline{CM}_n\left(\mathbf{d}, U_{(i)}^m \right) \textrm{ is disconnected}\right) \leq \Prob\left( \overline{CM}_n\left(\mathbf{d}, m^{\beta-1} \right) \textrm{ is disconnected}\right) + \Prob\left( U_{(i)}^m > m^{\beta-1} \right).
	\end{align*}
	The above proof shows that
	\begin{align*}
	\Prob\left( \overline{CM}_n\left(\mathbf{d}, m^{\beta-1} \right) \textrm{ is disconnected}\right) = O(m^{2\beta-1}).
	\end{align*}
	The second term is negligible, since by the Chernoff bound,
	\begin{align*}
	\Prob\left( U_{(i)}^m > m^{\beta-1} \right) \leq (1+o(1)) \exp\left\{ -\frac{m^\beta}{2}\right\}  = O(m^{2\beta-1}).
	\end{align*}
\end{proof}

\subsection{Number of edges outside the giant component}\label{sec:NumberOfEdgesOutsideTheGiant}
First, we prove Theorem \ref{thm:FailureSublinear} putting together the results proved in the Section~\ref{sec:TypicalStructuresPercolation}.

\FailureSublinear*

\begin{proof}[Proof of Theorem \ref{thm:FailureSublinear}] By Lemma \ref{lem:CascadePercolation}, the number of edges outside the largest component in $\overline{CM}_n(\mathbf{d})$ after $i$ failures can be derived by considering percolation on this graph with removal probability $q=i/m$. The edges outside the giant can be divided into edges in line components, edges in cyclic components and edges in more complex components (i.e. components which contain a vertex of degree at least three).
	From Proposition~\ref{prop:NumberOfLinesNodesOutsideGiant} it follows that 
	\begin{align*}
	\frac{m}{i^2}\Big(\sum_{k=2}^\infty (k-1)L_k(n)\Big)\pto \frac{ 4 p_2^2}{(d-2p_2)^2}.
	\end{align*}
	
	Next we bound the number of edges outside the giant that are contained in components that are cycles or contained in more complex components. In view of Remarks~\ref{rem:NonVanishingProbabilityInitiallyConnected} and~\ref{rem:CanAlsoLookAtExplodedCMd}, we point out that this is bounded by the same quantity in $\Cmd$. Therefore, to conclude the proof, it suffices to show that the total number of edges contained in cycle components or more complex components outside the giant is of order $\op(i^2/m)$ for $\Cmd$. 
	
	By Proposition~\ref{prop:NumberOfOtherStuffOutsideGiant}, the number of edges in $\Cmd$ outside the giant that is contained in a more complex component is of order $\op(i^2/m)$. To count the number of edges in cycle components, recall that $C'_k(n)$ denotes the number of  cyclic components with $k$ edges in $\Cmd$, and that it satisfies~\eqref{eq:CycleComponentBound}, i.e. $\lim_{n \to \infty} \E[\sum_{k \geq 1}kC_k'(n)]<\infty$. Using Markov's inequality, this implies that
	\begin{align*}
	\frac{m}{i^2} \sum_{k \geq 1}kC_k'(n) = \op(1).
	\end{align*}
	Consequently, the number of edges in $\Cmd$ outside the giant that are not contained in line components is indeed of order $\op (i^2/m)$. 
\end{proof}

Theorem~\ref{thm:FailureSublinear} prescribes the likely number of edges outside the giant component as an effect of sequentially removing edges uniformly at random. The initially connected configuration model is likely to disintegrate, as more edges are removed, in a unique giant and several small components, the majority of which are either lines or isolated nodes as long as the number of edge failures is sublinear. 

Next, we show that the number of edges outside the giant component in $\CMD$ is unlikely to be of a higher order of magnitude than its typical value during the cascade. We stress that this results concerns the percolation process, which can in turn be used to derive a large deviations bound for the sequential removal process.

\begin{theorem}\label{thm:DeviationEdgesOutsideGiant}
	Consider $CM_n(\mathbf{d},q)$ and $\overline{CM}_n(\mathbf{d},q)$ with $q=i/m$, with $\sqrt{m}\ll i\ll m^{1-\delta}$ for some $\delta>0$. Then, for every $\alpha>0$,
	\begin{align}
	\Prob\Big(|E(\CMD \setminus \Cmax)| \frac{m}{i^2}\geq m^\alpha\Big)=O(m^{-3}), 
	\label{eq:OutsideGiantDeviationsNotConnectedCM}
	\end{align}
	and
	\begin{align}
	\Prob\Big(|E(\overline{CM}_n(\mathbf{d},q) \setminus \Cmax)| \frac{m}{i^2}\geq m^\alpha\Big)=O(m^{-3}).
	\label{eq:OutsideGiantDeviationsConnectedCM}
	\end{align}
\end{theorem}

\begin{proof}
	First, we show~\eqref{eq:OutsideGiantDeviationsNotConnectedCM} by using the explosion algorithm. Again, in view of Remarks~\ref{rem:NonVanishingProbabilityInitiallyConnected} and~\ref{rem:CanAlsoLookAtExplodedCMd}, it suffices to show
	\begin{align}\label{eq:IntermediateStuffOutsideGiant}
	\Prob\Big(|E(\Cmd \setminus \Cmax)| \frac{m}{i^2}\geq m^\alpha\Big)=O(m^{-3}). 
	\end{align}
	We partition the different kinds of contributions to the total number of edges outside the giant of $\Cmd$ as follows: edges can be contained in a line, a cycle or a more complex component. Due to Proposition~\ref{prop:LinesDecay}, the number of edges in line components in $\Cmd$ is bounded by 
	\begin{align}
	\Prob\Big(\sum_{k=2}^\infty (k-1)L_k'(n)\frac{m}{i^2}\geq m^\alpha\Big)=O(m^{-3}).
	\end{align}
	To bound the edges in cycles, we point out that if follows from~\cite[(3.7)]{Federico2016} that all moments of the number of cycles in $\Cmd$ converge to a finite constant. That is, convergence of the first $j$ factorial moments is shown, which implies convergence of the first $j$ moments. Again, this result is shown under the condition that the number of vertices of degree one is of order~$O(\sqrt m)$ in~\cite{Federico2016}, but this assumption is not used for the results with respect to the isolated cycles. Consequently, for every~$j$, as long as $p_2\in (0,1)$,
	\begin{align*}
	\lim_{n\to \infty}	\E\Big[\Big(\sum_{k=1}^\infty kC'_k(n)\Big)^j\Big]\leq \E\Big[\Big( \sum_{k=1}^\infty Poi \Big(\frac{(2p_2)^k}{2kd^k}\Big) k \Big)^j\Big]<\infty,
	\end{align*}
	where the Poisson variables in the second term are all independent. Therefore, for every $\alpha >0$, by applying Markov's inequality, we obtain
	\begin{align*}
	\Prob\Big(\frac{m}{i^2}\sum_{k=1}^\infty kC_k'(n)\geq m^\alpha\Big) &\leq \Prob\Big(\sum_{k=1}^\infty kC_k'(n)\geq m^\alpha\Big) \leq \E\Big[\Big(\sum_{k=1}^\infty kC_k'(n)\Big)^{3/\alpha}\Big]m^{-3}=O(m^{-3}).
	\end{align*}
	
	\noindent
	To bound the number of edges in other components we use Proposition~\ref{prop:OtherStuffDecay} to obtain 
	\begin{align*}
	\Prob\Big(\frac{m\# \{ e \notin \Cmax\colon d_{\max}(\mathcal C(e)) \geq 3\}}{i^2}\geq m^{\alpha}\Big) = O(m^{-3}).
	\end{align*}
	Thus we obtain~\eqref{eq:IntermediateStuffOutsideGiant} by summing the three different contributions. 
\end{proof}

Theorem~\ref{thm:DeviationEdgesOutsideGiant} reveals that it is very unlikely for the number of edges outside the giant to be larger than of order $\Theta(i^2/m)$ when applying the percolation process. We can use this result to show that this also holds under a sequential removal process. That is, we can prove Theorem~\ref{thm:DeviationsOutsideGiant}, which we recall next.

\DeviationsOutsideGiant*

\begin{proof}[Proof of Theorem~\ref{thm:DeviationsOutsideGiant}]
	First, we prove~\eqref{eq:FirstClaimCorollaryEdgesOutsideGiant} for $i$ such that $m^{1/2+\epsilon} \ll i \ll m^\alpha$ for some $\epsilon \in (0,\alpha-1/2)$ sufficiently small. We use this partial result to show~\eqref{eq:SecondClaimCorollaryEdgesOutsideGiant}. In the proof of~\eqref{eq:SecondClaimCorollaryEdgesOutsideGiant}, it is implied that~\eqref{eq:FirstClaimCorollaryEdgesOutsideGiant} also holds for all $i \ll m^{1/2+\epsilon}$ for some $\epsilon \in (0,\alpha-1/2)$ sufficiently small, concluding the proof of Theorem~\ref{thm:DeviationEdgesOutsideGiant} follows.
	
	Let $U_{(1)}^m \leq ... \leq U_{(m)}^m$ denote $m$ uniformly distributed order statistics, and note that
	\begin{align*}
	\Prob\left( |\hat{E}_m(i)| \leq m-i-i^\alpha \right) = \Prob\left( \big|E(\overline{CM}_n(\textbf{d},U_{(i)}^m)\backslash \Cmax )\big| > i^\alpha \right).
	\end{align*}
	Let $\epsilon>0$ (sufficiently) small, and note that by the Chernoff bound,
	\begin{align*}
	\Prob\left(U_{(i)}^m  \not\in   \left[ \frac{i}{m} m^{-\epsilon},\frac{i}{m} m^{\epsilon}\right] \right) = O(m^{-3}).
	\end{align*}
	Therefore,
	\begin{align*}
	\Prob&\left(E( \big|\overline{CM}_n(\textbf{d},U_{(i)}^m)\backslash \Cmax) \big| > i^\alpha \right) \leq \max_{q \in \left[i m^{-(1+\epsilon)},i m^{-(1-\epsilon)}\right] } \Prob\left( \big|E(\overline{CM}_n(\textbf{d},q)\backslash \Cmax )\big| > i^\alpha \right) + O(m^{-3}).
	\end{align*}
	We observe that for every $q \in \left[i m^{-(1+\epsilon)},i m^{-(1-\epsilon)}\right]$ with $i = o(m^{\alpha})$,
	\begin{align*}
	\frac{q^2 m}{i^\alpha} \leq \frac{i^{2-\alpha}}{m^{1-2\epsilon}} = o( m^{-(1-\alpha)^2+2\epsilon}) = o(1)
	\end{align*}
	for all $\epsilon >0$ sufficiently small. By Theorem~\ref{thm:DeviationEdgesOutsideGiant}, we conclude that
	\begin{align*}
	\Prob&\left( |\hat{E}_m(i)| \leq m-i-i^\alpha \right) \leq \max_{q \in \left[i m^{-(1+\epsilon)},i m^{-(1-\epsilon)}\right]} \Prob\left( \big|E(\overline{CM}_n(\textbf{d},q)\backslash \Cmax) \big| > q^2/m \right) + O(m^{-3}) =O(m^{-3}).
	\end{align*}
	
	To prove~\eqref{eq:SecondClaimCorollaryEdgesOutsideGiant}, note that by Proposition~\ref{propo:DisconnectivityEarly},
	\begin{align*}
	\Prob&\left( m-i - |\hat{E}_m(i)| > i^\alpha \textrm{ for some } 1 \leq i \leq m^{1/8} \right) \\ 
	&\hspace{2cm} \leq \Prob\left( m-i - |\hat{E}_m(i)| \neq 0 \textrm{ for some } 1 \leq i \leq m^{1/8} \right) = o(m^{-1/2}).
	\end{align*}
	Moreover, if $k \gg m^{1/2+\epsilon}$ for some $\epsilon \in (0,1/2)$, using that $k \leq m$, it follows directly from~\eqref{eq:FirstClaimCorollaryEdgesOutsideGiant},
	\begin{align*}
	\Prob&\left( m-i - |\hat{E}_m(i)| > i^\alpha \textrm{ for some } m^{1/2+\epsilon} \leq i \leq k \right) = O(k m^{-3}) = o(m^{-1/2}).
	\end{align*}
	Therefore, to conclude~\eqref{eq:SecondClaimCorollaryEdgesOutsideGiant}, it suffices to show that for some $\epsilon \in (0,1/2)$ sufficiently small,
	\begin{align*}
	\Prob&\left( m-i - |\hat{E}_m(i)| > i^{1/8} \textrm{ for some } m^{1/8} \leq i \leq m^{1/2+\epsilon} \right) = o(m^{-1/2}).
	\end{align*}
	For convenience, write $|\tilde{E}_m(i)| = m-i - |\hat{E}_m(i)| $, and consider values of $i$ such that $m^{1/8} \leq i \leq m^{1/2+\epsilon}$ for some $\epsilon \in (0,1/2)$. Note that
	\begin{align*}
	\Prob&\left(|\tilde{E}_m(i)| > i^{1/8}\right) = \Prob\left(|\tilde{E}_m(i)| > i^{1/8} , |\tilde{E}_m(m^{1/2+\epsilon})| < |\tilde{E}_m(i)|/2 \right) \\
	&\hspace{3.5cm}+ \Prob\left(|\tilde{E}_m(i)| > i^{1/8} , |\tilde{E}_m(m^{1/2+\epsilon})| \geq |\tilde{E}_m(i)|/2 \right).
	\end{align*}
	We observe that $|\tilde{E}_m(i)|$ is the number of edges outside the giant when removing $i$ edges uniformly at random. Write $\xi(i,j)$ as the number of edges that are removed out of this set of $|\tilde{E}_m(i)|$ edges if we remove another $j-i$ edges uniformly at random. Then, $|\tilde{E}_m(m^{1/2+\epsilon})|\geq |\tilde{E}_m(i)|-\xi(i,m^{1/2+\epsilon})$, and hence
	\begin{align*}
	\Prob&\left(|\tilde{E}_m(i)| > i^{1/8} , |\tilde{E}_m(m^{1/2+\epsilon})| < |\tilde{E}_m(i)|/2 \right)  \leq \Prob\left(|\tilde{E}_m(i)| > i^{1/8} ,  \xi(i,m^{1/2+\epsilon}) > |\tilde{E}_m(i)|/2 \right).
	\end{align*}
	Since $\xi(i,j) \leq j$ for every $j \geq i$, it follows from the first claim of the corollary that with probability $1-O(m^{-3})$,
	\begin{align*}
	|\tilde{E}_m(i)| \leq |\tilde{E}_m(m^{1/2+\epsilon})| + \xi(i,m^{1/2+\epsilon}) = o(m).
	\end{align*}
	In other words, the probability that an edge out of the set of $|\tilde{E}_m(i)|$ is chosen to be removed is strictly bounded by $|\tilde{E}_m(i)|/(m-m^{1/2+\epsilon}) =o(1)$ with probability  $1-O(m^{-3})$. In that case, the probability for more than half of the $|\tilde{E}_m(i)| > i^{1/8}$ edges to be removed has an exponentially decaying tail. In particular, this implies that
	\begin{align*}
	\Prob&\left(|\tilde{E}_m(i)| > i^{1/8} , |\tilde{E}_m(m^{1/2+\epsilon})| < |\tilde{E}_m(i)|/2 \right)  \leq \Prob\left(|\tilde{E}_m(i)| > i^{1/8} ,  \xi(i,m^{1/2+\epsilon}) > |\tilde{E}_m(i)|/2 \right) = O(m^{-3}).
	\end{align*}
	For the other term, we observe that for every $m^{1/8} \leq i \leq m^{1/2+\epsilon}$,
	\begin{align*}
	\Prob\left(|\tilde{E}_m(i)| > i^{1/8} , |\tilde{E}_m(m^{1/2+\epsilon})| \geq \frac{|\tilde{E}_m(i)|}{2} \right) &\leq \Prob\left( |\tilde{E}_m(m^{1/2+\epsilon})| > \frac{i^{1/8}}{2} \right) \leq \Prob\left( |\tilde{E}_m(m^{1/2+\epsilon})| > \frac{m^{1/64}}{2} \right).
	\end{align*}
	Similarly as in the proof of the first claim of the corollary, we observe that
	\begin{align*}
	\Prob\left(U_{(m^{1/2+\epsilon})}^m \not\in \left[  m^{-1/2-2\epsilon},  m^{-1/2+2\epsilon} \right] \right) = O(m^{-3}),
	\end{align*}
	and hence
	\begin{align*}
	\Prob&\left( |\tilde{E}_m(m^{1/2+\epsilon})| > 2m^{1/64} \right) \leq \Prob\left(U_{(m^{1/2+\epsilon})}^m \not\in \left[  m^{-1/2-2\epsilon},  m^{-1/2+2\epsilon} \right] \right) \\
	&\hspace{3cm} + \max_{q \in [ m^{-1/2-2\epsilon},  m^{-1/2+2\epsilon}]} \Prob\left( \big|E(\overline{CM}_n(\textbf{d},q)\backslash \Cmax) \big| > 2m^{1/64} \right).
	\end{align*}
	It follows from Theorem~\ref{thm:DeviationEdgesOutsideGiant} that the second term is also of order $O(m^{-3})$ for every $\epsilon< 1/256$. We conclude that for every $m^{1/8} \leq i \leq m^{1/2+\epsilon}$ with $\epsilon < 1/256$,
	\begin{align*}
	\Prob\left(|\tilde{E}_m(i)| > i^{1/8}\right) = O(m^{-3}),
	\end{align*}
	from which~\eqref{eq:SecondClaimCorollaryEdgesOutsideGiant} follows by the union bound. Moreover, it implies that~\eqref{eq:FirstClaimCorollaryEdgesOutsideGiant} holds for all $\sqrt{m} \ll i \ll m^{1/2+\epsilon}$ for some $\epsilon \in (0,\alpha-1/2)$ as well.
\end{proof}

\subsection{Linear number of edge removals}\label{sec:LinearEdgeFailures}
For completeness, we provide a brief overview of known results about what $\CMD$ looks like when the removal probability is a fixed value $q \in (0,q_c)$, as studied in~\cite{Jan09b}. It is shown that for any fixed $q \in (0, q_c)$, $\CMD$ has a unique giant component and many small components. However, in this phase the giant does not contain $n-o(n)$ vertices, as in the case when $q\to 0$. 

From \cite{Jan09b}, it is known that there exists a function $\xi_{\mathbf{d}} (q)$ defined for $q < q_c$ such that in $CM_n(\mathbf{d},q)$ 

\begin{align}\label{eq:EdgesGiant}
\frac{|E(\Cmax)|}{(1-q)m}\pto \xi_{\mathbf{d}} (q),
\end{align}
i.e., the proportion of edges in the giant component concentrates for every $q<q_c$. The exact formula for $\xi_{\mathbf{d}} (q)$ comes from \cite[Theorem 3.9]{Jan09b},
\begin{align}
\xi_{\mathbf{d}} (q)=\frac{1-\rho}{\sqrt{1-q}}+\frac{(1-\rho)^2}{2},
\end{align}
where $\rho$ is defined as the solution of the equation

\begin{align}
(1-q)^{1/2}G'_D\big(1-(1-q)^{1/2}-\rho(1-q)^{1/2}\big)+(1-(1-q)^{1/2})d=\rho d,
\end{align}  
where $G_D$ is the probability generating function of $D$. The same concentration result, with a different limit function, holds also for the number of vertices in the largest component.

It is worth mentioning that in this case, lines and isolated vertices do not constitute the vast majority of the small components anymore, but there is also a positive density of more complex small components to appear. These are mainly trees.

If $q>q_c$, instead, there is no giant component that contains a non-vanishing proportion of the edges or vertices and usually the high-degree vertices determine the size of the largest components, i.e.~$|\Cmax|=\Theta_\Prob (d_{max})$~\cite{Jan08}.

\section{Cascading failure process}\label{sec:CascadingFailureProcess}
The results in the previous section explain the way the connected configuration model is likely to disintegrate as edge failures occur sequentially and uniformly at random. This yields the load surge values, and hence what remains to be done for proving Thoerem~\ref{thm:MainResultSublinear}, is to compare the load surge values to the corresponding surplus capacities at the edges.

To prove our main results as stated in Theorem~\ref{thm:MainResultSublinear}, we follow the proof strategy as laid out in Section~\ref{sec:RoadMapCascFailureProcess}. Recall that the connected disintegrates in a giant component and a sublinear number of edges outside the giant as long as no more than $o(m)$ edges are removed. We point out that, intuitively, this implies that the only dominant contribution to the total failure size comes from the number of edges that are contained in the giant component upon failure. We make this statement rigorous in this section.

This section is structured as follows. In Section~\ref{sec:CascadingNoEdgeDisconnections}, we consider the setting where no disconnections takes place yet. Since it follows from Theorem~\ref{thm:Firstdisconnection} that it is unlikely that the connected configuration model becomes disconnected before $\Theta(\sqrt{m})$, we can show that Theorem~\ref{thm:FailureSublinear} holds if $k=o(\sqrt{m})$. For larger thresholds, we consider the failure size tail of the giant component. To derive this tail behavior, we translate this problem to a first-passage time problem over a random moving boundary in Section~\ref{sec:RWFormulation}. In Section~\ref{sec:AsympBehvGiant}, we derive the behavior of this first-passage time. We conclude Theorem~\ref{thm:FailureSublinear} in Section~\ref{sec:ProofOfMainResultRG} by using the strategy as laid out in Section~\ref{sec:RoadMapCascFailureProcess}.

\subsection{No edge disconnections}\label{sec:CascadingNoEdgeDisconnections}
Before we move to the tail of the failure size, we first consider the scenario where no disconnections have occurred yet during the cascade, or alternatively, the setting where the failure mechanism is applied to a graph with a star topology. As long as edge failures do not cause (edge) disconnections in the graph, the load surge function remains the same at every surviving edge. Recall that in this case it holds that $|E_j^m(i)|=m-i$ for all surviving edges $e_j \in [m]$, and hence recursion~\eqref{eq:LoadSurgeRecursion} is solved by
\begin{align}\label{eq:LoadSurgeStarTopology}
l_j^m(i) = \frac{\theta}{m}+(i-1) \cdot \frac{1-\theta/m}{m}.
\end{align}

\noindent
In other words, given that no disconnections have occurred after $k$ edge failures, the load surge function behaves deterministically until step $k$ (at every surviving edge). In this case, the problem reduces to a first-passage time problem, i.e.~the event $\{A_{n,\mathbf{d}} \geq k \}$ corresponds to the event that the smallest $k$ uniformly distributed order statistics are below the linear load surge function. The following result follows by applying Theorem~1 of~\cite{Sloothaak2016}.
\begin{proposition}
	Define $A^\star_{n+1}$ as the number of edge failures in a star network with $n+1$ nodes and $m=n$ edges, and load surge function given by~\eqref{eq:LoadSurgeStarTopology}. For every $k:=k_m$ satisfying $k \rightarrow \infty$ and $m-k \rightarrow \infty$ as $m  \rightarrow \infty$,
	\begin{align*}
	\Prob\left( A^\star_{n+1} \geq k\right) \sim \frac{2\theta}{\sqrt{2\pi}} \sqrt{\frac{m-k}{m}} k^{-1/2}.
	\end{align*}
	\label{prop:StarTopology}
\end{proposition}

In particular, if $1 \ll k \ll m$, it holds that 
\begin{align*}
\Prob\left( A^\star_{n+1} \geq k\right) \sim \frac{2\theta}{\sqrt{2\pi}} k^{-1/2}.
\end{align*}
A crucial argument used in the proof of Proposition~\ref{prop:StarTopology} is that as $n=m \rightarrow \infty$,
\begin{align*}
\Prob\left( A^\star_{n+1} \geq k\right) \sim \Prob\left(U^m_{(i)} \leq \frac{\theta+i-1}{m}, \hspace{.2cm} i=1,...,k\right).
\end{align*}
The asymptotic behavior of the latter expression is obtained by observing that the edge failure distribution is quasi-binomial for this particular load surge function, and exploiting the analytic expression for the probability distribution function to derive the tail behavior. Alternatively, this result can be derived by relating this problem to an equivalent setting where one is interested in the first-passage time of a random walk bridge, as is done in~\cite{SloWacZwa2017}. We use such a relation in the more involved case where disconnections occur as well. Before moving to the general case, we briefly recall the translation to the equivalent random walk problem in the much simpler setting where no (edge) disconnections occur yet.

\bigskip
\noindent
\textbf{\textit{Translation to random walk setting:}}\\
Note that
\begin{align}
\begin{split}
&\left(U_{(1)}^{m},U_{(2)}^{m},...,U_{(m)}^{m}\right)\overset{d}{=} \left(\frac{\textrm{Exp}_{1}(1)}{m},\frac{\sum_{j=1}^2 \textrm{Exp}_{j}(1)}{m},...,\frac{\sum_{j=1}^m \textrm{Exp}_{j}(1)}{m} \, \bigg| \, \sum_{j=1}^{m+1} \textrm{Exp}_{j}(1)=m \right),
\end{split}
\label{eq:RelationUniformExponentials}
\end{align}
where the exponentially distributed random variables are independent. Define the random walk 
\begin{align}
S_i = \sum_{j=1}^i \left(1-\textrm{Exp}_j(1)\right), \hspace{0.5cm} i\geq 1,
\end{align}
where $S_0=0$. Then,
\begin{align*}
\Prob\left( A^\star_{n+1} \geq k\right) &\sim \Prob\left(U^m_{(i)} \leq \frac{\theta+i-1}{m}, \hspace{.2cm} i=1,...,k\right)= \Prob\left(S_i \geq 1-\theta, \hspace{.2cm} i=1,...,k \big| S_{m+1}=1 \right).
\end{align*}
Define for every $x \in \mathbb{R}$,
\begin{align}
\tau^*_x := \min\{ i\geq 1 : S_i \leq x \}.
\end{align}
Then, the objective can be written as
\begin{align*}
\Prob\left( A^\star_{n+1} \geq k\right) \sim \Prob\left(\tau^*_{1-\theta} > k \big| S_{m+1}=1 \right) \sim \Prob\left(\tau^*_{1-\theta} > k \big| S_{m}=0 \right).
\end{align*}
Using the main result in~\cite{SloWacZwa2017}, we observe that for $k=o(m)$,
\begin{align*}
\Prob\left( A^\star_{n+1} \geq k\right) \sim \Prob\left(\tau^*_{1-\theta} > k \big| S_{m}=0 \right) &\sim \sqrt{\frac{2}{\pi}} \E\left(-S_{\tau_{1-\theta}}\right) k^{-1/2} = \frac{2\theta}{\sqrt{2\pi}}  k^{-1/2}, 
\end{align*}
where the latter equality follows from the memoryless property of exponentials.

A similar random walk construction can be done for the case where disconnections do take place. Before we consider this process, we first show the result of Theorem~\ref{thm:MainResultSublinear} in the phase where it is unlikely to have any disconnections in the connected configuration model. That is, removing $k=o(\sqrt{m})$ edges uniformly at random is unlikely to cause any disconnection by Theorem~\ref{thm:Firstdisconnection}. Due to the coupling between the cascading failure process and sequential edge-removal process, Proposition~\ref{prop:StarTopology} prescribes exactly the asymptotic behavior of the edge failure size in that case. 

\begin{theorem}
	The probability that the cascading failure process disconnects the network is given by
	\begin{align}
	\Prob (A_{n,\mathbf{d}} \geq T_{n,\mathbf{d}})\sim \frac{2\theta}{\sqrt{2\pi}} \left( \frac{2 p_2}{d-2 p_2} \right)^{1/4} \Gamma\left( \frac{3}{4} \right) m^{-1/4},
	\end{align}
	where $\Gamma(\cdot)$ denotes the gamma function. Consequently, for any threshold $1 \ll k \ll \sqrt{m}$,
	\begin{align}
	\Prob \left(A_{n,\mathbf{d}} \geq k \right) \sim \frac{2\theta}{\sqrt{2\pi}} k^{-1/2}.
	\label{eq:MainResultEdgeskVerySmall}
	\end{align}
	\label{thm:kVerySmall}
\end{theorem}

\noindent
\begin{proof}
	Note that
	\begin{align*}
	\Prob(A_{n, \mathbf{d}}\geq T_{n,\mathbf{d}})  &= \int_{0}^\infty  \Prob\left(\sqrt{m} T_{n,\mathbf{d}}  \in dt \right) \Prob(A^\star _{m+1}\geq t \sqrt{m}) \,dt \sim \frac{2\theta}{\sqrt{2\pi}} m^{-1/4} \E\left[ (\sqrt{m} T_{n,\mathbf{d} })^{-1/2} \right] \\
	&\sim \frac{2\theta}{\sqrt{2\pi}} m^{-1/4} \int_{0}^\infty \frac{ 4p_2}{d-2 p_2} t^{1/2} e^{-\frac{2t^2 p_2}{d-2p_2}} \,dt = \frac{2\theta}{\sqrt{2\pi}} \left( \frac{2p_2}{d-2 p_2} \right)^{1/4} \Gamma\left( \frac{3}{4} \right) m^{-1/4},
	\end{align*}
	where the second assertion follows due to the uniform convergence result of $A^\star _{m+1}$, see Theorem~1 in~\cite{SloWacZwa2017Split}, and the third assertion follows from Theorem~\ref{thm:Firstdisconnection}. For the second claim of the theorem, note that
	\begin{align*}
	\Prob \left(A_{n,\mathbf{d}} \geq k \right) &= \Prob \left(A_{n,\mathbf{d}} \geq k , A_{n,\mathbf{d}} < T_{n,\mathbf d}\right) + \Prob \left(A_{n,\mathbf{d}} \geq k , A_{n,\mathbf{d}} \geq T_{n,\mathbf d} \right),
	\end{align*}
	where due to Proposition~\ref{prop:StarTopology},
	\begin{align*}
	\Prob \left(A_{n,\mathbf{d}} \geq k , A_{n,\mathbf{d}} < T_{n,\mathbf d}\right) &= \Prob \left(A_{n,\mathbf{d}} \geq k \big| A_{n,\mathbf{d}} < T_{n,\mathbf d}\right) \Prob \left(A_{n,\mathbf{d}} < T_{n,\mathbf d}\right)\\
	&= \Prob \left(A^\star_{m+1} \geq k \right) \underbrace{\Prob \left(A_{n,\mathbf{d}} < T_{n,\mathbf d}\right)}_{=1-o(1)} \sim \frac{2\theta}{\sqrt{2\pi}} k^{-1/2},
	\end{align*}
	and
	\begin{align*}
	\Prob \left(A_{n,\mathbf{d}} \geq k , A_{n,\mathbf{d}} \geq T_{n,\mathbf d}\right) \leq  \Prob \left(A_{n,\mathbf{d}} \geq T_{n,\mathbf d}\right) = O(m^{-1/4}) = o(k^{-1/2}).
	\end{align*}
\end{proof}

\subsection{Random walk formulation}\label{sec:RWFormulation}
This section is devoted to introducing a related random walk, and to showing that the number of edge failures in the giant asymptotically behaves the same as the first-passage time of a random walk bridge.

Recall that $|\hat{E}_m(i)|$ denotes the number of edges in the giant when $i$ edges have been removed uniformly at random, and let $e_{i}$ correspond to the edge corresponding to the $i$'th order statistic of the surplus capacities. Define the sequence of processes $\{L_{i,m}: 1 \leq i \leq m+1 , m \geq 1\}$, where $L_{1,m}=1$, and for $2 \leq i \leq m+1$,
\begin{align}\label{eq:SurgeFakeNoTheta}
L_{i,m} = \left\{\begin{array}{ll}
\frac{m+1-\sum_{j=1}^{i-1} L_{j,m}}{|\hat{E}_m(i-2)|} & \textrm{if } e_{i-1} \in \mathcal{C}_{\max}, \\
0 & \textrm{if } e_{i-1} \not\in \mathcal{C}_{\max} .
\end{array}\right.
\end{align}
Note that this corresponds to the load surge increments in the giant if $\theta=1$, rescaled by a factor $m+1$. We consider a sequence of random walks $(S_{i,m})_{m\geq 1, 1 \leq i \leq m+1}$ defined as
\begin{align}
S_{i,m} = \sum_{j=1}^i X_{j,m}
\end{align}
with increments
\begin{align}\label{eq:JumpsFake}
X_{i,m} = L_{i,m}-\textrm{Exp}_{i,m}(1),
\end{align}
where $\textrm{Exp}_{i,m}(1)$ are independent exponential random variables with unit rate. We note that $e_m \in \Cmax$ and $|\hat{E}_m(m-1)|=1$ by definition, since removing $m-1$ edges leaves only isolated nodes and one component with two nodes connected by a single edge which inherently is the component that contains the largest number of edges. Therefore,
\begin{align*}
\sum_{j=1}^{m+1} L_{j,m} = \sum_{j=1}^{m} L_{j,m} + \frac{m+1-\sum_{j=1}^{m} L_{j,m}}{1} = m+1,
\end{align*}
and hence the random walk satisfies the property
\begin{align}
S_{m+1,m} = \sum_{j=1}^{m+1} X_{j,m} = m+1 - \sum_{j=1}^{m+1} \textrm{Exp}_{j,m}(1).
\label{eq:EndRWIdentity}
\end{align}
Finally, for all $m \geq 1$ define the stopping times
\begin{align}
\tau_m = \min\{ 1 \leq i \leq m : S_{i,m} < 1-\theta \},
\label{eq:TauStoppingTimeDefin}
\end{align}
and $\tau=m+1$ whenever $ S_{i,m} \geq 1-\theta$ for all $i=1,...,m$. 

In case of the star topology, i.e.~no edge disconnections occur, it holds that $|\hat{E}_m(i)| = m-i$, causing $L_{i,m}=1$ for all $1 \leq i \leq m+1$. In that case we observed that the failure size tail could be written as the first-passage tail of the random walk bridge. The above random walk formulation is a generalization that accounts for edges that may possibly no longer be contained in the giant as edge failures occur. 

\begin{proposition}
	Suppose $m^\delta \ll k \ll m^{1-\delta}$ for some $\delta \in (0,1/2)$. If
	\begin{align}
	\Prob\left(\tau_m \geq k \bigg| \, S_{m+1,m} =0 \right) \sim \frac{2\theta}{\sqrt{2\pi}} k^{-1/2},
	\label{eq:HypothesisRWStatement}
	\end{align}
	then
	\begin{align*}
	\Prob\left( \hat{A}_{n,\mathbf{d}} \geq \kappa(k) \right) \sim \Prob\left(\tau_m \geq k \bigg| \, S_{m+1,m} =0 \right).
	\end{align*}\label{prop:RelationToRW}
\end{proposition}

\begin{proof}[Proof of Proposition~\ref{prop:RelationToRW}]
	Write $\hat{l}(\cdot)$ for the load surge function corresponding to the edges in the giant. Note that by construction,
	\begin{align*}
	\Prob\left( \hat{A}_{n,\mathbf{d}} \geq \kappa(k) \right)& =
	\Prob\left(U_{(i)}^m \mathbbm{1}_{\{e_{i} \in \mathcal{C}_{\max} \} }\leq \hat{l}(\kappa(i)), \;\;\; i=1,...,k \right).
	\end{align*}
	In other words, whenever an edge is contained in the giant, one checks whether this edge has sufficient capacity to deal with the load surge function. Instead of looking only at those instances, we would like to compare all order statistics to an appropriately chosen function.  For this purpose, we introduce the function $l^*(\cdot)$ defined as
	\begin{align*}
	l^*(i) = \hat{l}(\kappa(i-1)+1), \hspace{1cm} i=1,...,m.
	\end{align*}
	We note this function satisfies two important properties:
	\begin{itemize}
		\item[(p1)] $l^*(i)=\hat{l}(\kappa(i))$ if $e_i \in \mathcal{C}_{\max}$;
		\item[(p2)] $l^*(i)=l^*(i-1)$ if $e_{i-1} \not\in \mathcal{C}_{\max}$.
	\end{itemize}
	Moreover, as $\hat{l}(\cdot)$ is non-decreasing for all $i\geq 2$, this holds as well for $l^*(i)$. We define the two stopping times,
	\begin{align*}
	\hat{T} = \min\{ 1\leq i \leq m: U_{(i)}^m \mathbbm{1}_{\{e_{i} \in \mathcal{C}_{\max} \} }> \hat{l}(\kappa(i)) \}
	\end{align*}
	and
	\begin{align*}
	T^* = \min\{ 1\leq i \leq m: U_{(i)}^m  > l^*(i) \}.
	\end{align*}
	We observe that the first property~(p1) implies that $T^* \leq \hat{T}$, and $T^* = \hat{T}$ if $e_{T^*} \in  \mathcal{C}_{\max}$. Together with the second property~(p2) and the observation that $U_{(\hat{T})}^m \geq U_{(T^*)}^m$, this implies
	\begin{align}
	\hat{T} = \min\{ j \geq T^* : e_j \in \mathcal{C}_{\max} \}.
	\label{eq:RelationTStarAndTHat}
	\end{align}
	Therefore,
	\begin{align}
	\Prob\left( \hat{A}_{n,\mathbf{d}} \geq \kappa(k) \right) = \Prob\left(\hat{T} > k \right) = \Prob\left({T}^{*}> k \right) + \Prob\left({T}^{*} \leq k ; \hat{T} > k \right). 
	\label{eq:SplitTStarAndTHat}
	\end{align}
	
	To conclude the proof, we relate the random walk to the stopping time $T^\star$, and show that the second contribution in~\eqref{eq:SplitTStarAndTHat} is negligible. For the first claim, we consider the perturbed increments $\{L_{i,m}(\theta); 1\leq i \leq m+1, m\geq 1 \}$ with $L_{1,m}(\theta)=\theta$ and
	\begin{align}\label{eq:SurgeFake}
	L_{i,m}(\theta) = \left\{\begin{array}{ll}
	\frac{m+1-\sum_{j=1}^{i-1} L_{j,m}(\theta)}{|\hat{E}_m(i-2)|} & \textrm{if } e_{i-1} \in \mathcal{C}_{\max}, \\
	0 & \textrm{if } e_{i-1} \not\in \mathcal{C}_{\max} .
	\end{array}\right.
	\end{align}
	Note that this corresponds to the load surge increments rescaled by a factor $m+1$. In particular, we observe $L_{\cdot,\cdot} = L_{\cdot,\cdot}(1) $, and
	\begin{align*}
	(m+1)  l^*(i) = \theta \left(\frac{m+1}{m}-1\right) + \sum_{j=1}^i L_{j,m}(\theta) = \frac{\theta}{m} + \sum_{j=1}^i L_{j,m}(\theta) .
	\end{align*}
	Note that $\theta/m=O(1/m)$, and
	\begin{align*}
	\max_{i=1,...,k}\left\{ \left|\sum_{j=1}^i \left( L_{j,m}(\theta) - L_{j,m} \right) - (\theta-1)\right| \right\} \leq \max_{i=1,...,k}   \frac{|1-\theta|}{|\hat{E}_m(i)|},
	\end{align*}
	which is of order $O(1/m)$ with probability $1- o(m^{-2})$ by Theorem~\ref{thm:DeviationsOutsideGiant}. Since
	\begin{align*}
	\Prob\left({T}^{*}> k \right) &= \Prob\left((m+1) U_{(i)}^m \leq (m+1)l^*(i), \;\;\; i=1,...,k \right) \\
	&= \Prob\left(\sum_{j=1}^i \textrm{Exp}_j(1) \leq \frac{\theta}{m} + \sum_{j=1}^i L_{j,m}(\theta), \;\;\; i=1,...,k \, \big\vert\, \sum_{j=1}^{m+1} \textrm{Exp}_j(1) = m+1 \right)\\
	&= \Prob\left(\sum_{j=1}^i X_{j,m} \geq - \frac{\theta}{m} + \sum_{j=1}^i (L_{j,m}-L_{j,m}(\theta)), \;\;\; i\leq,k \, \big\vert\, \sum_{j=1}^{m+1} \textrm{Exp}_j(1) = m+1 \right),
	\end{align*}
	it follows that
	\begin{align*}
	\Prob&\left({T}^{*}> k \right) = \Prob\left(\sum_{j=1}^i X_{j,m} \geq 1-\theta+o(1), \;\;\; i=1,...,k \, \big\vert\, \sum_{j=1}^{m+1} \textrm{Exp}_j(1) = m+1 \right) + o(m^{-2}).
	\end{align*}
	Due to our hypothesis~\eqref{eq:HypothesisRWStatement}, it follows that as $m\rightarrow \infty$,
	\begin{align*}
	\Prob\left({T}^{*}> k \right)  &\sim \Prob\left(\sum_{j=1}^i X_{j,m} \geq 1-\theta, \;\;\; i=1,...,k \, \big\vert\, \sum_{j=1}^{m+1} \textrm{Exp}_j(1) = m+1 \right) = \Prob\left(\tau_m \geq k \bigg| \, S_{m+1,m} =0 \right).
	\end{align*}
	
	To conclude the result, it remains to be shown that the second term in~\eqref{eq:SplitTStarAndTHat} is of order $o(k^{-1/2})$. Since we assumed that $m^\delta \ll k \ll m^{1-\delta}$ for some $\delta \in (0,1/2)$, we observe that there exists an $\alpha \in (0,1)$ such that both $k^2/m \ll m^\alpha \ll k$. For all such $\alpha \in (0,1)$, it holds that 
	\begin{align*}
	\Prob\left({T}^{*} \leq k ; \hat{T} > k \right) = \Prob\left({T}^{*} \in [k-m^\alpha, k] ; \hat{T} > k \right) + \Prob\left({T}^{*} < k-m^\alpha ; \hat{T} > k \right).
	\end{align*}
	We note that by our hypothesis and our previous result,
	\begin{align*}
	\Prob\left({T}^{*} \in [k-m^\alpha, k] ; \hat{T} > k \right) &\leq \Prob\left({T}^{*} > k-m^\alpha \right) - \Prob\left({T}^{*} > k ; \hat{T} > k \right) \\
	&\sim \frac{2\theta}{\sqrt{2\pi}} (k-m^\alpha)^{-1/2} - \frac{2\theta}{\sqrt{2\pi}} k^{-1/2} = o(k^{-1/2}).
	\end{align*}
	Finally, we observe that by~\eqref{eq:RelationTStarAndTHat},
	\begin{align*}
	\Prob\left({T}^{*} < k-m^\alpha ; \hat{T} > k \right) &\leq \sum_{j=1}^{k-m^\alpha} \Prob\left({T}^{*} = j ; \hat{T} - T^* > m^\alpha \right) \\
	&\leq \sum_{j=1}^{k-m^\alpha} \sum_{r=0}^m \left(\frac{m-j+1-r}{m-k}\right)^{m^\alpha} \Prob\left( |\hat{E}_m(j-1)| =r \right) \\
	&\leq o(m^{-1/2}) + \sum_{j=1}^{k-m^\alpha} \left(\frac{j^\alpha}{m-k}\right)^{m^\alpha} = o(m^{-1/2}) + O\left(k \left(\frac{k^\alpha}{m-k} \right)^{m^\alpha} \right) =o(m^{-1/2}),
	\end{align*}
	where the third assertion follows from Theorem~\ref{thm:DeviationEdgesOutsideGiant}.
\end{proof}

To derive the asymptotic probability of $\{\hat{A}_{n,\mathbf{d}} \geq \kappa(k)\}$ to occur for $k\ll m^{1-\delta}$ for some $\delta \in (0,1)$, Proposition~\ref{prop:RelationToRW} implies that it suffices to show that the asymptotic behavior of the first-passage time of the defined random walk is given by~\eqref{eq:HypothesisRWStatement}.

\subsection{Behavior of the number of edge failures in the giant}\label{sec:AsympBehvGiant}
We start the analysis by showing that if $k= o(m^{\alpha})$ for some $\alpha \in (0,1)$, then Proposition~\ref{prop:SublinearK} holds for the number of failures in the giant. We recap this proposition next.

\SublinearK*

For $k= o(\sqrt{m})$, this result is already proven in Theorem~\ref{thm:kVerySmall}. For the remainder of the proofs in this section, we therefore assume $k= \Omega(\sqrt{m})$. 

To prove Proposition~\ref{prop:SublinearK}, we will extensively use the random walk
\begin{align*}
S_i = \sum_{j=1}^i \left(1-\textrm{Exp}_j(1)\right),\hspace{0.5cm} i\geq 1,
\end{align*}
where $S_0=0$. This is related to $\tau_m$ as defined in~\eqref{eq:TauStoppingTimeDefin} through the relation
\begin{align}
\tau_m = \min\{1 \leq i \leq m : S_i < 1-\theta + \sum_{j=1}^i \left(1-L_{j,m} \right) \},
\label{eq:RelationStoppingTimeRegularRW}
\end{align}
and $\tau_m = m+1$ if $S_i \geq 1-\theta + \sum_{j=1}^i \left(1-L_{j,m}\right)$ for all $1\leq i \leq m$. Moreover, for a sequence $g=\{g_i\}_{i \in \mathbb{N}}$, let~$T_g$ correspond to the first-passage time of the random walk $S_i$ over this sequence, i.e.,
\begin{align*}
T_g = \min\{ i \in \mathbb{N} : S_i < g_i\}.
\end{align*}

We use the following strategy to prove Proposition~\ref{prop:SublinearK}. First, we show that for a particular class of (deterministic) boundary sequences, it holds that
\begin{align*}
\Prob( T_g > k) \sim \frac{2\theta}{\sqrt{2\pi}} k^{-1/2}
\end{align*}
as $k \rightarrow \infty$. Next, we show that the boundary as given in~\eqref{eq:RelationStoppingTimeRegularRW} falls in this class of boundary sequences with sufficiently high probability, and hence 
\begin{align*}
\Prob( \tau_m > k) \sim \frac{2\theta}{\sqrt{2\pi}} k^{-1/2}
\end{align*}
as $k \rightarrow \infty$. Finally, we show that conditioning on the event that the random walk returns to zero at time $m+1$ does not affect the tail behavior, i.e. for all $k=o(m^\alpha)$ for some $\alpha \in (0,1)$, it holds that as $m \rightarrow \infty$,
\begin{align*}
\Prob\left( \hat{A}_{n,\mathbf{d}} \geq \kappa(k) \right) \sim \Prob\left( \tau_m > k\big\vert S_{m+1}=0 \right) \sim \Prob( \tau_m > k) \sim \frac{2\theta}{\sqrt{2\pi}} k^{-1/2}.
\end{align*}

\subsubsection{First-passage time for moving boundaries initially constant for sufficient time}
Before moving to the proof of Proposition~\ref{prop:SublinearK}, we consider the first-passage time behavior of $(S_i)_{i \in \mathbb{N}}$ for a particular class of moving boundaries. The next lemma shows that the first-passage time over a boundary that is monotone non-decreasing, grows slower than $\sqrt i$, and is initially constant for a sufficiently large time, behaves the same as the first-passage time over the constant boundary. 

\begin{lemma}
	Suppose $l:= l_k$ is such that $k^\alpha \ll l \ll k$ as $k \rightarrow \infty$ for some $\alpha \in (0,1)$. Define the boundary sequence
	\begin{align*}
	g_{i,l}^+ = \left\{ \begin{array}{ll}
	1-\theta & \textrm{if } i\leq l,\\
	i^\gamma & \textrm{if } i >l,
	\end{array}\right.
	\end{align*}
	with $\gamma \in (0,1/2)$. Then, as $k \rightarrow \infty$,
	\begin{align*}
	\Prob\left(T_{g^+} > k  \right) \sim \Prob\left(T_{1-\theta} > k  \right) .
	\end{align*}
	\label{lem:UBSimilarStoppingTimeBehavior}
\end{lemma}

\begin{remark}
	We point out that the class of boundary sequences as described in Lemma~\ref{lem:UBSimilarStoppingTimeBehavior} is not covered by the literature. That is, almost all related literature consider moving boundaries that can be described by sequences of the form $(g_i)_{i \in \mathbb{N}}$, i.e. that do not depend on $k$. The exception is~\cite{Doney2012}, but this paper restricts to constant boundaries only.
	
	Still, the literature offers some partial results. First, since $\{ T_{g^+} > k  \} \subseteq \{ T_{1-\theta} > k \}$, 
	\begin{align*}
	\Prob\left(T_{g^+} > k  \right) \leq \Prob\left(T_{1-\theta} > k  \right).
	\end{align*}
	For a lower bound, we would like to remark that $g_{i,l}^+$ is a non-decreasing sequence in $i \geq 1$, where $g_{i,l}^+\leq i^\gamma$ for all $i \in \mathbb{N}$. Therefore, due to Proposition~1 in~\cite{WachtelDenisov2016} (or Theorem~2 in~\cite{Greenwood1987}),
	\begin{align*}
	\Prob\left(T_{g^+} > k  \right) \geq \Prob\left(T_{i^\gamma} > k  \right) \sim c_\gamma \Prob\left(T_{0} > k  \right) \sim c_\gamma' \Prob\left(T_{1-\theta} > k  \right),
	\end{align*}
	where $c_\gamma,c_\gamma' \in [0,\infty)$. Moreover, since 
	\begin{align*}
	\sum_{i=1}^\infty \frac{i^\gamma}{i^{3/2}} < \infty,
	\end{align*}
	it holds that $c_\gamma,c_\gamma' > 0$~\cite{Greenwood1987}. In order to prove Lemma~\ref{lem:UBSimilarStoppingTimeBehavior}, we therefore need to show that $c_\gamma'=1$. 
\end{remark}

\begin{proof}
	First, recall that $\{ T_{g^+} > k  \} \subseteq \{ T_{1-\theta} > k \}$, and hence
	\begin{align*}
	\Prob\left(T_{g^+} > k  \right) \leq \Prob\left(T_{1-\theta} > k  \right).
	\end{align*}
	Therefore, it suffices to show that the reversed inequality holds asymptotically.
	
	We bound the moving boundary $g^+$ by an appropriate piecewise constant boundary with finitely many jumps. For each of these constant intervals, we use the results in~\cite{Doney2012} to show that the trajectory of the random walk is asymptotically indistinguishable with respect to the boundary $g^+$ and the piecewise constant one. Finally, we glue the intervals together to conclude the result.
	
	First, note that since $\gamma \in (0,1/2)$, we can assume without loss of generality that $\alpha>0$ is such that $k^{\alpha(1+\eta)} \ll l \ll k^{\alpha((2\gamma)^{-1}-\eta)}$ with $\eta= ((2\gamma)^{-1}-1)/4 >0$. To define the piecewise constant boundary, let
	\begin{align*}
	r := \min\{j \in \mathbb{N} \colon \alpha (2\gamma)^{-j} >1 \},
	\end{align*}
	and note that $1 \leq r < \infty$ since $2\gamma \in (0,1)$ and $\alpha\in (0,1)$. Choose a fixed $\epsilon>0$ sufficiently small such that 
	\begin{itemize}
		\item $\epsilon < \alpha \eta$, which implies that $l = o\left(k^{\alpha(2\gamma)^{-1}-\epsilon}\right)$;
		\item $\alpha/(2\gamma)-\epsilon < \alpha/(2\gamma)^2 - 2 \epsilon < ... < \alpha/(2\gamma)^r-r \epsilon$;
		\item $\epsilon<(1-\alpha (2\gamma)^{-r})/r$.
	\end{itemize}
	
	\noindent
	Define $t_{j,k}^\epsilon, j \geq 0$ with $t_{0,k}^\epsilon=l$ and
	\begin{align*}
	t_{j,k}^\epsilon = k^{\alpha (2\gamma)^{-j}-j\epsilon}, \hspace{0.5cm} 1 \leq j \leq r.
	\end{align*}
	We point out that $r$ corresponds to the number of times the piecewise constant boundary makes a jump, and the values $t_{j,k}^\epsilon$, $0 \leq j \leq r-1$, correspond to the times where the piecewise constant boundaries jump. Since $\epsilon>0$ is chosen sufficiently small, $l= t_{0,k}^\epsilon \ll t_{1,k}^\epsilon \ll ... \ll t_{r-1,k}^\epsilon \ll k \ll t_{r,k}^\epsilon$ as $k \rightarrow \infty$. Write
	\begin{align*}
	h^{(j)} = \left\{ \begin{array}{ll}
	1- \theta & \textrm{if } j=0,\\
	k^{\alpha (2\gamma)^{-(j-1)}/2-j \epsilon/2} & \textrm{if } 1 \leq j \leq r,
	\end{array}\right.
	\end{align*}
	and define the boundary sequence as
	\begin{align*}
	h_{i,k}^\epsilon = \left\{ \begin{array}{ll}
	h^{(0)}=1-\theta & \textrm{if } i \leq t_{0,k}^\epsilon = l,\\
	h^{(j)} & \textrm{if } t_{j-1,k}^\epsilon < i \leq t_{j,k}^\epsilon, \; 1 \leq j \leq r-1,\\
	h^{(r)}  & \textrm{if } i > t_{r-1,k}^\epsilon.
	\end{array}\right.
	\end{align*}
	We point out that by construction, 
	\begin{align*}
	h^{(j)}/\sqrt{t_{j-1,k}^\epsilon} = k^{-\epsilon/2} \Longrightarrow h^{(j)} = o\left(\sqrt{t_{j-1,k}^\epsilon}\right), \hspace{0.5cm} 1 \leq j \leq r,
	\end{align*}
	and hence $h_{i,k}^\epsilon=o(\sqrt{i})$ for all $i \leq k$ as $k \rightarrow \infty$. Moreover,
	\begin{align*}
	h^{(j)}/(t_{j,k}^\epsilon)^\gamma = k^{j\epsilon(\gamma-1/2)} \Longrightarrow h^{(j)} \geq (t_{j,k}^\epsilon)^\gamma \hspace{0.5cm} 1 \leq j \leq r.
	\end{align*}
	Consequently,
	\begin{align*}
	h_{i,k}^\epsilon \geq g_{i,l}^+, \hspace{0.5cm} 1 \leq i \leq k,
	\end{align*}
	and therefore we obtain the lower bound
	\begin{align*}
	\Prob\left(T_{g^+} > k  \right) \geq \Prob\left(T_{h^\epsilon} > k  \right).
	\end{align*}
	
	Next, we provide a lower bound for the tail behavior of $T_{h^\epsilon}$. Fix $\delta >0$, and note that
	\begin{align*}
	\Prob\left(T_{h^\epsilon} > k  \right) \geq \Prob\left(T_{h^\epsilon} > k ; S_{t_{j,k}^\epsilon} \in \left(\delta \sqrt{t_{j,k}^\epsilon},1/\delta \sqrt{t_{j,k}^\epsilon}\right) \; \forall \, 0 \leq j \leq r-1 \right).
	\end{align*}
	Conditioning on the position of the random walk at the times $t_{j,k}^\epsilon$, $0 \leq j \leq r-1$ yields
	\begin{align*}
	\Prob\left(T_{h^\epsilon} > k  \right) \geq &\int_{u_0 = \delta \sqrt{t_{0,k}^\epsilon}}^{1/\delta \sqrt{t_{0,k}^\epsilon}} \cdots \int_{u_{r-1}=\delta \sqrt{t_{r-1,k}^\epsilon}}^{1/\delta \sqrt{t_{r-1,k}^\epsilon}}   \Prob\left(T_{h^{(r)}} >k-t_{r-1,k}^\epsilon \big| S_0 = u_{r-1} \right) \\
	&\hspace{3cm} \cdot\prod_{j=0}^{r-1}  \Prob\left(S_{t_{j,k}^\epsilon-t_{j-1,k}^\epsilon} \in du_{j} ; T_{h^{(j)}} > t_{j,k}^\epsilon-t_{j-1,k}^\epsilon \big| S_0 = u_{j-1} \right),
	\end{align*}
	where we write $t_{-1}=0$ and $u_{-1}=0$ for convenience. In other words, we partition the trajectory of the random walk in intervals where the boundary is constant. Recall that for every $0 \leq j \leq r-1$, it holds that $t_{j,k}^\epsilon-t_{j-1,k}^\epsilon= t_{j,k}^\epsilon(1+o(1))$, and $h^{(j)} = o(\sqrt{t_{j-1,k}^\epsilon})$ for every $1 \leq j \leq r$. Applying Proposition~18 in~\cite{Doney2012}, we obtain uniformly in $u_{j} = \Theta(\sqrt{t_{j,k}^\epsilon})$, $1 \leq j \leq r-1$,
	\begin{align*}
	&\frac{\Prob\left(S_{t_{j,k}^\epsilon-t_{j-1,k}^\epsilon} \in du_{j} ; T_{h^{(j)}} > t_{j,k}^\epsilon-t_{j-1,k}^\epsilon \big| S_0 = u_{j-1} \right)}{d u_j} \\ 
	&\hspace{3cm} \sim  \sqrt{\frac{2}{\pi}} \frac{V(u_{j-1}-h^{(j)})}{\sqrt{t_{j,k}^\epsilon-t_{j-1,k}^\epsilon}} \frac{u_j-h^{(j)}}{t_{j,k}^\epsilon-t_{j-1,k}^\epsilon} \textrm{exp}\left\{-\frac{\left(u_j-h^{(j)}\right)^2}{2(t_{j,k}^\epsilon-t_{j-1,k}^\epsilon)} \right\},
	\end{align*}
	where $V(\cdot)$ denotes the renewal function corresponding to the decreasing ladder height process of the random walk. The behavior of this function is well-understood: it is non-decreasing and $V(t) \sim t / \E(-S_{T_0}) =t$ as $t\rightarrow \infty$. Since by construction, $h^{(j)}=o(u_{j-1})=o(u_{j})$ for every $1 \leq j \leq r-1$, we obtain
	\begin{align*}
	&\frac{\Prob\left(S_{t_j-t_{j-1}} \in du_{j} ; T_{h^{(j)}} > t_j-t_{j-1} \big| S_0 = u_{j-1} \right)}{d u_j} \\
	&\hspace{1cm}=(1+o(1)) \sqrt{\frac{2}{\pi}} \frac{V(u_{j-1}-(1-\theta))}{\sqrt{t_{j}-t_{j-1}}} \frac{u_j-(1-\theta)}{t_{j}-t_{j-1}} \textrm{exp}\left\{-\frac{\left(u_j-(1-\theta)\right)^2}{2(t_{j}-t_{j-1})} \right\} \\
	&\hspace{1cm}=(1+o(1)) \frac{\Prob\left(S_{t_j-t_{j-1}} \in du_{j} ; T_{1-\theta} > t_j-t_{j-1} \big| S_0 = u_{j-1} \right)}{d u_j}.
	\end{align*}
	Similarly, it holds uniformly in $u_{r-1} = \Theta(\sqrt{t_{r-1,k}^\epsilon})$~\cite[Proposition~18]{Doney2012},
	\begin{align*}
	\Prob\left(T_{h^{(r)}} > k-t_{r-1,k}^\epsilon \big| S_0 = u_{r-1} \right) \sim  \Prob\left(T_{1-\theta} > k-t_{r-1,k}^\epsilon \big| S_0 = u_{r-1} \right).
	\end{align*}
	Then,
	\begin{align*}
	\Prob\left(T_{h^\epsilon} > k  \right) &\geq (1+o(1))  \Prob\left(T_{1-\theta} > k ; S_{t_{j,k}^\epsilon} \in \left(\delta \sqrt{t_{j,k}^\epsilon},  \frac{\sqrt{t_{j,k}^\epsilon}}{\delta}\right) \; \forall \, 0 \leq j \leq r-1 \right).
	\end{align*}
	Conditioning on staying above the constant boundary and applying the union bound yields
	\begin{align*}
	\Prob\left(T_{h^\epsilon} > k  \right) &\geq (1+o(1)) \Prob\left(T_{1-\theta} >k \right) \left( 1- \sum_{j=0}^{r-1} \Prob\left(S_{t_{j,k}^\epsilon}  \not\in \left(\delta \sqrt{t_{j,k}^\epsilon},1/\delta\sqrt{t_{j,k}^\epsilon}\right) \big| T_{1-\theta} > k \right)\right)\\
	&= (1+o(1))\left( 1-r \left( 1- e^{-\frac{\delta^2}{2}}  \right) - r e^{-\frac{1}{2\delta^2}}\right) \Prob\left(T_{1-\theta} >k \right) .
	\end{align*}
	Letting $\delta \downarrow 0$, we find that for every $\epsilon >0$ sufficiently small,
	\begin{align*}
	\liminf_{k\rightarrow\infty} \frac{ \Prob\left(T_{g^+} > k  \right)}{\Prob\left(T_{1-\theta} >k \right)} \geq \liminf_{k\rightarrow\infty} \frac{\Prob\left(T_{h^\epsilon} > k  \right)}{\Prob\left(T_{1-\theta} >k \right)} =1,
	\end{align*}
	from which we conclude that the result holds.
\end{proof}

The next lemma shows a similar result as Lemma~\ref{lem:UBSimilarStoppingTimeBehavior}, yet with a boundary that is monotone non-increasing and initially constant for a sufficiently large time. The proof is similar to the proof of Lemma~\ref{lem:UBSimilarStoppingTimeBehavior}, and therefore given in Appendix~\ref{app:ResultsCasc}.

\begin{lemma}
	Suppose $l:= l_k$ is such that $k^\alpha \ll l \ll k$ for some $\alpha \in (0,1)$ as $k \rightarrow \infty$. Define the boundary sequence
	\begin{align*}
	g_{i,l}^- = \left\{ \begin{array}{ll}
	1-\theta & \textrm{if } i\leq l,\\
	-i^\gamma & \textrm{if } i >l,
	\end{array}\right.
	\end{align*}
	with $\gamma\in (0,1/2)$. Then, as $k \rightarrow \infty$,
	\begin{align*}
	\Prob\left(T_{g^-} > k  \right) \sim \Prob\left(T_{1-\theta} > k  \right) .
	\end{align*}
	\label{lem:LBSimilarStoppingTimeBehavior}
\end{lemma}

From Lemmas~\ref{lem:UBSimilarStoppingTimeBehavior} and~\ref{lem:LBSimilarStoppingTimeBehavior}, the following corollary follows directly:

\begin{corollary}
	Suppose $l:= l_k$ is such that both $k^\alpha \ll l \ll k$ for some $\alpha \in (0,1)$, and the boundary sequence satisfies
	\begin{align*}
	g_{i,l} = \left\{ \begin{array}{ll}
	1-\theta & \textrm{if } i\leq l,\\
	o(i^\gamma) & \textrm{if } i >l,
	\end{array}\right.
	\end{align*}
	for some $\theta>0$ and $\gamma \in (0,1/2)$. Then, as $k \rightarrow \infty$,
	\begin{align*}
	\Prob\left(T_{g} > k  \right) \sim \Prob\left(T_{1-\theta} > k  \right) \sim \frac{2\theta}{\sqrt{2\pi}} k^{-1/2}.
	\end{align*}
	\label{cor:SimilarStoppingTimeBehavior}
\end{corollary}

\subsubsection{Proof of Proposition~\ref{prop:SublinearK}}
To prove Proposition~\ref{prop:SublinearK}, we first show that the tail of $\tau_m$ behaves the same as that of $T_{1-\theta}$, after which we use the relation in Proposition~\ref{prop:RelationToRW} to derive the tail of $A_{n,\mathbf{d}}$. In view of~\eqref{eq:RelationStoppingTimeRegularRW}, we need to understand the behavior of the random walk
\begin{align*}
Y_{i,m} = \sum_{j=1}^i (1- L_{j,m}), \hspace{1cm} 1 \leq i \leq m+1,
\end{align*}
where $Y_{0,m}=0$. In order to apply Corollary~\ref{cor:SimilarStoppingTimeBehavior}, we therefore need to show that the random walk is likely to be close to zero for a sufficiently long time $l$, and within $[-i^\gamma, i^\gamma]$ for all $l \leq i \leq k$ for some $0 <\gamma < 1/2$. 

\begin{proposition}
	Suppose $k=o(m^\alpha)$ for some $\alpha \in (0,1)$ and $\gamma \in  (\alpha/2, 1/2)$. Then, as $m \rightarrow \infty$,
	\begin{align*}
	\Prob\left( \left| \sum_{j=1}^i \left(1-L_{j,m} \right) \right| > i^\gamma \textrm{ for some } 1 \leq i \leq k \right) = o(m^{-1/2}).
	\end{align*}
	\label{prop:NoExceedanceOfLs}
\end{proposition}

\begin{proof}
	From Proposition~\ref{propo:DisconnectivityEarly}, it follows that the probability that there are disconnections when removing less than $o(m^{1/4 -\epsilon})$ edges with $\epsilon \in (0,1/4)$ is of order $o(m^{-1/2})$. Therefore, for every $l = o(m^{1/4-\epsilon})$ with $\epsilon \in (0, 1/4)$, it is likely that $L_{i,m} = 0$ for every $i \leq l$ , and hence
	\begin{align*}
	\Prob&\left( \left| \sum_{j=1}^i \left(1-L_{j,m} \right) \right| > i^\gamma \textrm{ for some } 1 \leq i \leq l \right)\leq \Prob\left( \left| \sum_{j=1}^i \left(1-L_{j,m} \right) \right| \neq 0 \textrm{ for some } 1 \leq i \leq l \right) = o(m^{-1/2}).
	\end{align*}
	Therefore, to prove the proposition, it suffices to show that for every $k$ for which $m^{1/4-\epsilon} \ll k \ll m^\alpha$ for some $\alpha \in (0,1)$ and $\epsilon \in (0,1/4)$, 
	\begin{align*}
	\Prob\left( \left| \sum_{j=l}^i \left(1-L_{j,m} \right) \right| > i^\gamma \textrm{ for some } l \leq i \leq k \right) = o(m^{-1/2}),
	\end{align*}
	where e.g. $l = m^{(1/4-\epsilon)/2}$. Write $\pi_1=1$, and 
	\begin{align*}
	\pi_i = \frac{|\hat{E}_m(i-2)|}{m-i+2}, \hspace{1cm} 2 \leq i \leq m+1,
	\end{align*}
	a random variable representing the probability that edge $e_{i-1}$ is in the giant. Let ${Ber}(\pi)$ denote a Bernoulli distributed random variable with success probability $\pi$, and note that
	\begin{align}
	L_{i,m} = \left( \frac{1}{\pi_i} + \frac{Y_{i-1,m}}{|E_m(i-2)|} \right) \textrm{Ber}(\pi_i) \geq 0.
	\label{eq:StepVariableAsBernoulli}
	\end{align}
	In view of Theorem~\ref{thm:DeviationsOutsideGiant}, we observe that $\pi_i$ is likely to be
	\begin{align*}
	\pi_i \geq \frac{m-i+2-(i-2)^\alpha}{m-i+2} \geq 1-i^{\alpha-1}.
	\end{align*}
	More precisely, let $\mathcal{E} = \{\pi_i=1 \;  \forall i < l, \pi_i \geq 1-i^{\alpha-1}, \; \forall l \leq i \leq k  \}$. Then,
	\begin{align*}
	\Prob&\left( \left| \sum_{j=1}^i \left(1-L_{j,m} \right) \right| > i^\gamma \textrm{ for some } l \leq i \leq k \right)\leq \Prob\left( \left| \sum_{j=1}^i \left(1-L_{j,m} \right) \right| > i^\gamma \textrm{ for some } l \leq i \leq k  \,\bigg|\, \mathcal{E} \right) + \Prob(\mathcal{E}^c),
	\end{align*}
	where due to Theorem~\ref{thm:DeviationsOutsideGiant}, it holds that $\Prob(\mathcal{E}^c) = o(m^{-1/2})$. Next, we show that the summed probabilities have an exponentially decaying tail. Define the stopping time 
	\begin{align*}
	\sigma_i = \sup\left\{j \in \mathbb{N} : j \leq i, Y_{i,m} \geq 0 \right\}.
	\end{align*}
	We remark that $\sigma_i \geq l$. Due to~\eqref{eq:StepVariableAsBernoulli}, it holds for every $l \leq i \leq k$,
	\begin{align*}
	\Prob\left( Y_{i,m} < - i^\gamma  \bigg| \mathcal{E} \right) &\leq \sum_{r=l}^{i-1} \Prob\left(  \sum_{j=1}^i L_{j,m}  > i + i^\gamma ; \sigma_i =r \bigg| \mathcal{E} \right)\\
	& \leq \sum_{r=l}^{i-1} \Prob\left( -Y_{r,m} + \sum_{j=r+1}^i \frac{1}{\pi_i} \textrm{Ber}(\pi_i)  > (i-r) + i^\gamma ; \sigma_i=t \bigg| \mathcal{E} \right)(1+o(1)) \\ 
	&\leq \sum_{r=l}^{i-1}  \; \Prob\left(  \sum_{j=r+1}^i \frac{1}{\pi_i} \textrm{Ber}(\pi_i)  > (i-r) + i^\gamma \bigg| \mathcal{E} \right)(1+o(1)) .
	\end{align*}
	Applying Chernoff's bound, we obtain for every $t > 0$ 
	\begin{align*}
	\Prob&\left(  \sum_{j=r+1}^i \frac{1}{\pi_i} \textrm{Ber}(\pi_i)  > (i-r) + i^\gamma \bigg| \mathcal{E} \right)  \leq  e^{-t(i-r+i^\gamma)} \E \left[ \exp\left\{ t \sum_{j=r+1}^i \frac{1}{\pi_j} \textrm{Ber}(\pi_j) \right\} \bigg| \mathcal{E} \right].
	\end{align*}
	Although the random variables $\pi_1,...,\pi_i$ are not independent, they are conditioned to be close to one and satisfy a Markovian property. Let $p=1-i^{\alpha-1}$, and note that the conditional event $\mathcal{E}$ implies that $\pi_j \geq p$ for all $1 \leq j \leq i$. Define $\mathcal{F}_i$ as the filtration generated by removing $i$ edges uniformly at random. Applying the law of total expectation and noting that $1+x(e^{t/x}-1)$ is a (strictly) decreasing function for all $t >0$, we observe that
	\begin{align*}
	\E \left[ \exp\left\{ t \sum_{j=1}^i \frac{1}{\pi_j} \textrm{Ber}(\pi_j) \right\} \bigg| \mathcal{E} \right] &= \E \left[ \E \left[ \exp\left\{ t \sum_{j=1}^i \frac{1}{\pi_j} \textrm{Ber}(\pi_j) \right\} \bigg| \mathcal{F}_{i-1} ; \mathcal{E} \right] \right] \\
	&=\E \left[ \exp\left\{ t \sum_{j=1}^{i-1} \frac{1}{\pi_j} \textrm{Ber}(\pi_j) \right\} \E \left[ 1+\pi_i (e^{t/\pi_i}-1) \bigg| \mathcal{F}_{i-1} ; \mathcal{E} \right] \right] \\ 
	&\leq \left( 1+p (e^{t/p}-1) \right) \E \left[ \exp\left\{ t \sum_{j=1}^{i-1} \frac{1}{\pi_j} \textrm{Ber}(\pi_j) \right\} \bigg| \mathcal{E} \right].
	\end{align*}
	Applying the same argument recursively yields the bound
	\begin{align*}
	\Prob\left(  \sum_{j=r+1}^i \frac{1}{\pi_i} \textrm{Ber}(\pi_i)  > (i-r) + i^\gamma \bigg| \mathcal{E} \right) \leq  e^{-t(i-r+i^\gamma)}  \left( 1+p (e^{t/p}-1) \right)^{i-r}
	\end{align*}
	for every $t \geq 0$. We point out that the right-hand side corresponds exactly to the Chernoff bound that would have been obtained if we consider the sum of $i-r$ Bernoulli distributed independent random variables with parameter $p$. It follows that
	\begin{align*}
	\Prob&\left(  \sum_{j=r+1}^i \frac{1}{\pi_i} \textrm{Ber}(\pi_i)  > (i-r) + i^\gamma \bigg| \mathcal{E} \right) \leq \exp\left\{- \frac{i^{2\gamma}}{2(i-r)p(1-p)} \right\} =  \exp\left\{-\frac{1}{2} i^{2\gamma-\alpha} \right\}(1+o(1)).
	\end{align*}
	We conclude that
	\begin{align*}
	\Prob\left( Y_{i,m} < - i^\gamma  \bigg| \mathcal{E} \right) \leq  i \exp\left\{-\frac{1}{2} i^{2\gamma-\alpha} \right\} (1+o(1)).
	\end{align*}
	
	On the other hand, we can use analogous arguments to bound $\Prob\left( Y_{i,m} > i^\gamma  \big| \mathcal{E} \right)$. This would yield
	\begin{align*}
	\Prob\left( Y_{i,m} > i^\gamma  \bigg| \mathcal{E} \right)\leq i \exp\left\{-\frac{1}{4} i^{2\gamma-\alpha} \right\} (1+o(1)).
	\end{align*}
	We conclude that
	\begin{align*}
	\Prob\left( \left| \sum_{j=l}^i \left(1-L_{j,m} \right) \right| > i^\gamma \textrm{ for some } l \leq i \leq k \right) &\leq  \sum_{i=l}^k \Prob\left( \left| \sum_{j=1}^i \left(1-L_{j,m} \right) \right| > i^\gamma  \bigg| \mathcal{E} \right) + o(m^{-1/2})\\
	&\leq \sum_{i=l}^k  2 i \exp\left\{-\frac{1}{8} i^{2\gamma-\alpha}\right\} + o(m^{-1/2}) = o(m^{-1/2}).
	\end{align*}
\end{proof}

The tail behavior of $\tau$ follows directly by combining Proposition~\ref{prop:NoExceedanceOfLs} and Corollary~\ref{cor:SimilarStoppingTimeBehavior}.

\begin{corollary}
	If $k= o(m^{\alpha})$ for some $\alpha \in (0,1)$, then as $m \rightarrow \infty$,
	\begin{align*}
	\Prob\left(\tau_m > k  \right) \sim \Prob\left(T_{1-\theta} > k  \right) \sim \frac{2\theta}{\sqrt{2\pi}} k^{-1/2}.
	\end{align*}
	\label{cor:TauAssBehavior}
\end{corollary}

Proposition~\ref{prop:SublinearK} follows from Corollary~\ref{cor:TauAssBehavior} if we show that conditioning on the event that $S_{m+1,m} = S_{m+1} =0$ does not change the tail of the stopping time $\tau_m$. Indeed, this turns out to be the case.

\begin{proof}[Proof of Proposition~\ref{prop:SublinearK}]
	We show that asymptotically the behavior of the conditioned stopping time $\tau_m | S_{m+1}=0$ is determined solely by what happens for the increments until time $k$. Note that by Proposition~\ref{prop:RelationToRW} 
	\begin{align*}
	\Prob\left( \hat{A}_{n,\mathbf{d}} \geq \kappa(k) \right)& \sim \Prob\left(S_i \geq 1-\theta + \sum_{j=1}^i \left(1-L_{j,m} \right) , \;\;\; i=1,...,k \bigg| \,  S_{m+1} =0 \right)=\Prob\left(\tau_m > k \big| \,  S_{m+1} =0 \right).
	\end{align*}
	Fix $\epsilon \in (0,1)$. We bound the probability terms both from above and below, and show that these bounds asymptotically behave the same as $\epsilon \downarrow 0$. Denote by $f_{i}(\cdot)$ the density of the random walk $S_i$ at time $i \geq 1$. Since the density of the random walk is bounded, we point out that it holds that~\cite{Petrov1975}
	\begin{align}
	\lim_{m \rightarrow \infty} \sup_{x \in \mathbb{R}} \left| \sqrt{m}f_m(\sqrt{m}x)-\phi(x)\right|=0,
	\label{eq:DensityUniformConvergenceToNormal}
	\end{align}
	where $\phi(\cdot)$ denotes the standard normal density function. Note that
	\begin{align*}
	\Prob\left(\tau_m > k \big| \,  S_{m+1} =0 \right) &= \Prob\left(\tau_m > k ; S_k \leq \epsilon \sqrt{m} \big| \,  S_{m+1} =0 \right)  + \Prob\left(\tau_m > k ; S_k > \epsilon \sqrt{m} \big| \,  S_{m+1} =0 \right).
	\end{align*}
	For the first term, we observe
	\begin{align*}
	\Prob\left(\tau_m > k ; S_k \leq \epsilon \sqrt{m} \big| \,  S_{m+1} =0 \right) &= \frac{1}{f_{m+1}(0)} \int_{-\infty}^{\epsilon \sqrt{k}} \Prob\left(\tau_m > k ; S_k \in du \right) f_{m+1-k}(-u)  \\
	& \leq \frac{1}{f_{m+1}(0)} \Prob\left(\tau_m > k  \right) \sup_{u \in [ 1-\theta + \sum_{j=1}^k \left(1-L_{j,m} \right) , \epsilon \sqrt{k} ]} f_{m+1-k}(-u)  .
	\end{align*}
	Due to~\eqref{eq:DensityUniformConvergenceToNormal},
	\begin{align*}
	f_{m+1}(0) = \frac{(1+o(1))}{\sqrt{2\pi m}}
	\end{align*}
	and 
	\begin{align*}
	\sup_{x \in \mathbb{R}} f_{i}(\sqrt{i} x) \leq \frac{1+o(1)}{\sqrt{2\pi i}}
	\end{align*}
	as $i \rightarrow \infty$. This yields the upper bound
	\begin{align*}
	\limsup_{m \rightarrow \infty} \frac{\Prob\left(\tau_m > k ; S_k \leq \epsilon \sqrt{m} \big| \,  S_{m+1} =0 \right) }{ \Prob\left(\tau_m > k \right) } \leq  \limsup_{m \rightarrow \infty} \frac{\sqrt{2\pi m}}{\sqrt{2\pi (m+1-k)}} = 1.
	\end{align*}
	For the second term, we show it is negligible. Note that
	\begin{align*}
	\Prob\left(\tau_m > k ; S_k > \epsilon \sqrt{m} \big| \,  S_{m+1} =0 \right)& \leq \frac{\Prob\left(S_k > \epsilon \sqrt{m} \right) }{f_{m+1}(0) } = (1+o(1)) \sqrt{2\pi m } \, \Prob\left(S_k > \epsilon \sqrt{m} \right).
	\end{align*}
	Applying Chernoff's bound, it holds for every $t \geq 0$,
	\begin{align*}
	\Prob\left(S_k > \epsilon \sqrt{m} \right) \leq \textrm{exp}\left\{-t \epsilon \sqrt{m} + k t + k \log\left(\frac{1}{1+t}\right) \right\}.
	\end{align*}
	In particular, this holds for $t = \epsilon \sqrt{m} / (k- \epsilon \sqrt{m}) >0$ (for $m$ large enough). Using this choice of $t$ and applying series expansions, we derive
	\begin{align*}
	\Prob\left(S_k > \epsilon \sqrt{m} \right) &\leq \textrm{exp}\left\{- \left(\frac{\epsilon^2 m}{k} + \epsilon \sqrt{m}\right) \frac{1}{1-\epsilon\sqrt{m}/k} + k \log\left(1-\frac{\epsilon\sqrt{m}}{k} \right) \right\} \\
	&= \textrm{exp}\left\{- \frac{\epsilon^2 m+\epsilon \sqrt{m}k}{k}  \left( 1+\frac{\epsilon\sqrt{m}}{k} + O\left(\frac{m}{k^2} \right) \right) - \epsilon\sqrt{m} - \frac{\epsilon^2 m}{2 k} -   O\left(\frac{m^{3/2}}{k^2} \right) \right\} \\
	&= \textrm{exp}\left\{ -\frac{\epsilon^2 m}{k} +o(1) \right\} .
	\end{align*}
	Due to Corollary~\ref{cor:TauAssBehavior},
	\begin{align*}
	\Prob\left(\tau_m > k  \right) = \Theta\left(k^{-1/2}\right),
	\end{align*}
	and hence
	\begin{align*}
	\frac{\Prob\left(\tau_m > k ; S_k > \epsilon \sqrt{m} \big| \,  S_{m+1} =0 \right) }{\Prob\left(\tau_m > k  \right) } = O\left( \sqrt{km} \; \textrm{exp}\left\{ -\frac{\epsilon^2 m}{k} +o(1) \right\} \right) = o(1).
	\end{align*}
	We conclude the upper bound
	\begin{align*}
	\limsup_{m \rightarrow \infty} \frac{\Prob\left(\tau_m > k \big| \,  S_{m+1} =0 \right) }{\Prob\left(\tau_m > k  \right) } \leq 1.
	\end{align*}
	
	For a lower bound, we observe
	\begin{align*}
	\Prob\left(\tau_m > k \big| \,  S_{m+1} =0 \right) &\geq \Prob\left(\tau_m > k ; S_k \leq \epsilon \sqrt{m} \big| \,  S_{m+1} =0 \right) \\
	&\geq  (1+o(1))\sqrt{2\pi m} \Prob\left(\tau_m > k  \right) \inf_{u \in [ 1-\theta + \sum_{j=1}^k \left(1-L_{j,m} \right) , \epsilon \sqrt{k} ]} f_{m+1-k}(-u) .
	\end{align*}
	Due to Proposition~\ref{prop:NoExceedanceOfLs}, it holds with probability $1-o(m^{-1/2})$ that
	\begin{align*}
	\bigg| \sum_{j=1}^k \left(1-L_{j,m} \right) \bigg| =o(\sqrt{k}).
	\end{align*}
	Combining this observation with~\eqref{eq:DensityUniformConvergenceToNormal} yields
	\begin{align*}
	\inf_{u \in [ 1-\theta + \sum_{j=1}^k \left(1-L_{j,m} \right) , \epsilon \sqrt{k} ]} f_{m+1-k}(-u) = (1+o(1)) \frac{1}{\sqrt{2\pi m}} e^{-\frac{\epsilon^2}{2}}.
	\end{align*}
	We conclude that
	\begin{align*}
	\liminf_{m \rightarrow \infty} \frac{\Prob\left(\tau_m > k \big| \,  S_{m+1} =0 \right) }{\Prob\left(\tau_m > k  \right) } \geq e^{-\frac{\epsilon^2}{2}}.
	\end{align*}
	As $\epsilon \downarrow 0$, the lower bound tends to one as well. We conclude that as $m \rightarrow \infty$,
	\begin{align*}
	\Prob\left(\tau_m > k \big| \,  S_{m+1} =0 \right) \sim \Prob\left(\tau_m > k  \right) \sim \frac{2\theta}{\sqrt{2\pi}}k^{-1/2}
	\end{align*}
	due to Corollary~\ref{cor:TauAssBehavior}.
\end{proof}

\subsection{Proof of main result}\label{sec:ProofOfMainResultRG}
As is laid out in the proof strategy described in Section~\ref{sec:RoadMapCascFailureProcess}, it only remains to be shown that the stopping times $\upsilon(k)$ and $\varrho(k)$, as defined in~\eqref{eq:upsilonDefinition} and~\eqref{eq:varrhoDefinition} respectively, are close to $k$. This follows from the extremely likely events that only a few components disconnect from the giant, and that such components are relatively small. In particular, it is likely that $\upsilon(k)= k - o(k)$ and $\varrho(k)= k + o(k)$. 

\begin{lemma}
	Suppose $k=o(m^\alpha)$ for some $\alpha \in (0,1)$. Then,
	\begin{align*}
	\Prob\left( \upsilon(k) \leq k- k^\alpha \right) = o(m^{-1/2}).
	\end{align*}
	\label{lem:upsilonBound}
\end{lemma}

\begin{lemma}
	Suppose $k=o(m^\alpha)$ for some $\alpha \in (0,1)$. Then,
	\begin{align*}
	\Prob\left( \varrho(k) > k+ k^{\frac{\alpha+1}{2}} \right) = o(m^{-1/2}).
	\end{align*}
	\label{lem:varrhoBound}
\end{lemma}

The proofs of these two lemma are fairly straightforward, and given in Appendix~\ref{app:ResultsCasc} for completeness. Using these lemmas, we can prove our main result.

\noindent
\begin{proof}[Proof of Theorem~\ref{thm:MainResultSublinear}]
	Using the proof strategy as laid out in Section~\ref{sec:RoadMapCascFailureProcess}, we have the upper bound
	\begin{align*}
	\Prob\left(A_{n,\mathbf{d}} \geq k  \right) &\leq \Prob\left(\hat{A}_{n,\mathbf{d}} \geq \kappa(\upsilon(k)) \right) \leq \Prob\left(\hat{A}_{n,\mathbf{d}} \geq \kappa(\upsilon(k))  \; \big| \; \upsilon(k) \geq k - k^\alpha \right) + \Prob\left(\upsilon(k) \leq k - k^\alpha \right),
	\end{align*}
	where, due to Proposition~\ref{prop:SublinearK},
	\begin{align*}
	\Prob\left(\hat{A}_{n,\mathbf{d}} \geq \kappa(\upsilon(k))  \; \big| \; \upsilon(k) \geq k - k^\alpha \right) &\leq \Prob\left(\hat{A}_{n,\mathbf{d}} \geq \kappa(k - k^\alpha) \right) \sim \frac{2\theta}{\sqrt{2\pi}} (k - k^\alpha)^{-1/2} \sim \frac{2\theta}{\sqrt{2\pi}} k^{-1/2}.
	\end{align*}
	Due to Lemma~\ref{lem:upsilonBound},
	\begin{align*}
	\Prob\left(\upsilon(k) \leq k - k^\alpha \right) = o(m^{-1/2}) = o(k^{-1/2}).
	\end{align*}
	For the lower bound, we observe
	\begin{align*}
	\Prob&\left(A_{n,\mathbf{d}} \geq k  \right) \geq \Prob\left(\hat{A}_{n,\mathbf{d}} \geq \kappa(\varrho(k)) \right) \geq \Prob\left(\hat{A}_{n,\mathbf{d}} \geq \kappa(\varrho(k))  \; \big| \; \varrho(k) \leq k + k^{(\alpha+1)/2} \right) \Prob\left( \varrho(k) \leq k + k^{(\alpha+1)/2} \right).
	\end{align*}
	By Proposition~\ref{prop:SublinearK},
	\begin{align*}
	\Prob\left(\hat{A}_{n,\mathbf{d}} \geq \kappa(\varrho(k))  \; \big| \; \varrho(k) \leq k + k^{(\alpha+1)/2} \right) &\geq \Prob\left(\hat{A}_{n,\mathbf{d}} \geq \kappa\left(k + k^{(\alpha+1)/2}\right) \right) \\
	&\sim \frac{2\theta}{\sqrt{2\pi}} \left(k + k^{(\alpha+1)/2}\right)^{-1/2} \sim \frac{2\theta}{\sqrt{2\pi}} k^{-1/2},
	\end{align*}
	and due to Lemma~\ref{lem:varrhoBound},
	\begin{align*}
	\Prob\left( \varrho(k) \leq k + k^{(\alpha+1)/2} \right) = 1 - o(m^{-1/2}).
	\end{align*}
\end{proof}

\section{Universality principle}\label{sec:UniversalityPrinciple}
Theorem~\ref{thm:FailureSublinear} described the tail behavior of the failure size in case of sublinear thresholds. In this section, we consider whether the scale-free behavior prevails if the threshold is of linear size, i.e. the threshold is of the same order as the number of vertices/edges. We conjecture that the scale-free behavior prevails up to a critical point. Also, we provide an explanation why the scale-free behavior can extend to a wide class of other graphs as well. We stress that the arguments that are provided in this section are intuitive in nature, and rigorous proofs of the claims remain to be established.

The proof of Theorem~\ref{thm:FailureSublinear} relies on the translation to a first-passage time of the random walk bridge $S_{i,m}$ over a (moving) boundary that is close to constant $1-\theta$. We prove that the difference $S_{i,m}-S_i$ is small (Proposition~\ref{prop:NoExceedanceOfLs}) if $k$ is sublinear, causing the random walk $S_{i,m}$ to be asymptotically indistinguishable from $S_i$ up to time $k$. That is, the random walk bridge is asymptotically indistinguishable from the one corresponding to the star topology. Since $S_i$ is a random walk with independent identically distributed increments with zero mean and finite variance, it is well-known by Donsker's theorem that appropriately scaling the random walk bridge $S_i$ yields convergence to a Brownian bridge. Therefore, the probability that $A_{n,\mathbf{d}}$ exceeds $k$ asymptotically behaves the same as the probability that a Brownian bridge stays above zero until time $k$, multiplied by a constant that relates to the translation of the boundary to any other (constant) boundary. We recall that in case of $S_i = \sum_{j=1}^i (1-\textrm{Exp}_j(1))$ and a boundary (close to) $1-\theta$, this constant is given by $\theta$. 

When the threshold $k:=k_m$ is of the same order as $m$, this analysis does not follow through. The (maximum) difference between the two random walks up to time $k$ is likely to become of order $\Theta(\sqrt{k})=\Theta(\sqrt{m})$, an order of magnitude that affects the asymptotic behavior. Moreover, the number of edges outside the giant is also no longer likely to be of size $o(m)$, and hence we need to understand what the failure behavior typically is in these components as well. The natural question that comes to mind is whether the scale-free behavior in the failure size tail prevails or not for linear-sized thresholds. Next, we argue heuristically why this type of behavior prevails up to a certain critical point.

If we remove $k = \beta m$ edges uniformly at random, where $\beta \in (0,q_c)$, then there is a non-vanishing proportion of edges outside the largest (giant) component~\cite{MolRee95}. Nevertheless, the sizes of the components outside the giant are relatively small, i.e. the number of edges in such a component is at most $\Op(\log m)$. It turns out that this causes it to be likely that whenever a component detaches from the giant, the cascade stops in the small component. More specifically, suppose that the edge with the $i$'th smallest surplus capacity is contained in the giant upon failure, and causes a component to detach a component from the giant. Due to the way that the total load surge is defined, we conjecture that it is likely that the total load surge is close to $i/m + \Theta(1/\sqrt{m})$ upon failure. Since the number of edges in the smaller component is likely to be $O(\log m)$, all the surplus capacities of these edges are likely to be at least $i/m + \Omega(1/\log m)$, which is much larger than the total load surge. That is, all edges are likely to have sufficient capacity to deal with the load, and no more failures occur in the smaller detached component. Moreover, even if the cascade continues, it would contribute at most a logarithmic number of edges to the total failure size. These observations lead to the claim that the dominant contribution in the failure size comes from the number of edges that are contained in the giant upon failure.

To track the failure behavior in the giant component, one can use the sequence of random walks $S_{i,m}$. That is, the translation of the failure size to the first-passage time of the random walk bridge over the constant boundary $1-\theta$ remains (likely to be) true if $k \leq \beta m$ with $\beta \in (0,q_c)$. Although the random walk is no longer close to $S_i$, the increments of $S_{i,m}$ do have zero mean and a variance that is likely to be finite (and non-constant), and hence satisfy a martingale property. Therefore, we conjecture that the probability that the failure size in the giant component exceeds $\kappa(k)$ behaves the same as the probability that a Brownian bridge (with non-constant variance) stays above zero until time $k$, multiplied by a constant.

Since we argued that it should hold that $ A_{n,\textbf{d}} \approx  \hat{A}_{n,\textbf{d}}$, this leads to the following conclusion. Write $k=\alpha m$ with $\alpha \in (0,1)$. In view of~\eqref{eq:EdgesGiant}, we observe that for all $i:=i_m $ sufficiently smaller than the critical point $q_c$, it holds that $\kappa(i) \approx m \int_0^{i/m} \xi_{\mathbf{d}}(q)\mathrm d q$. Write
\begin{align}
\beta_\alpha:= \min_{x \in (0,1)}\left\{\int_0^x\xi_{\mathbf{d}}(q)\mathrm d q = \alpha\right\}.
\label{eq:BetaDefinitionRG}
\end{align}
Then $\varrho(k) \approx \beta_\alpha m$, since
\begin{align*}
\kappa(\varrho(k)) = k = \alpha m \approx  \kappa(\beta_\alpha m).
\end{align*}
Therefore,
\begin{align*}
\Prob( A_{n,\textbf{d}} \geq k ) \sim \Prob( \hat{A}_{n,\textbf{d}} \geq k ) \sim \Prob( \hat{A}_{n,\textbf{d}} \geq \kappa(\beta_\alpha m) ).
\end{align*}

Summarizing, we have the following conjecture for $\overline{CM}_n(d,q)$. Suppose $k= \alpha m$ with $\alpha \in (0,1)$ such that $\beta_\alpha < q_c$ is satisfied. Then, 
\begin{align*}
\Prob( A_{n,\textbf{d}} \geq k ) \sim f(\alpha) k^{-1/2}.
\end{align*}

\begin{figure}[htb]
	\centering
	\includegraphics[height=6cm]{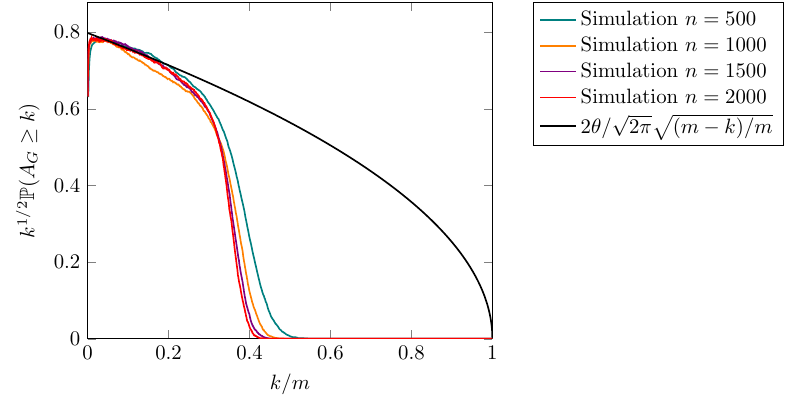}
	\caption{Erased configuration model.}
	\label{fig:CMTheta1}
\end{figure}

To support this conjecture, we performed a Monte-Carlo simulation experiment. In particular, we tested the conjecture against the erased configuration model. That is, we create this graph according to the configuration model mechanism with a prescribed degree sequence. We merge multiple edges and erase self-loops. Moreover, after sampling such a graph, we remove any smaller components such that the final graph is a simple connected graph. Due to the properties of the configuration model, the number of self-loops and multiple edges is very small (if any), and only a finite number of vertices and edges are not contained in the giant. As a result, this graph and the configuration model conditioned on connectivity are indistinguishable asymptotically, and will lead to the same asymptotic result for the number of edge failures. In our simulations we choose $n \in \{500,1000,1500,2000\}$, and a degree sequence with $n_1=\lceil n^{1/3} \rceil$ vertices of degree one, $n_2=n_3=n/2- \lceil n^{1/3} \rceil$ and $n_4= \lceil n^{1/3} \rceil$. Therefore, the number of edges is (close to) $m=5/4n$. The results are displayed in Figure~\ref{fig:CMTheta1}. Indeed, it appears that our conjecture holds in this case.

Not only do we believe that this conjecture holds for the connected configuration model, we argue that the scale-free behavior may hold for a wider range of graphs. In particular, the relevant properties of the configuration model that we used in the analysis are the following. First, it is likely that no (significant) disconnections occur at the beginning of the cascading failure process. For example, in the case of a configuration model, we showed that the first disconnection is likely to occur after $\Theta(\sqrt{m})$ edge failures. Secondly, whenever the cascading failure process causes disconnections to occur, a giant component appears and disconnections only create relatively small components. It is well-known that this property is satisfied up to a certain critical threshold $q_c$ for the configuration model, but this holds in fact for many more types of random graphs. In other words, for our result to prevail in other graph topologies, the graph should satisfy the following two properties:
\begin{itemize}
	\item The first disconnection (if any) is only likely to occur after $\Omega(m^\delta)$ failures, where $\delta >0$;
	\item There exists a critical parameter $q_c$ such that if $q < q_c$, the largest component of the percolated graph is unique w.h.p. and contains a non-vanishing proportion of the vertices and edges. Moreover, all other components are likely to be relatively small for $q< q_c$, e.g.~the second largest component contains at most $\Op( (\log m)^c)$ number of edges for some $c < \infty$.
\end{itemize}

Whenever a graph $G=(V,E)$ satisfies these two properties, we conjecture that the number of edge failures $A_G$ exhibits scale-free behavior. That is, for a range of thresholds $k:= k_m$, it holds that
\begin{align}
\Prob\left( A_{G} \geq k \right) \sim f_{G}(\alpha) k^{-1/2},
\label{eq:UniversalResult}
\end{align}
where $f_{G}(\cdot) >0$ and  $\alpha = \lim_{m \rightarrow \infty} k/m$. In particular, $f_G(0) = 2\theta/\sqrt{2\pi}$. The function $f_G$ depends on the specifics of the graph, such as average degree and more detailed connectivity properties. The range of values for $k$ for which~\eqref{eq:UniversalResult} holds depends on the critical threshold $q_c$ and the typical number of edges that are outside the giant component. 

\begin{figure}[htb]
	\centering
	\includegraphics[height=6cm]{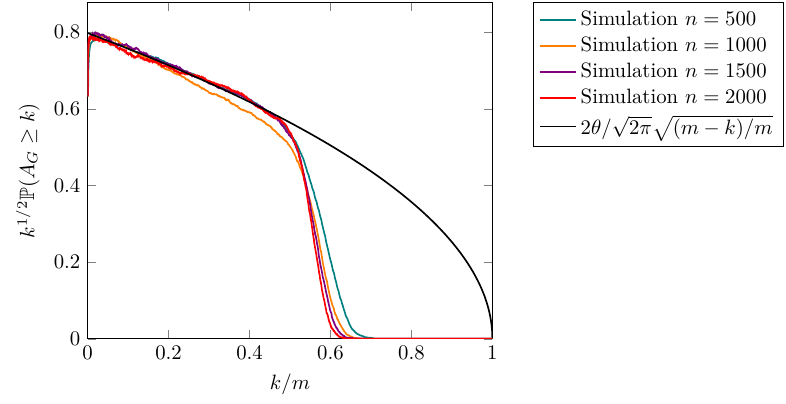}
	\caption{$\lceil \sqrt{n} \rceil \times \lceil \sqrt{n} \rceil$ square lattice graph.}
	\label{fig:LatticeTheta1}
\end{figure}

To test conjecture~\eqref{eq:UniversalResult}, we first consider a $\lceil \sqrt{n} \rceil \times \lceil \sqrt{n} \rceil$ square lattice graph where the opposite boundaries are not connected with $n \in \{500,1000,1500,2000 \}$. On the square lattice it is known that there is a phase transition for the existence of a giant component when $q_c=1/2$~\cite{Kes80}. Moreover, significant disconnections occur only after quite a significant number of failures have occurred. Indeed, to disconnect one edge $e$ (not on the boundary) from the giant we need to remove at least six edges (the ones that share an end-vertex with $e$). This suggests that since there are roughly $2n$ edges in the graph, the first time the process produces an edge disconnected from the giant should be of order $\Theta (n^{5/6})$. Moreover, it is known that in this regime the second largest component in a box of volume $n$ is polylogarithmic in $n$ \cite{KesZha90}. Thus, it satisfies the conditions we conjecture for a graph to be in the same mean-field universality class as the configuration model. In Figure~\ref{fig:LatticeTheta1}, we observe that indeed the first significant disconnections happen after a much longer time than in $\CMd$, and the limiting function for $k^{1/2} \Prob\left( A_{Lattice} \geq k \right) $ remains very close to the setting where no disconnection takes place for a relatively long time. 

\begin{figure}[htb]
	\centering
	\includegraphics[height=6cm]{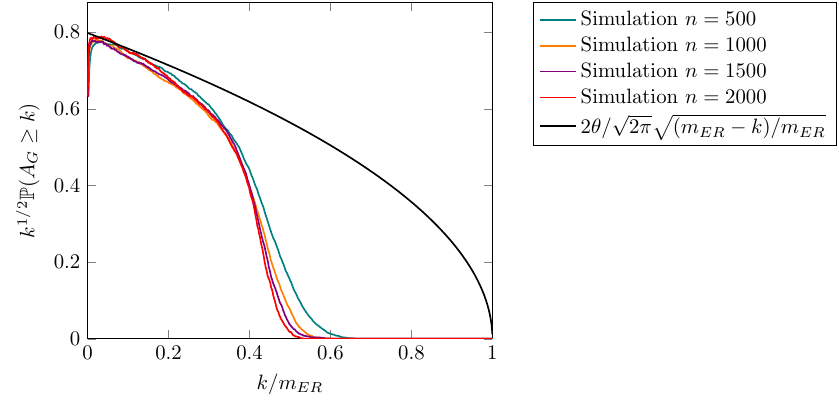}
	\caption{Giant component of the Erd\"os-R\'enyi random graph with edge retention probabability $2/n$, where $m_{ER}=\lambda (1-\eta_\lambda^2)n/2$.}
	\label{fig:ERFigure}
\end{figure}

Secondly, we consider the giant component of the Erd\"os-R\'enyi random graph. That is, for every pair of the $n$ vertices, there exists an edge with probability $\lambda/n$, and we consider the cascading failure process on the giant component. The giant is uniquely defined asymptotically: it is well-known that for every $\lambda>1$ there exists a unique giant component $\Cmax$ for which holds that~\cite{Hofs17},
\begin{align*}
\frac{|\{v : v \in\mathcal{C}_{\max}\}|}{n} \overset{\Prob}{\longrightarrow} \zeta_\lambda,
\end{align*}
where $\zeta_\lambda = 1-\eta_\lambda >0$ and $\eta_\lambda$ satisfies the fixed-point equation
\begin{align*}
\eta_\lambda = \E\left( \eta_\lambda^{\textrm{Pois}(\lambda)} \right)=1- e^{-\lambda\eta_\lambda}.
\end{align*}
In our simulation experiment, we choose $n \in \{500,1000,1500,2000\}$ possible vertices, and $\lambda=2$. Therefore, the graph on which we perform the cascading failure process is likely to have around $\zeta_\lambda n$ vertices and $\lambda (1-\eta_\lambda^2)n/2$ edges (with $\Theta(\sqrt{n})$ fluctuations). From the definition of the Erd\"os-R\'enyi random graph it is clear that if we run a percolation process on it, the resulting graph is again an Erd\"os-R\'enyi random graph, but with a smaller $\lambda$. It is known that for every $\lambda >1$ all the components outside the giant are at most of size $\Theta(\log n)$. Therefore, the disintegration of the network is similar to the one of the configuration model, yet with the critical edge-removal probability corresponding to $q_c= (\lambda -1)/\lambda$ in this case. 

However, the first disconnection is likely to occur after finitely many edge failures, since the number of vertices with degree one in the giant of an Erd\"os R\'enyi graph is likely to be of order $\Theta(n)$. In other words, this graph violates the condition that the first disconnections should occur after $\Omega(m^{\delta})$ edge failures for some $\delta>0$. Nevertheless, it appears from our simulation result that~\eqref{eq:UniversalResult} still prevails, see Figure~\ref{fig:ERFigure}. In other words, the condition on no early disconnections can possibly be relaxed. The analysis would require significant changes, and particularly, the results on the distribution of first-passage times over moving boundaries such as Lemmas \ref{lem:UBSimilarStoppingTimeBehavior} and~\ref{lem:LBSimilarStoppingTimeBehavior}, as the load surge is no longer almost deterministic at the beginning of the process.

\begin{figure}[htb]
	\centering
	\includegraphics[height=6cm]{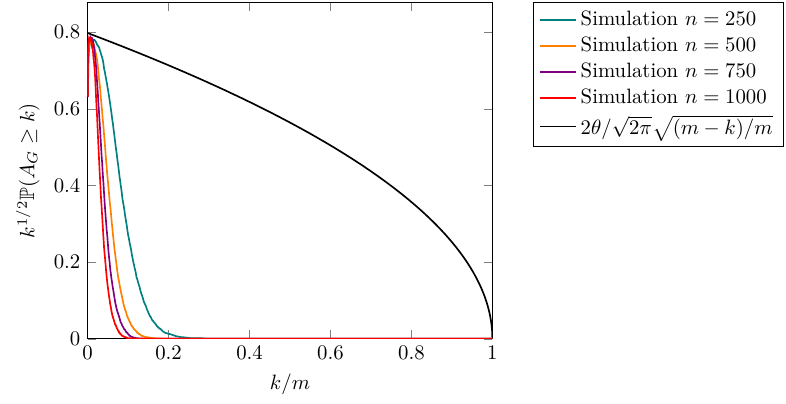}
	\caption{Graph $G=CS(n,4)$.}
	\label{fig:CircularStarTheta1}
\end{figure} 

In contrast, we consider a graph $G=CS(n,4)$ consisting of $n$ star-components with $4$ edges each, connected by a single path of edges connecting all the components. This graph thus consists of $5n$ vertices and also $m=5n$ edges. We observe that as soon as one of the $n$ edges on the single path fails, the remaining graph is a (connected) tree. Therefore, with high probability, the graph would disconnect in two components both of order $\Theta(m)$ after removing only a fixed number of edges. This effect is likely to occur again in both components when edges are removed uniformly at random in these components. This violates both properties we needed to prove our result for the configuration model. We observe in Figure~\ref{fig:CircularStarTheta1} that indeed $k^{1/2} \Prob\left( A_{CS} \geq k \right)$ does not seem to converge to a single function as $n \rightarrow \infty$.

\bibliographystyle{plain}
\bibliography{bibCM}

\newpage
\appendix
\section{Lists of variables}\label{app:Notation}
\begin{table}[htb]
	\centering
	\setlength\extrarowheight{2.3pt}
	\begin{tabular}{ | c | p{12cm} |}    
		\hline
		$n$ & \# vertices \\
		\hline
		$\mathbf{d}$ & (Fixed) degree sequence\\
		\hline
		$m=m_n(\mathbf{d})$ & \# edges in a configuration model with degree sequence $\mathbf{d}$ \\
		\hline
		$\CMd$ & Configuration model with degree sequence $\mathbf{d}$ \\
		\hline
		$\overline{CM}_n(\mathbf{d})$ & Configuration model with degree sequence $\mathbf{d}$ conditioned to be connected\\
		\hline
		$n_i$ & \# vertices with degree $i$ \\
		\hline
		$p_i$ & Fraction of vertices with degree $i$ \\
		\hline
		$D_n$ & Degree of a vertex chosen uniformly at random from $[n]$ \\
		\hline
		$D$ & Limiting degree variable of $D_n$ \\
		\hline
		$d$ & Average degree of $\mathbf{d}$ \\
		\hline
		$\theta$ & Disturbance constant \\
		\hline
		$l_j^m(i)$ & Total load surge at edge $j$ after experiencing $i$ load surges in a graph with $m$ edges \\
		\hline
		$|E_j^m(i)|$ & \# edges in the component containing edge $j$ after experiencing $i$ load surges \\
		\hline
		$A_{n,\mathbf{d}}$ & \# edge failures after the cascading failure process \\
		\hline
	\end{tabular}
	\caption{List of variables commonly used throughout this chapter.}
	\label{tab:NotationTableFirst}
\end{table}

\begin{table}[htb]
	\centering
	\setlength\extrarowheight{2.3pt}
	\begin{tabular}{ | c | p{12cm} |}    
		\hline
		$q$ & Removal probability in the percolation process\\
		\hline
		$CM_n(\mathbf{d},q)$ & Percolated configuration model with removal probability $q$ \\
		\hline
		$\mathcal{C}$ & Component of a graph, sometimes also denoted as $\mathcal{C}(x)$ when referring to the component that contains vertex or edge $x$ \\
		\hline
		$\Cmax$ & Largest component of a graph \\
		\hline
		$T_{n,\mathbf{d}}$ & \# edges that are sequentially removed uniformly at random for the first disconnection to occur in the connected configuration model\\
		\hline
		$T$ & Limiting variable of $m^{1/2}T_{n,\mathbf{d}}$\\
		\hline
		$|\hat{E}_m(i)|$ & \# edges in the giant (largest) component after $i$ edges have been removed uniformly at random \\
		\hline
		$|\tilde{E}_m(i)|$ & \# edges outside the giant (largest) component after $i$ edges have been removed uniformly at random \\
		\hline
	\end{tabular}
	\caption{List of variables commonly used in percolation/sequential edge-removal process.}
	\label{tab:PercolationNotation}
\end{table}

\begin{table}[htb]
	\centering
	\setlength\extrarowheight{2.3pt}
	\begin{tabular}{ | c | p{12cm} |}    
		\hline
		$R_n$ & \# removed half-edges in the first step of the explosion algorithm \\
		\hline
		$\mathbf{d}'$ & Degree sequence of configuration graph in step three of the explosion method ($\mathbf{d}' \in \mathbb{N}^{n+R_n}$)\\
		\hline
		$CM_N(\mathbf{d}')$ & Configuration model in step three of the explosion method with degree sequence $\mathbf{d}'$ \\
		\hline
		$CM_n(\mathbf{d},q)$ & Resulting configuration model in step four of the explosion method, indistinguishable from the percolated configuration model with removal probability $q$ \\
		\hline
		$n_i'$ & \# vertices of degree $i$ in $CM_N(\mathbf{d}')$ \\
		\hline
		$p_i'$ & $\lim_{n \rightarrow \infty} n_i'/n$ 	\\
		\hline
		$n_{l,j}$ & \# vertices of degree $l$ in $\textbf{d}$ that have degree $j$ in $\textbf{d}'$\\
		\hline
		$p_{l,j}$ & Probability for a vertex of degree $l$ to retain $j$ half-edges after the first step of the explosion algorithm\\
		\hline
		$L_k'(n)$ & \# components that are lines of length $k\geq 2$ in $\Cmd$\\
		\hline
		$C_k'(n)$ & \# components that are cycles of length $k \geq 1$ in $\Cmd$\\
		\hline
	\end{tabular}
	\caption{List of variables commonly used in the explosion algorithm.}
	\label{tab:OverviewExplosionAlgNotation}
\end{table}

\begin{table}[!ht]
	\centering
	\setlength\extrarowheight{2.3pt}
	\begin{tabular}{ | c | p{12cm} |}    
		\hline
		$\hat{A}_{n,\mathbf{d}}$ & \# edges that were contained in the largest component upon failure during the cascade \\
		\hline
		$\tilde{A}_{n,\mathbf{d}}$ & \# edges that were contained outside the largest component upon failure during the cascade \\
		\hline
		$A_{n+1}^*$ & \# edge failures in the cascading failure process on a star topology with $n+1$ nodes and $m=n$ edges \\
		\hline
		$\kappa(i)$ & \# edges that are contained in the largest component upon removal when $i$ edges are sequentially removed uniformly at random \\
		\hline 
		$\upsilon(i)$ & Minimum \# edges that need to be removed uniformly at random for the sum of $\upsilon(i)$ and \# edges outside the giant to exceed $i$ \\
		\hline
		$\varrho(i)$ & \# edges that need to be removed uniformly at random such that $i$ edges were contained in the giant component upon failure \\
		\hline
		$S_i$ & $\sum_{j=1}^i \left(1-\textrm{Exp}_j(1)\right)$ \\
		\hline
		$L_{i,m}$ & Scaled perturbed load surge, formally defined as in~\eqref{eq:SurgeFake} \\
		\hline
		$S_{i,m}$ & Random walk defined as $\sum_{j=1}^i L_{j,m}-\textrm{Exp}_{j,m}(1)$ \\
		\hline
		$Y_{i,m}$ & Random walk $\sum_{j=1}^i (1- L_{j,m})$\\
		\hline
		$\tau_m$ & First-passage time of random walk $S_{i,m}$ to be less than $1-\theta$ \\
		\hline
		$T_g$ & First-passage time of random walk $S_i$ to move below a boundary sequence $(g_i)_{i \in \mathbb{N}}$ \\
		\hline
	\end{tabular}
	\caption{List of variables commonly used in the failure process.}
	\label{tab:NotationTableFinal}
\end{table}
 
\cleardoublepage
\section{Proofs of results on the disintegration of the graph}\label{app:PercolationResults}
In this appendix, we present several proofs of results given in Section~\ref{sec:Disintegration}.

\begin{proof}[Proof of Lemma~\ref{lem:CascadePercolation}]
	Due to the i.i.d.~property of the surplus capacities, it holds that $\tilde{E}'(i)$ has the distribution of $i$ edges chosen uniformly without repetitions from $[m]$, so $\Prob (\tilde{E}'(i)=B)=\binom{m}{i}^{-1}$.
	
	Since all edges in $CM_n(\mathbf{d},q)$ are removed independently with probability $q$, it holds for all sets $B \subseteq E$ with $|B|=i$ that
	\begin{align}
	\Prob (E'(G(q))=B)=q^i(1-q)^{m-i},
	\end{align}
	while 
	\begin{align}\label{eq:RemovedBinomial}
	\Prob (|E'(G(q))|=i)= \binom{m}{i}q^i(1-q)^{m-i}.
	\end{align}
	Since $\{E'(G(q))=B \} \subseteq \{|E'(G(q))|=i\}$, we obtain
	\begin{align}
	\Prob(E'(G(q))=B\mid |E'(G(q))|=i)=\binom{m}{i}^{-1},
	\end{align}
	so \eqref{eq:iUniform} holds.
	From \eqref{eq:RemovedBinomial} we obtain \eqref{eq:RemovedEdges} by concentration of $\textrm{Bin}(m,q)$ if $qm\to \infty$.
\end{proof}

\begin{proof}[Proof of Lemma~\ref{lem:MultiVariateMoments}]
	We use the explosion algorithm from Algorithm~\ref{alg:ExplosionAlgorithm}. To illustrate the type of arguments we use to prove this statement, we first consider the first moment only.
	
	Define $ \Vcal_j$ as the set of all vertices of degree~$j$ in $\mathbf d'$. Recall the degree sequence $\mathbf{d}'$ from Lemma~\ref{lem:NewDegreeDistribution}, and we define 
	\begin{align*}
	L_k=\{ \{v_1,v_2,...,v_k\}: v_1, v_k \in \mathcal V_1; v_2,...,v_{k-1} \in \mathcal V_2 \},
	\end{align*}
	the set of all collections of $k$ vertices that could form a line. Note that
	\begin{align*}
	L_k' (n) = \sum_{l \in \mathcal L_k} \mathbbm 1_{\{l \text{ forms a line}\}}.
	\end{align*} 
	
	\noindent
	Due to Lemma~\ref{lem:NewDegreeDistribution}, we observe that
	\begin{align*}
	\E & [L_k'(n)] = \E\left[\sum_{l \in \mathcal L_k} \mathbbm 1_{\{l \text{ forms a line}\}} \right] \\
	&=\E\left[ \binom{n_1'}{ 2} \binom{n_2'}{k-2}\frac{2k-4}{2m-1}\frac{2k-6}{2m-3}  \cdots \frac{2}{2m-2k+5} 
	\frac{1}{2m-2k+3} \right]\\
	&=\E \left[ \frac{{n_1'}^2 {n_2'}^{k-2}}{2 (k-2)!}\frac{(2k-4)!!}{(2m)^{k-1}} \right](1+o(1)) = \frac{i^2}{m} \frac{1}{4}\Big(1+ \frac{2p_2}{d}\Big)^2\Big(\frac{2p_2}{d}\Big)^{k-2}(1+o(1)).
	\end{align*}
	
	Next, we generalize these arguments to higher and mixed moments of the variables  $(L_k'(n))_{n\geq 2}$. We follow the same approach used in~\cite{Federico2016} where the convergence of the number of lines to a sequence of Poisson variables was proved in the critical window for connectivity. Using the method of moments, in this case, we need to prove concentration of the number of vertices and edges in line components.
	
	We prove~\eqref{eq:MultivariateMoments} by induction. With a slight abuse of notation, we start the induction step at $k=1$, or alternatively, at $k=2$ with $r_2=0$. Then, both sides in \eqref{eq:MultivariateMoments} are equal to one, and hence the induction hypothesis is satisfied.
	
	Next, we show how to advance the induction hypothesis. We define 
	\begin{align*}
	W_k(\mathbf{r}) = \left\{ \bigcup_{j=2}^k \{ l_j(1),...,l_j(r_j) \} : l_j(h) \in \mathcal{L}_j \textrm{ for all } 1 \leq h \leq r_j, 2 \leq j \leq k \right\},
	\end{align*}
	the collection of sets of $\sum_{j=2}^k r_j$ possible lines. Moreover, for a set $w_k(\mathbf{r}) \in W_k(\mathbf{r})$, we define $\mathcal{E}(w_{k}(\mathbf r))$ as the event that all elements in the set $w_k(\mathbf{r})$ form a line component in $\Cmd$. Then, using the tower property, we can rewrite
	\begin{align*}
	&\E [L_2'(n)^{r_2}\cdots L_k'(n)^{r_h}] = \E\left[ \sum_{w_{k}(\mathbf r) \in W_k(\mathbf{r})}  \indi_{\mathcal{E}(w_k(\mathbf{r}))} \right]\\
	&= \E \left[ \sum_{w_{k-1}(\mathbf r) } \indi_{\mathcal{E}(w_{k-1}(\mathbf r))} \E\left[ \sum_{l_k(1),\ldots,l_k(r_k)\in \mathcal L_k} \indi_{l_k(1)} \indi_{l_k(2)} \cdots \indi_{l_k(r_k)}\mid \mathcal{E}(w_{k-1}(\mathbf r))\right] \right],
	\end{align*}
	where $\indi_{l_k(h)}$ denotes the indicator of the event that the set $l_k(h)$ forms a line. Next, we show that for every $w_k(\mathbf{r}) \in W_k(\mathbf{r})$, 
	\begin{align*}
	\Big(\frac{m}{i^2}\Big)^{r_k} \sum_{l_k(1),\ldots,l_k(r_k)\in \mathcal L_k} &	\E[\indi_{l_k(1)} \indi_{l_k(2)} \cdots \indi_{l_k(r_k)}\vert \mathcal{E}(w_{k-1}(\mathbf r))]\\
	&=\left( \frac{1}{4}\Big(1+ \frac{2p_2}{d}\Big)^2\Big(\frac{2p_2}{d}\Big)^{k-2}\right)^{r_k}(1+o(1)).
	\end{align*}
	Note that by induction, this suffices to conclude~\eqref{eq:MultivariateMoments}.
	
	First, note that if any of the $l_k(j), 1 \leq j \leq r_k,$ contains any of the vertices used in $w_{k-1}(\mathbf r)$, then it is not possible for it to form a line of length~$k$ as these vertices are already contained in line components of a smaller length. In that case, $\E[\indi_{l_k(1)} \indi_{l_k(2)} \cdots \indi_{l_k(r_k)}\mid \mathcal{E}(w_{k-1}(\mathbf r))]=0$. Therefore, we can consider the part of the graph excluding the line components in $w_{k-1}(\mathbf r)$. This part of the graph is also a configuration model, but with a slightly altered degree sequence. Note that we exclude finitely many line components of finite length, and hence only exclude finitely many vertices and edges. With a slight abuse of notation, write $r'= r'(l_k(1),...,l_k(r_k))$ as the number of mutually distinct lines. Then the probability that $r'$ mutually distinct lines can be formed in the graph excluding the line components in $w_{k-1}(\mathbf r)$ is given by
	\begin{align*}
	\E[\indi_{l_k(1)} \indi_{l_k(2)} \cdots \indi_{l_k(r_k)}\vert \mathcal{E}(w_{k-1}(\mathbf r))]= \frac{(2k-4)!!^{r'}}{(2m)^{r'(k-1)}}(1+o(1)).
	\end{align*}
	We sum these contributions over the number of possible sets that exist in the subgraph. Write $C(r_k,r')$ as the number of distinct sets of $r_k$ lines that contain the same set of $r'$ mutually distinct lines of length $k$. Importantly, $C(r_k,r_k)=1$ and $C(r_k,r')$ is a finite integer if $1 \leq r' \leq r_k-1$. Recalling Lemma~\ref{lem:NewDegreeDistribution}, we obtain
	\begin{align*}
	\Big(\frac{m}{i^2}\Big)^{r_k} &\sum_{l_k(1),\ldots,l_k(r_k)\in \mathcal L_k} \E[\indi_{l_k(1)} \indi_{l_k(2)} \cdots \indi_{l_k(r_k)}\vert \mathcal{E}(w_{k-1}(\mathbf r))] \\
	&=\Big(\frac{m}{i^2}\Big)^{r_k} \E\left[ \sum_{r'=1}^{r_k} C(r_k,r') \left( \frac{{n_1'}^2 {n_2'}^{k-2}}{2 (k-2)!}\frac{(2k-4)!!}{(2m)^{k-1}}  \right)^{r'} \right](1+o(1)) \\
	&=  \left( \frac{1}{4}\Big(1+ \frac{2p_2}{d}\Big)^2\Big(\frac{2p_2}{d}\Big)^{k-2}\right)^{r_k} (1+o(1)),
	\end{align*}
	concluding the proof.
\end{proof}

\begin{proof}[Proof of Corolllary~\ref{cor:HigherMoments}]
	Note that it follows from Lemma~\ref{lem:MultiVariateMoments}
	\begin{align*}
	\E [L_k'(n)] &= \frac{i^2}{m} \frac{1}{4}\Big(1+ \frac{2p_2}{d}\Big)^2\Big(\frac{2p_2}{d}\Big)^{k-2}(1+o(1)), \hspace{1cm} k\geq 2.
	\end{align*}
	Consequently, for every $\varepsilon > 0$ there exists a $N >0$ such that for all $n \geq N$,
	\begin{align*}
	\frac{m}{i^2}\E [L_k'(n)] \leq \frac{1}{4}\Big(1+ \frac{2p_2+\varepsilon}{d-\varepsilon}\Big)^2\Big(\frac{2p_2+\varepsilon}{d-\varepsilon}\Big)^{k-2}, \hspace{1cm} k\geq 2.
	\end{align*}
	In particular, for $\varepsilon$ small enough $\frac{2p_2+\varepsilon}{d-\varepsilon}<1$ so that the sequence converges to zero exponentially fast in $k$. We apply dominated convergence to obtain	
	\begin{align}
	\E\Big[\frac{m}{i^2}\sum_{k=2}^\infty kL_k'(n)\Big] &\to  \frac{(d - p_2) (d + 2 p_2)^2}{2 d (d - 2 p_2)^2}, \\
	\E\Big[\frac{m}{i^2}\sum_{k=2}^\infty (k-1)L_k'(n)\Big] &\to  \frac{(d + 2 p_2)^2}{4 (d - 2 p_2)^2}.
	\end{align}
	In other words, we have derived the expected number of vertices and edges in line components in $\Cmd$. Next, we prove~\eqref{eq:HigherMoments} for every $j\geq 2$. We define for a sequence $\mathbf{r}$ of positive integer values, $|\mathbf{r}|_1=\sum_{h=1}^k r_h$, i.e. its~$\ell_1$ norm. We write
	\begin{align*}
	\E\Big[\Big(\frac{m}{i^2}\sum_{k=2}^\infty kL_k'(n)\Big)^j\Big]=\frac{m^j}{i^{2j}}\sum_{\mathbf r:|\mathbf{r}|_1=j}\prod_{h\geq 2}\E[(hL_h'(n))^{r_h}].
	\end{align*}
	For every $\varepsilon \geq 0$, it holds for all $n$ sufficiently large that
	\begin{align}
	\label{eq:LineMultivariateDominated}
	\Big(\frac{m}{i^2}\Big)^{|\mathbf{r}|_1}\prod_{h\geq 2}\E[(L_h'(n))^{r_h}]\leq \frac{1}{4^{|\mathbf{r}|_1}}\Big(1+ \frac{2p_2+\varepsilon}{d-\varepsilon}\Big)^{2|\mathbf{r}|_1} \Big(\frac{2p_2+\varepsilon}{d-\varepsilon}\Big)^{\sum_{h\geq 2} (h-2)r_h}.
	\end{align}
	If $\varepsilon$ is small enough, then $\frac{2p_2+\varepsilon}{d-\varepsilon}<1$, and thus the sequence is decreasing exponentially fast in $\sum_{h\geq 2} hr_h$. Applying dominated convergence thus yields~\eqref{eq:HigherMoments}. 
\end{proof}

\section{Proofs of results on the cascading failure process}\label{app:ResultsCasc}

\begin{proof}[Proof of Lemma~\ref{lem:LBSimilarStoppingTimeBehavior}]
	First, we observe that
	\begin{align*}
	\Prob\left(T_{g^-} > k  \right) \geq \Prob\left(T_{1-\theta} > k  \right),
	\end{align*}
	and hence it suffices to show that the reversed inequality holds asymptotically. Our proof will be similar to the proof of Lemma~\ref{lem:UBSimilarStoppingTimeBehavior}, but adapted to provide an upper bound. 
	
	Fix $\epsilon >0$ small as in Lemma~\ref{lem:UBSimilarStoppingTimeBehavior}, and define the piecewise constant boundary
	\begin{align*}
	\hat{h}_{i,k}^\epsilon = \left\{ \begin{array}{ll}
	h^{(0)}=1-\theta & \textrm{if } i \leq t_{0,k}^\epsilon = l,\\
	-h^{(j)} & \textrm{if } t_{j-1,k}^\epsilon < i \leq t_{j,k}^\epsilon, \; 1 \leq j \leq r-1,\\
	-h^{(r)}  & \textrm{if } i > t_{r-1,k}^\epsilon.
	\end{array}\right.
	\end{align*}
	for $r$, times $t_{j,k}^\epsilon$ and levels $h^{(j)}$ defined as in the proof of Lemma~\ref{lem:UBSimilarStoppingTimeBehavior}. We note that for every fixed $\delta>0$,
	\begin{align*}
	&\Prob\left(T_{g^-} > k  \right) \leq \Prob\left(T_{\hat{h}^\epsilon} > k ; S_{t_{j,k}^\epsilon} \in \left(\delta\sqrt{t_{j,k}^\epsilon},1/\delta\sqrt{t_{j,k}^\epsilon}\right), \;\;\; \forall \, 0 \leq j \leq r-1 \right) \\
	&\hspace{5cm} + \sum_{j=0}^r \Prob\left(T_{g^-}  > k ; S_{t_{j,k}^\epsilon}  \not\in \left(\delta\sqrt{t_{j,k}^\epsilon},1/\delta\sqrt{t_{j,k}^\epsilon}\right) \right).
	\end{align*}
	Using analogous arguments as in Lemma~\ref{lem:UBSimilarStoppingTimeBehavior}, we obtain
	\begin{align*}
	&\Prob\left(T_{\hat{h}^\epsilon} > k ; S_{t_{j,k}^\epsilon} \in \left(\delta\sqrt{t_{j,k}^\epsilon},1/\delta\sqrt{t_{j,k}^\epsilon}\right), \;\;\; \forall \, 0 \leq j \leq r-1 \right) \\
	&\hspace{8cm}\leq (1+o(1)) \Prob\left(T_{1-\theta} > k \right).
	\end{align*}
	For the other terms, define the sequence $(\tilde{h}_i)_{i \in \mathbb{N}}$ with $\tilde{h}_i = \min\{1-\theta,-i^\gamma\}$. Then, due to Theorem~1 of~\cite{DenisovSakhanenkoWachtel2016},
	\begin{align*}
	\sum_{j=0}^r &\Prob\left(T_{g^-}  > k ; S_{t_{j,k}^\epsilon}  \not\in \left(\delta\sqrt{t_{j,k}^\epsilon},1/\delta\sqrt{t_{j,k}^\epsilon}\right) \right) \\
	&\hspace{2cm}\leq  \sum_{j=0}^{r-1} \Prob\left( S_{t_{j,k}^\epsilon}  \not\in \left(\delta\sqrt{t_{j,k}^\epsilon},1/\delta\sqrt{t_{j,k}^\epsilon}\right) \big| T_{\tilde{h}}  > k \right) \Prob\left( T_{\tilde{h}}  > k \right) \\
	&\hspace{2cm} \leq (1+o(1)) r \left(1-e^{-\frac{\delta^2}{2}} + e^{-\frac{1}{2\delta^2}} \right) \Prob\left( T_{\tilde{h}}  > k \right).
	\end{align*}
	Letting $\delta \downarrow 0$ yields~\cite{DenisovSakhanenkoWachtel2016}
	\begin{align*}
	\sum_{j=0}^r &\Prob\left(T_{g^-}  > k ; S_{t_{j,k}^\epsilon}  \not\in \left(\delta\sqrt{t_{j,k}^\epsilon},1/\delta\sqrt{t_{j,k}^\epsilon}\right) \right) = o\left(  \Prob\left( T_{\tilde{h}}  > k \right) \right) = o\left(  k^{-1/2}\right).
	\end{align*}
	Since 
	\begin{align*}
	\Prob\left(T_{1-\theta} > k \right) \sim \frac{2\theta}{\sqrt{2\pi}} k^{-1/2},
	\end{align*}
	we conclude that
	\begin{align*}
	\limsup_{k \rightarrow \infty} \frac{\Prob\left(T_{g^-} > k  \right) }{\Prob\left(T_{1-\theta} > k  \right) } \leq 1.
	\end{align*}
\end{proof}

\begin{proof}[Proof of Lemma~\ref{lem:upsilonBound}]
	This statement is a consequence of Theorem~\ref{thm:DeviationsOutsideGiant}. Recall that by definition,
	\begin{align*}
	\upsilon(k)+|\tilde{E}_m(\upsilon(k))| \geq k,
	\end{align*}
	and $\upsilon(k) \leq k$. Moreover,
	\begin{align*}
	|\tilde{E}_m(\upsilon(k))| = m -\upsilon(k) - |\hat{E}_m(\upsilon(k))| \leq \max_{1\leq j \leq k} \left\{ m -j - |\hat{E}_m(j)| \right\}.
	\end{align*}
	It follows that
	\begin{align*}
	\Prob\left( |\tilde{E}_m(\upsilon(k))| \geq k^\alpha \right) &\leq \Prob\left( \max_{1\leq j \leq k} \left\{ m -j - |\hat{E}_m(j)| \right\} \geq k^\alpha \right) \\
	& \leq \Prob\left( \max_{1\leq j \leq k} \left\{ m -j - |\hat{E}_m(j)| \right\} \geq j^\alpha \right) = o(m^{-1/2})
	\end{align*}
	by Theorem~\ref{thm:DeviationsOutsideGiant}. We conclude that
	\begin{align*}
	\Prob\left( \upsilon(k) \leq k- k^\alpha \right) \leq \Prob\left( |\tilde{E}_m(\upsilon(k))| \geq k^\alpha \right) = o(m^{-1/2}).
	\end{align*}
\end{proof}

\begin{proof}[Proof of Lemma~\ref{lem:varrhoBound}]
	Again, this is a consequence of Theorem~\ref{thm:DeviationsOutsideGiant}. We note that for every $1 \leq l \leq m$,
	\begin{align*}
	\kappa(l) = \sum_{i=1}^s \textrm{Ber}(\pi_{i+1}),
	\end{align*}
	where 
	\begin{align*}
	\pi_i =  \frac{|\hat{E}_m(i-2)|}{m-i+2}
	\end{align*}
	is a random variable. Due to Theorem~\ref{thm:DeviationsOutsideGiant}, 
	\begin{align*}
	\Prob\left(\pi_i \leq 1 - \frac{i^\alpha}{m-i+2} \textrm{ for some } 2 \leq i \leq k+1 \right) = o(m^{-1/2}).
	\end{align*}
	Since for every $\alpha \in (0,1)$, the function $i^{\alpha-1}/(m-i+1)$ is an increasing function in $i$, it follows that 
	\begin{align*}
	\Prob\left(\pi_i \leq 1 - k^{\alpha-1} \textrm{ for some } 2 \leq i \leq k+1 \right) = o(m^{-1/2}).
	\end{align*}
	We derive the bound
	\begin{align*}
	\Prob\left( \varrho(k) > k+ k^\alpha \right) &= \Prob\left( \kappa\left(k+ k^{(\alpha+1)/2}\right) \leq k \right) \leq \Prob\left( \sum_{i=1}^{k+ k^{(\alpha+1)/2}} \textrm{Ber}(\pi_{i+1}) \leq k \right) \\
	& \leq \Prob\left( \sum_{i=1}^{k+ k^{(\alpha+1)/2}} \textrm{Ber}(k^{\alpha-1}) > k^{(\alpha+1)/2} \right) \\
	&\leq \exp\left\{-\frac{1}{2}k^{(1-\alpha)/2} (1+o(1)) \right\} = o(m^{-1/2}),
	\end{align*}
	where the last inequality is due to the Chernoff bound. 
\end{proof}

\end{document}